\documentclass{amsart}

\usepackage{amscd}
\usepackage{amsmath}
\usepackage{amssymb}
\usepackage{amsthm}
\usepackage{arydshln}
\usepackage{bbm}
\usepackage{bm}
\usepackage{colonequals}
\usepackage{color}
\usepackage{enumitem}
\usepackage{latexsym}
\usepackage{mathrsfs}
\usepackage{mathtools}
\usepackage{stmaryrd}
\usepackage{tikz-cd}
\usepackage{varwidth}
\usepackage[all]{xy}
\usepackage{yhmath}
\usepackage{enumitem}

\usepackage{time}

\usepackage[dvipdfmx, bookmarkstype=toc, colorlinks=false, pdfborder={0 0 1}, bookmarks=true, bookmarksnumbered=true]{hyperref}

\newcommand{\temp}{\mathrm{temp}}

\newcommand{\ur}{\mathrm{ur}}

\newcommand{\BT}{\mathrm{BT}}

\newcommand{\lan}{\langle}
\newcommand{\ran}{\rangle}
\newcommand{\ol}{\overline}

\renewcommand{\L}{{}^{L}}

\renewcommand{\-}{\text{-}}

\newcommand{\ringofintegers}{\mathcal{O}}
\newcommand{\padicfield}{k}
\newcommand{\Langlandsparameter}{\phi_{\pi}}
\newcommand{\Langlandsparameternorep}{\phi}

\newcommand{\bbA}{\mathbb{A}}
\newcommand{\C}{\mathbb{C}}
\newcommand{\F}{\mathbb{F}}
\newcommand{\Q}{\mathbb{Q}}
\newcommand{\R}{\mathbb{R}}
\newcommand{\Z}{\mathbb{Z}}

\newcommand{\bfo}{\mathbf{o}}

\newcommand{\bfB}{\mathbf{B}}

\newcommand{\bfG}{\mathbf{G}}
\newcommand{\bfH}{\mathbf{H}}

\newcommand{\bfT}{\mathbf{T}}
\newcommand{\bfU}{\mathbf{U}}

\newcommand{\x}{\mathbf{x}}

\newcommand{\G}{\mathbf{G}}

\newcommand{\mff}{\mathfrak{f}}

\newcommand{\mfp}{\mathfrak{p}}

\newcommand{\mfw}{\mathfrak{w}}

\newcommand{\mcA}{\mathcal{A}}

\newcommand{\mcE}{\mathcal{E}}

\newcommand{\mcH}{\mathcal{H}}

\newcommand{\mcO}{\mathcal{O}}

\newcommand{\biota}{\boldsymbol{\iota}}

\DeclareMathOperator{\cInd}{c-Ind}

\DeclareMathOperator{\diag}{diag}

\DeclareMathOperator{\id}{id}

\DeclareMathOperator{\nInd}{n-Ind}

\DeclareMathOperator{\val}{val}
\DeclareMathOperator{\vol}{vol}

\DeclareMathOperator{\Ad}{Ad}
\DeclareMathOperator{\Artin}{Artin}

\DeclareMathOperator{\Cent}{Cent}

\DeclareMathOperator{\End}{End}
\DeclareMathOperator{\Frob}{Frob}
\DeclareMathOperator{\Gal}{Gal}
\DeclareMathOperator{\Hom}{Hom}

\DeclareMathOperator{\Ind}{Ind}

\DeclareMathOperator{\Lie}{Lie}
\DeclareMathOperator{\LLC}{LLC}

\DeclareMathOperator{\Res}{Res}

\DeclareMathOperator{\Swan}{Swan}
\DeclareMathOperator{\Sym}{Sym}
\DeclareMathOperator{\Tr}{Tr}

\DeclareMathOperator{\GL}{GL}

\DeclareMathOperator{\SL}{SL}
\DeclareMathOperator{\SO}{SO}
\DeclareMathOperator{\Sp}{Sp}

\DeclareMathOperator{\SU}{SU}

\DeclareMathOperator{\Param}{\delta}

\renewcommand{\Re}{{\mathrm{Re}}}

\theoremstyle{plain}
\newtheorem{thm}{Theorem}[section]
\newtheorem*{thm*}{Theorem}
\newtheorem{prop}[thm]{Proposition}
\newtheorem{lem}[thm]{Lemma}
\newtheorem{cor}[thm]{Corollary}
\newtheorem{conj}[thm]{Conjecture}

\theoremstyle{definition}

\theoremstyle{remark}
\newtheorem{rem}[thm]{Remark}

\newtheorem*{claim*}{Claim}

\SelectTips{cm}{11}

\title{Simple supercuspidal $L$-packets of split special orthogonal groups over dyadic fields}

\author{Moshe Adrian}
\address{Department of Mathematics, Queens College, CUNY, Queens, NY 11367-1597}
\email{moshe.adrian@qc.cuny.edu}

\author{Guy Henniart}
\address{Universit\'e Paris-Saclay, CNRS, Laboratoire de Math\'ematiques d'Orsay, 91405, Orsay, France.}
\email{Guy.Henniart@math.u-psud.fr}

\author{Eyal Kaplan}
\address{Department of Mathematics, Bar Ilan University, Ramat Gan 5290002, Israel}
\email{kaplaney@gmail.com}

\author{Masao Oi}
\address{Department of Mathematics, National Taiwan University, Astronomy Mathematics Building 5F, No.\ 1, Sec.\ 4, Roosevelt Rd., Taipei 10617, Taiwan}
\email{masaooi@ntu.edu.tw}

\date{\now, \today}

\begin{document}

\begin{abstract}
We consider the split special orthogonal group $\SO_{N}$ defined over a $p$-adic field.
We determine the structure of any $L$-packet of $\SO_{N}$ containing a simple supercuspidal representation (in the sense of Gross--Reeder).
We also determine its endoscopic lift to a general linear group.
Combined with the explicit local Langlands correspondence for simple supercuspidal representations of general linear groups, this leads us to get an explicit description of the $L$-parameter as a representation of the Weil group of $F$.
Our result is new when $p=2$ and our method provides a new proof even when $p\neq2$.
\end{abstract}

\subjclass[2010]{Primary: 11S37, 22E50; Secondary: 11F70, 11F80, 11F85}

\maketitle

\tableofcontents

\section{Introduction}
The aim of this paper is to give a complete description of the local Langlands correspondence for simple supercuspidal representations of split special orthogonal groups.

Let us first briefly recall the local Langlands correspondence.
Suppose that $\G$ is a connected reductive group over a $p$-adic field $F$.
Let $\Pi(\G)$ denote the set of equivalence classes of irreducible admissible representations of $\G(F)$ and $\Phi(\G)$ denote the set of $\hat{\G}$-conjugacy classes of $L$-parameters of $\G$.
Here, recall that an $L$-parameter of $\G$ is a homomorphism $W_{F}\times\SL_{2}(\C)\rightarrow\hat{\G}\rtimes W_{F}$ satisfying certain conditions, where $W_{F}$ is the Weil group of $F$ and $\hat{\G}$ is the Langlands dual group of $\G$ over $\C$.
\textit{The local Langlands correspondence for $\G$}, which is still conjectural in general, asserts that there exists a natural finite-to-one map
\[
\LLC_{\G}\colon\Pi(\G)\rightarrow\Phi(\G).
\]
In other words, it is expected that the set $\Pi(\G)$ can be partitioned into the disjoint union of finite sets $\Pi^{\G}_{\phi}:=\LLC_{\G}^{-1}(\phi)$ (called \textit{$L$-packets}) labelled by $L$-parameters $\phi\in\Phi(\G)$:
\[
\Pi(\G)
=
\bigsqcup_{\phi\in\Phi(\G)}\Pi_{\phi}^{\G}.
\]

The local Langlands correspondence has been established for several specific groups.
Especially, when $\G$ is $\GL_{N}$, the correspondence was constructed by Harris--Taylor \cite{HT01} and the second named author \cite{Hen00}.
Also, for a certain class of classical groups, the correspondence was constructed by Arthur \cite{Art13} (quasi-split symplectic and orthogonal groups) and Mok \cite{Mok15} (quasi-split unitary groups).

When $\G$ is one of these groups (assume that $\G$ is split for simplicity), we can naturally regard $\L{\G}$ as a subgroup of $\GL_{N}(\C)\times W_{F}$ for an appropriate positive integer $N$.
This means that we may think of an $L$-parameter of $\G$ as a representation of $W_{F}\times\SL_{2}(\C)$ equipped with additional structure.
For example, when $\G$ is the split odd special orthogonal group $\SO_{2n+1}$, $\hat{\G}$ is given by $\Sp_{2n}(\C)$ with trivial Galois action.
Thus an $L$-parameter of $\SO_{2n+1}$ is regarded as a $2n$-dimensional symplectic representation of $W_{F}\times\SL_{2}(\C)$.
Given this observation, it is natural to ask the following question:
\begin{quote}
Describe $\LLC_{\G}$ explicitly.
More precisely, for a given $\pi\in\Pi(\G)$,
\begin{enumerate}
\item
describe its $L$-parameter $\phi$ as a representation of $W_{F}\times\SL_{2}(\C)$;
\item
determine the $L$-packet $\Pi^{\G}_{\phi}$ containing $\pi$.
\end{enumerate}
\end{quote}

Here we concentrate on the case of \textit{simple supercuspidal representations} in the sense of Gross--Reeder (see Section \ref{sec:SSC}).
We briefly discuss previous works.

For $\G=\GL_{N}$, $L$-packets are singletons, and the parameter of an irreducible supercuspidal representation $\pi$ of $\GL_{N}(F)$ is an $N$-dimensional irreducible representation $\phi$ of $W_{F}$, taken up to equivalence.
When $\pi$ is a simple supercuspidal representation, Bushnell--Henniart (\cite{BH14}), using the theory of cyclic base change and that of tame base change, explicitly described the projective representation determined by $\phi$.
They also established a characterization of $\pi$ via the $\varepsilon$-factors $\varepsilon(s,\pi\times\chi,\psi)$, for tame characters $\chi$ of $F^{\times}$.
That translates into an analogous characterization of $\phi$, and it suffices to produce such an explicit $\phi$ such that $\varepsilon(s,\phi\otimes\chi,\psi)=\varepsilon(s,\pi\times\chi,\psi)$.
That was done by Adrian--Liu when $N$ is prime to $p$ (\cite{AL16}), thus giving a more direct proof of the essentially tame local Langlands correspondence \cite{BH05-1, BH05-2, BH10} in the case of a simple supercuspidal representation.
That was also done for general $N$ by Imai and Tsushima (\cite{IT23}), who determined $\phi$ (for a simple supercuspidal $\pi$) explicitly using geometry.

When $\G$ is one of our split classical groups, the parameter of a simple supercuspidal representation $\pi$ of $\G(F)$ can be seen as a direct sum $\phi$ of irreducible representations $\phi_{i}$ of $W_{F}$.
The known methods to explicitly describe the parameter rather determine the supercuspidal representation $\pi_i$ of $\GL_{d_{i}}(F)$ (where $d_{i}=\dim(\phi_{i}))$ with parameter $\phi_{i}$ (the \textit{endoscopic lift} to a general linear group), and one can then apply the results for $\GL_{d_{i}}(F)$ to get $\phi_{i}$.
There are at least three such methods:
\begin{enumerate}
\item\label{method:1}
One can use the endoscopic and twisted endoscopic character relations by which Arthur's results are expressed.
\item\label{method:2}
One can use M{\oe}glin's criterion using reducibility points of representations parabolically induced from $\pi\boxtimes\tau$ where $\tau$ is a (supercuspidal) representation of some $\GL_{r}(F)$.
\item\label{method:3}
One can use $\gamma$-factors for the pairs $(\pi,\tau)$.
Indeed the Rankin--Selberg (or Langlands--Shahidi) factor $\gamma(s,\pi\times\tau,\psi)$ should be the product of the factors $\gamma(s,\pi_{i}\times\tau,\psi)$, and by the above-mentioned characterization due to \cite{BH14} it is enough to consider the case where $\tau$ is a tame character of $F^{\times}=\GL_{1}(F)$.
\end{enumerate}

The three methods can be used separately or jointly.
For \eqref{method:2} one has to know that the criterion does indeed determine Arthur's parameter (that is due to Bin Xu \cite{Xu17-MM}, see below Remark \ref{rem:Moeglin-Xu}).
For \eqref{method:3} one has to know that Arthur's lifting from $\G$ to $\GL_{N}$ preserves the $\gamma$-factors.
That we establish in Appendices \ref{sec:lift} and \ref{sec:unram}, for a generic supercuspidal representation $\pi$ (see the end of this introduction for more detailed comments).

Method \eqref{method:2} has been used by Blondel--Henniart--Stevens \cite{BHS18} for $\G=\Sp_{2n}$, when $p$ is odd.
The method applies to a general supercuspidal representation, but determines the parameter only up to twist by an unramified quadratic character (see \cite[Section 5]{BHS18})\footnote{For a simple supercuspidal $\pi$, the second named author recently realized that the method of \cite{BHS18} actually allows to determine the full parameter, see \cite[Section 5.2]{BHS23}.
}.  In this work, a character is quadratic if and only if its square is trivial.

The first definitive result was obtained in \cite{Adr16} for $\G=\SO_{2n+1}$ by using Method \eqref{method:3}.
In \cite{Adr16}, the factors $\gamma(s,\pi\times\tau,\psi)$ is computed for tame characters $\tau$ (and arbitrary $p$).
Moreover, it is also proved that if the $L$-parameter is irreducible, it corresponds to a simple supercuspidal representation of $\GL_{2n}(F)$ which is then explicit, and so is the parameter of $\pi$.
Then, using results of Kaletha \cite{Kal13-wild, Kal15} (pertaining to Method \eqref{method:1}), one can determine the parameter of $\pi$ when $p$ is large.

When $\G$ is $\Sp_{2n}$ and $p$ is odd, the factor $\gamma(s,\pi\times\tau,\psi)$ is computed for tame characters $\tau$ in \cite{AK19}.
Based on this computation, the parameter is explicitly determined when $p$ does not divide $n$, and, adding the use of (a conjectured extension of) \cite{BHS18} (Method \eqref{method:2}), for any odd $p$.
When $p$ is $2$, $\gamma(s,\pi\times\tau,\psi)$ is computed in \cite{AK19} only when $F=\Q_2$, in which case there is a unique simple supercuspidal representation of $\Sp_{2n}(\Q_2)$.
There is no tame character in the parameter when $F=\Q_{2}$, and the parameter is explicitly described in \cite{AK19} provided it is irreducible.
Adding Method \eqref{method:1} (and inspired by the work \cite{Oi24} for odd $p$), the second named author subsequently proved that for $F=\Q_2$, the parameter is indeed irreducible, hence explicit (\cite{Hen23}, written in 2021).

Finally when $\G=\SO_{2n}$, for a simple supercuspidal representation $\pi$, the factor $\gamma(s,\pi\times\tau,\psi)$ for tame quadratic characters $\tau$ (and any $p$) is computed and the parameter of $\pi$ is predicted in \cite{AK21}. Here we supplement the result of \cite{AK21} by computing $\gamma(s,\pi\times\tau,\psi)$ for any tame $\tau$, not necessarily  quadratic. See the discussion after Theorem \ref{thm:main-intro} below, and Section \ref{sec:AK}.

Method \eqref{method:1} was used by the fourth named author in his thesis and subsequent works (\cite{Oi19-AJM,Oi19-PRIMS,Oi24}), to get a complete determination of the parameter of a simple supercuspidal representation $\pi$ of any of the following classical groups, provided $p$ is odd:
\begin{itemize}
\item
$\SO_{2n+1}$ (\cite{Oi19-AJM}),
\item
unramified quasi-split unitary groups (\cite{Oi19-PRIMS}), and
\item
$\Sp_{2n}$ and $\SO_{2n}$ (not necessarily split but quasi-split) (\cite{Oi24}).
\end{itemize}
(In Sections \ref{subsec:parametrization} and \ref{subsec:comparison}, we compare these results with those of \cite{Adr16,AK21} in the case of split special orthogonal groups).
Any odd $p$ can be covered uniformly by this approach, but the dyadic case (i.e., $p=2$) still remains.
Based on this motivation, the second and fourth named authors investigated the case of $\Sp_{2n}$ over a dyadic field in \cite{HO22} by elaborating on the method of \cite{Oi19-AJM,Oi19-PRIMS,Oi24}.
(As mentioned above, the second named author had already obtained the result for $F=\Q_{2}$ in \cite{Hen23} prior to \cite{HO22}.)

Now let us state the main result of this paper:
\begin{thm}[Theorems \ref{thm:main-1}, \ref{thm:main-2-a}, \ref{thm:main-2-b}]\label{thm:main-intro}
Let $\SO_{N}$ be the split special orthogonal group of degree $N$.
Let $\pi^{\SO_{N}}$ be a simple supercuspidal representation of $\SO_{N}(F)$ with $L$-parameter $\phi$.
Then, $\phi$ is trivial on $\SL_{2}(\C)$ and described as a representation of $W_{F}$ as follows:
\begin{enumerate}
\item
When $N=2n+1$, $\phi$ is a $2n$-dimensional irreducible symplectic representation of $W_{F}$.
Moreover, $\phi$ corresponds to a simple supercuspidal representation $\pi$ of $\GL_{2n}(F)$ under the local Langlands correspondence for $\GL_{2n}$.
\item
When $N=2n$ and $p=2$, $\phi$ is of the form $\phi=\phi_{0}\oplus\phi_{1}$, where
\begin{itemize}
\item
$\phi_{0}$ is a $(2n-1)$-dimensional irreducible orthogonal representation of $W_{F}$, which is the $L$-parameter of a simple supercuspidal representation $\pi$ of $\GL_{2n-1}(F)$, and
\item
$\phi_{1}$ is the determinant character of $\phi_{0}$.
\end{itemize}
\item
When $N=2n$ and $p\neq2$, $\phi$ is of the form $\phi=\phi_{0}\oplus\phi_{1}\oplus\phi_{2}$, where
\begin{itemize}
\item
$\phi_{0}$ is a $(2n-2)$-dimensional irreducible orthogonal representation of $W_{F}$, which is the $L$-parameter of a simple supercuspidal representation $\pi$ of $\GL_{2n-2}(F)$,
\item
$\phi_{1}$ is an unramified quadratic character of $F^{\times}$, and
\item
$\phi_{2}$ is a ramified quadratic character of $F^{\times}$.
\end{itemize}
\end{enumerate}
Furthermore, in each case, the quadratic characters and $\pi$ can be determined exactly from $\pi^{\SO_{N}}$ in terms of explicit parametrizing sets of simple supercuspidal representations (see Theorems \ref{thm:main-1},  \ref{thm:main-2-a}, \ref{thm:main-2-b} and Section \ref{sec:SSC}).
\end{thm}


As explained above, this result is new only when $p=2$; the case where $N=2n+1$ and $p\neq2$ is treated in \cite{Adr16} and \cite{Oi19-AJM}, and the case where $N=2n$ and $p\neq2$ is treated in \cite{AK19} and \cite{Oi24}.
However, we emphasize that our method provides an alternative proof even in these cases.  When $p$ is odd, the present method uses much less of the theory of endoscopy.
In particular, we do not need, for example, a complicated computation of twisted characters as performed in \cite{Oi19-AJM,Oi24}.

Let us explain the outline of the proof of Theorem \ref{thm:main-intro}.
We focus only on the case (2) in the following since the other cases can be treated in a similar manner.
Suppose that $\pi^{\SO_{N}}$ is a simple supercuspidal representation of $\SO_{N}(F)$ with $L$-parameter $\phi$.
Then $\phi$ is regarded as an orthogonal representation of $W_{F}\times\SL_{2}(\C)$.

The first step is the same as in \cite{Oi24}.
We can easily check that any simple supercuspidal representation of $\SO_{N}(F)$ is generic (with respect to a fixed Whittaker datum).
By combining this fact with a result of M{\oe}glin (\cite{Moeg11}) and Xu (\cite{Xu17-MM}), we see that $\phi$ is trivial on $\SL_{2}(\C)$ and decomposes into the direct sum of pairwise inequivalent orthogonal representations of $W_{F}$.
Let us write $\phi=\phi_{0}\oplus\cdots\oplus\phi_{r}$.

The second step is crucial; we prove that $\phi$ is of the form $\phi_{0}\oplus\phi_{1}$, where $\phi_{0}$ has Swan exponent one and $\phi_{1}$ is a tamely ramified quadratic character.
For this, we utilize the information of two different kinds of $\gamma$-factors attached to $\phi$; (1) the (standard) $\gamma$-factor $\gamma(s,\pi^{\SO_{N}}\times\chi,\psi)$ twisted by a tamely ramified character $\chi$ of $F^{\times}$, and (2) the special value of the adjoint $\gamma$-factor $\gamma(0,\Ad,\pi^{\SO_{N}},\psi_{0})$ where $\psi_{0}$ is an unramified character.
\begin{enumerate}
\item
In \cite{AK21}, the $\gamma$-factor twisted by $\chi$ is completely determined for any tamely ramified quadratic character $\chi$ of $F^{\times}$.
From this result, we immediately see that the Swan exponent of $\phi$ is given by $1$ and that exactly one of $\phi_{0},\ldots,\phi_{r}$ is a tamely ramified quadratic character.
\item
In general, it is conjectured by Hiraga--Ichino--Ikeda (\cite{HII08}) that the formal degree of a discrete series representation of a $p$-adic reductive group is related to the special value (at $0$) of the adjoint $\gamma$-factor of the $L$-parameter corresponding to the representation under the local Langlands correspondence.
Recently, it was announced by Beuzart-Plessis (\cite{BP21-OWR}) that this conjecture is proved for even special orthogonal groups; our results are conditional on \textit{ibid.}
By specializing the Hiraga--Ichino--Ikeda conjecture to the case of simple supercuspidal representations, we can compute the special value of the adjoint $\gamma$-factor of $\phi$.
\end{enumerate}
The argument of this step is as follows.
By (2), we know that the special value of the adjoint $\gamma$-factor of $\phi$, which is nothing but the exterior square $\gamma$-factor of $\phi$ in the current situation, is given by a rational power of $2$.
By (1), any irreducible constituent $\phi_{i}$ has Swan exponent $0$ and dimension greater than $1$ if it is neither the unique irreducible constituent of $\phi$ with Swan exponent $1$ nor the unique tamely ramified quadratic character contained in $\phi$.
However, if such a constituent $\phi_{i}$ existed, then it would contribute to the special value of the exterior square $\gamma$-factor of $\phi$ by an odd prime factor, which is not allowed.
Thus $\phi$ must be of the form $\phi_{0}\oplus\phi_{1}$ with $\phi_{0}$ and $\phi_{1}$ as above.

The final step is to determine $\phi$ exactly as a representation of $W_{F}$, but this can be done by the same argument as in \cite{Adr16} using the $\gamma$-factors, which is explained above. The idea is to determine $\phi_0$ using a converse theorem which includes the family of $\gamma$-factors $\gamma(s,\phi_0\times\chi,\psi)$, where $\chi$ is a tame character of $F^{\times}$. A priori, this family is too small and one must consider twists by arbitrary characters as well as by representations of $\GL_k$, $k>1$.
But at this point $\phi_0$ will correspond to a simple supercuspidal representation, and it turns out that for such representations tame twists are enough, see Section \ref{Main Theorem}. In turn we have to look at the $\gamma$-factor of $\pi^{\SO_{N}}$ twisted by any tamely ramified character of $F^{\times}$. We remark that in \cite{AK21}, we only considered twists by quadratic characters (and in particular, by the trivial character). The reason is twofold: First, if $\gamma(s, \pi^{\SO_{N}}\times\chi,\psi)$ is not holomorphic at $s=1$, then $\chi$ is quadratic ($\chi^2 = 1)$, which is why quadratic twists were effective in finding the one-dimensional summands of $\phi$, and second, the computation for $\chi=\mathbbm{1}$ together with the central character condition on the local Langlands correspondence for general linear groups, specified $\phi_0$ up to its restriction to the wild inertia subgroup. The computation for all tame characters $\chi$, which is more subtle, is given in this paper, see
Corollary~\ref{corollary:gamma for tamely ramified} and Theorem~\ref{thm:AK}.

We would like to mention other works that have dealt with the local Langlands correspondence for simple supercuspidal (more generally, epipelagic supercuspidal) representations of general reductive groups.
In \cite{Kal13-wild,Kal15}, Kaletha has constructed a correspondence for epipelagic supercuspidal representations when $p$ does not divide the order of the Weyl group.
Moreover, by assuming a further condition on $p$, he also proved the stability property for simple supercuspidal representations.
For large $p$, this tells us the size of any $L$-packet consisting of simple supercuspidal representations.
We also mention the work of Tam \cite{Tam23}, which is valid for all $p>2$, and treats epipelagic representations by a method similar to \cite{BHS18}.

In April 2023, Beth Romano sent us her preprint (now published, \cite{Rom23}) where she
classifies the epipelagic supercuspidal representations by their stable vectors and additive characters (see \cite{RY14}).
Subject to two hypotheses, one of which is the formal degree conjecture, she obtains some properties of the $L$-parameter $\phi$ of a simple supercuspidal representation of a split adjoint simple group: $\phi$ is trivial on $\SL_{2}(\C)$, has trivial $L$-factor and known Swan exponent.
The intersection of her work with ours is the case of $\SO_{2n+1}$.

Finally, we would like to briefly remark on some pending results that our work still depends on.  First, Arthur's result depends on future references in his bibliography, but most of those points have now been addressed in a preprint of Atobe--Gan--Ichino--Kaletha--Minguez--Shin \cite{AGIKMS24}; only the weighted fundemental lemma is pending. Similarly, the proof of the results announced by Beuzart-Plessis (\cite{BP21-OWR}) has not been posted yet.

Before we finish this introduction, we give one more comment on another subtlety concerning the $\gamma$-factor.
In the above arguments, we freely use the fact that Arthur's local Langlands correspondence preserves the twisted $\gamma$-factors (or more generally, the Rankin--Selberg local factors for the product of a classical group and a general linear group).
Although we believe that this fact should be well-known to experts, we give a justification in this paper (Appendix \ref{sec:lift}) because we could not find a suitable reference.
Since our argument is based on a global method, it is crucially important that the unramified case of Arthur's construction of local $A$-packets is consistent with the classical one by means of the Satake isomorphism.
However, somehow we were not able to find even this fact in any literature.
Hence we decided to also give a proof of this consistency in the unramified case (Appendix \ref{sec:unram}).
The proof we present in Appendix \ref{sec:unram} is due to Jean-Loup Waldspurger.
We would like to clarify that Appendix \ref{sec:unram} is constructed based on a letter from him (but, of course, we completely owe the responsibility for it).

\medbreak
\noindent{\bfseries Acknowledgments.}\quad
We would like to express our sincere gratitude to Jean-Loup Waldspurger for his detailed explanation of the content of Appendix \ref{sec:unram}.
We would also like to thank Beth Romano for sending us her preprint and for her comments on the first posted version of this paper.
The final draft of this paper was written while Henniart enjoyed the hospitality
of the Graduate School of Mathematical Sciences at the University of Tokyo.
Finally, we thank the anonymous referees for their thorough reading of this paper and for their suggestions.
One of the referees asked us about the image of wild and tame inertia in the $L$-parameter $\phi$ of Theorem \ref{thm:main-intro}. We think that \cite{IT23} can be used to gain some information, but leave this to future work.
Adrian was supported by a grant from the Simons Foundation \#422638 and by a PSC-CUNY award, jointly funded by the Professional Staff Congress and The City University of New York.
Kaplan was supported by the ISRAEL SCIENCE FOUNDATION (grant No.\ 376/21).
Oi was supported by JSPS KAKENHI Grant Number 20K14287.

\medbreak
\noindent
\textbf{Notation.}\quad
Let $F$ be a $p$-adic field (i.e., a finite extension of $\Q_{p}$; especially, of characteristic $0$), $\mcO$ its ring of integers, $\mfp$ its maximal ideal, and $k$ its residue field $\mcO/\mfp$.
We write $q$ for the cardinality of $k$.
We often regard $k^{\times}$ as the subgroup of $F^{\times}$ consisting of elements of finite prime-to-$p$ order via the Teichm\"uller lift.
We fix a uniformizer $\varpi$ of $F$.
For any element $x\in\mcO$, we write $\overline{x}$ for its image in $k$.

We fix a nontrivial additive character $\psi_{\F_{p}}\colon\F_{p}\rightarrow\C^{\times}$ and let $\psi\colon k\rightarrow\C^{\times}$ be the nontrivial additive character defined by $\psi=\psi_{\F_{p}}\circ\Tr_{k/\F_{p}}$.
Note that $\psi$ is invariant under the Frobenius, i.e., $\psi(x^{p})=\psi(x)$ for any $x\in k$.
We fix a nontrivial additive character of $F$ whose restriction to $\mcO$ lifts $\psi$ and again write $\psi$ for it by abuse of notation.

When $p\neq2$, we fix an element $\epsilon\in k^{\times}\smallsetminus k^{\times2}$.
When $p=2$, we simply put $\epsilon:=1$.

We let $I_{N}$ denote the identity matrix of size $N$ and $J_{N}$ denote the anti-diagonal matrix of size $N$ whose $(i, N+1-i)$-th entry is given by $(-1)^{i-1}$:
\[
J_{N} = \begin{pmatrix}
 &&&1\\
 &&-1&\\
 &\adots&&\\
 (-1)^{N-1}&&&
\end{pmatrix}.
\]

For an algebraic group $\bfG$ defined over $F$, we let $G$ denote the group $\bfG(F)$ of $F$-points of $\bfG$.

\section{Simple supercuspidal representations}\label{sec:SSC}

In this section, we summarize a classification of simple supercuspidal representations of $\SO_{N}(F)$ and also those of $\GL_{N}(F)$ which are self-dual.
See \cite[Sections 2.1 and 2.2]{Oi24} for a general recipe and the definition of simple supercuspidal representations.
The classification given here is basically the same as the one of \cite{Oi19-AJM, Oi24, HO22}, but requires a minor modification since we treat any $p$.

\subsection{Self-dual simple supercuspidal representations of \texorpdfstring{$\GL_{N}$}{GL(N)}}\label{subsec:SSC-GL}

Let us consider the case of $\GL_{N}$.
Let $I_{\GL_{N}}$ be the upper-triangular Iwahori subgroup of $\GL_{N}$:
\[
I_{\GL_{N}}
=
\begin{pmatrix}
 \mcO^{\times}&&\mcO\\
 &\ddots&\\
 \mfp&&\mcO^{\times}
\end{pmatrix}.
\]
We let $I^{+}_{\GL_{N}}$ be the pro-unipotent radical of $I_{\GL_{N}}$.
These subgroups can be thought of as the first two steps of the Moy--Prasad filtration of the parahoric subgroup of $\GL_{N}(F)$ associated to the barycenter of an alcove of the Bruhat--Tits building of $\GL_{N}(F)$.
We define $I^{++}_{\GL_{N}}$ to be the next step of this Moy--Prasad filtration.
Explicitly, these subgroups are given as follows:
\[
I_{\GL_{N}}^{+} = \begin{pmatrix}
 1+\mfp&&\mcO\\
 &\ddots&\\
 \mfp&&1+\mfp
\end{pmatrix}
\supset
I_{\GL_{N}}^{++} = \begin{pmatrix}
 1+\mfp&\mfp&&\mcO\\
 &\ddots&\ddots&\\
 &\mfp&\ddots&\mfp\\
\mfp^2&&&1+\mfp
\end{pmatrix}.
\]
Then,
\[
I_{\GL_{N}}^{+}/I_{\GL_{N}}^{++} \cong k^{\oplus N}\colon
(x_{ij})_{ij} \mapsto \bigl(\ol{x_{1,2}}, \ldots, \ol{x_{N-1, N}}, \ol{x_{N,1}\varpi^{-1}}\bigr).
\]
Let $Z_{\GL_{N}}$ be the center of $\GL_{N}(F)$ and $Z_{\GL_{N},0}$ the maximal compact subgroup of $Z_{\GL_{N}}$.
For any character $\omega$ of $k^{\times}$ and $a\in k^{\times}$, we define $\chi^{\GL_{N}}_{\omega,a}$ to be the character on $Z_{\GL_{N},0}I_{\GL_{N}}^{+}$ such that
\begin{itemize}
\item
$\chi^{\GL_{N}}_{\omega,a}|_{Z_{\GL_{N},0}}$ is the pull back of $\omega$ via the map $Z_{\GL_{N},0}\cong\mcO^{\times}\twoheadrightarrow k^{\times}$, and
\item
$\chi^{\GL_{N}}_{\omega,a}|_{I_{\GL_{N}}^{+}}$ is the pull back of the character on $k^{\oplus N}$ given by
\[
(x_{1},\ldots,x_{N-1},x_{N})\mapsto \psi(x_{1}+\cdots+x_{N-1}+ax_{N})
\]
via the map $I^{+}_{\GL_{N}}\twoheadrightarrow I^{+}_{\GL_{N}}/I^{++}_{\GL_{N}}\cong k^{\oplus N}$.
\end{itemize}
This character is \textit{affine generic} in the sense of Gross--Reeder (see \cite[Definition 2.3]{Oi24}).
Then the stabilizer group $N_{\GL_{N}(F)}(I_{\GL_{N}}^{+}; \chi_{\omega,a}^{\GL_{N}}):=\{n\in N_{\GL_{N}(F)}(I_{\GL_{N}}^{+}) \mid (\chi^{\GL_{N}}_{\omega,a})^{n}=\chi^{\GL_{N}}_{\omega,a}\}$ of $\chi^{\GL_{N}}_{\omega,a}$ is given by $Z_{\GL_{N}} I_{\GL_{N}}^{+} \lan\varphi_{a^{-1}}^{\GL_{N}}\ran$, where we put
\[
\varphi_{a^{-1}}^{\GL_{N}} :=
\begin{pmatrix}
0 & I_{N-1} \\
\varpi a^{-1} & 0
\end{pmatrix} \in \GL_{N}(F)
\]
(note that $(\varphi_{a^{-1}}^{\GL_{N}})^{N}=\varpi a^{-1} I_{N}$).
Thus, for any $\zeta\in\C^{\times}$, we can extend $\chi^{\GL_{N}}_{\omega,a}$ to a character $\tilde{\chi}^{\GL_{N}}_{\omega,a,\zeta}$ on $Z_{\GL_{N}}I_{\GL_{N}}^{+}\lan\varphi_{a^{-1}}^{\GL_{N}}\rangle$ by putting $\tilde{\chi}^{\GL_{N}}_{\omega,a,\zeta}(\varphi_{a^{-1}}^{\GL_{N}}):=\zeta$.

Let $\pi_{\omega,a,\zeta}^{\GL_{N}}$ be the representation of $\GL_{N}(F)$ defined by
\[
\pi_{\omega,a,\zeta}^{\GL_{N}}:=\cInd^{\GL_{N}(F)}_{Z_{\GL_{N}}I_{\GL_{N}}^{+}\lan\varphi_{a^{-1}}^{\GL_{N}}\rangle} \tilde{\chi}^{\GL_{N}}_{\omega,a,\zeta}.
\]
Then
\[
\{
\pi^{\GL_{N}}_{\omega,a,\zeta}
\mid
(\omega,a,\zeta)\in (k^{\times})^{\vee}\times k^{\times}\times\C^{\times}
\}
\]
is a complete set of representatives of the set of equivalence classes of simple supercuspidal representations of $\GL_{N}(F)$ by \cite[Proposition 2.7]{Oi24}.

Let us investigate when $\pi^{\GL_{N}}_{\omega,a,\zeta}$ is self-dual, or equivalently by the well-known result of Gelf'and and Kazhdan (see e.g., \cite[7.3]{BZ1}), $\theta$-stable, where $\theta$ is the involution of $\GL_{N}$ defined by
\[
\theta(g):=J_{N}{}^{t}g^{-1}J_{N}^{-1}.
\]

For a subgroup $H<\GL_N(F)$ which is stabilized by $\theta$, i.e., $\theta(H)=H$, and a representation $\varrho$ of $H$, denote
$\varrho^{\theta}(g)=\varrho(\theta(g))$. Observe that $\theta$ stabilizes $I_{\GL_{N}}^{+}$ and $I_{\GL_{N}}^{++}$, hence
$\theta$ acts on the quotient $I_{\GL_{N}}^{+}/I_{\GL_{N}}^{++}$. For a representative
$(x_{ij})_{ij}\in I_{\GL_{N}}^{+}/I_{\GL_{N}}^{++}$, the image of $\theta((x_{ij})_{ij})$ in $k^{\oplus N}$ is given by
\[\theta((x_{ij}))\mapsto \bigl(\ol{x_{N-1,N}}, \ldots, \ol{x_{1, 2}}, (-1)^{N}\ol{x_{N,1}\varpi^{-1}}\bigr).\]
Thus $(\chi^{\GL_{N}}_{\omega,a})^{\theta}=\chi^{\GL_{N}}_{\omega^{-1},(-1)^Na}$.
In addition $\theta(\varphi_{a}^{\GL_{N}})=-(\varphi_{(-1)^Na}^{\GL_{N}})^{-1}$, hence
\[(\tilde{\chi}^{\GL_{N}}_{\omega,a,\zeta})^{\theta}(\varphi_{(-1)^Na^{-1}}^{\GL_{N}})=
\tilde{\chi}^{\GL_{N}}_{\omega,a,\zeta}(-(\varphi_{a^{-1}}^{\GL_{N}})^{-1})=\omega(-1)\zeta^{-1}.\]
Therefore
\[
(\pi_{\omega,a,\zeta}^{\GL_{N}})^{\theta}
\cong
\pi_{\omega^{-1},(-1)^{N}a,\omega(-1)\zeta^{-1}}^{\GL_{N}}.
\]
Hence, $(\pi_{\omega,a,\zeta}^{\GL_{N}})^{\theta}\cong\pi_{\omega,a,\zeta}^{\GL_{N}}$ if and only if
\[
\omega=\omega^{-1},\quad
a=(-1)^{N}a,\quad
\zeta=\omega(-1)\zeta^{-1}.
\]

We first consider the case where $p\neq2$.
By the condition that $a=(-1)^{N}a$, there exists a self-dual simple supercuspidal representation only when $N$ is even.
In this case,
\[
\{
\pi^{\GL_{N}}_{\omega,a,\zeta}
\mid
(\omega,a,\zeta)\in (k^{\times})^{\vee}\times k^{\times}\times\C^{\times},\,
\omega^{2}=\mathbbm{1},\,
\zeta^{2}=\omega(-1)
\}
\]
gives a complete set of representatives of the set of equivalence classes of self-dual simple supercuspidal representations of $\GL_{N}(F)$.

We next consider the case where $p=2$.
In this case, the condition $\omega=\omega^{-1}$ implies that $\omega$ is the trivial character of $k^{\times}$.
Thus the condition $\zeta=\omega(-1)\zeta^{-1}$ means that $\zeta$ is a sign.
Moreover, the condition $a=(-1)^{N}a$ is always satisfied.
Therefore we see that
\[
\{
\pi^{\GL_{N}}_{\mathbbm{1},a,\zeta}
\mid
a\in k^{\times}, \zeta\in\{\pm1\}
\}
\]
is a complete set of representatives of the set of equivalence classes of self-dual simple supercuspidal representations of $\GL_{N}(F)$.

\subsection{Simple supercuspidal representations of \texorpdfstring{$\SO_{2n+1}$}{SO(2n+1)}}\label{subsec:SSC-odd}
%
We next consider the case of
\[
\SO_{2n+1}:=\{g\in\SL_{2n+1} \mid {}^{t}\!gJ_{2n+1}g=J_{2n+1}\}.
\]

We let $I_{\SO_{2n+1}}$ be the Iwahori subgroup of $\SO_{2n+1}(F)$ consisting of the elements of $\SO_{2n+1}(F)$ belonging to
\[
\begin{pmatrix}
\mcO^{\times}&&\multicolumn{1}{c:}{\mcO}&\multicolumn{1}{c:}{\mcO}&&&\\
 &\ddots&\multicolumn{1}{c:}{}&\multicolumn{1}{c:}{\vdots}&&\frac{1}{2}\mcO&\\
 \mfp&&\multicolumn{1}{c:}{\mcO^{\times}}&\multicolumn{1}{c:}{\mcO}&&&\\
 \cdashline{1-10}
 2\mfp&\cdots&\multicolumn{1}{c:}{2\mfp}&\multicolumn{1}{c:}{\mcO^{\times}}&\mcO&\cdots&\mcO\\
 \cdashline{1-10}
 &&\multicolumn{1}{c:}{}&\multicolumn{1}{c:}{2\mfp}&\mcO^{\times}&&\mcO\\
 &2\mfp&\multicolumn{1}{c:}{}&\multicolumn{1}{c:}{\vdots}&&\ddots&\\
 &&\multicolumn{1}{c:}{}&\multicolumn{1}{c:}{2\mfp}&\mfp&&\mcO^{\times}\\
\end{pmatrix}.
\]
(we give some details in Section \ref{subsec:Iwahori}).
Similarly to the case of $\GL_{N}$, we let $I^{+}_{\SO_{2n+1}}$ be the pro-unipotent radical of $I_{\SO_{2n+1}}$ and $I^{++}_{\SO_{2n+1}}$ the next step of the Moy--Prasad filtration with respect to the barycenter of the alcove corresponding to $I_{\SO_{2n+1}}$.
Then,
\[
I^{+}_{\SO_{2n+1}}/I^{++}_{\SO_{2n+1}}\cong k^{\oplus n+1}\colon
(g_{ij})_{ij}
\mapsto
(\overline{g_{12}},\ldots,\overline{g_{n,n+1}},\overline{g_{2n,1}\cdot2^{-1}\varpi^{-1}}).
\]
For any $a\in k^{\times}$, we define an affine generic character $\chi_{a}^{\SO_{2n+1}}$ of $I^{+}_{\SO_{2n+1}}$ by pulling back the character
\[
k^{\oplus n+1}\rightarrow \C^{\times}
\colon
(x_{1},\ldots,x_{n+1})\mapsto \psi(x_{1}+\cdots+x_{n}+ax_{n+1})
\]
via the map $I^{+}_{\SO_{2n+1}}\twoheadrightarrow I^{+}_{\SO_{2n+1}}/I^{++}_{\SO_{2n+1}}\cong k^{\oplus n+1}$.
Then the stabilizer group
of $\chi^{\SO_{2n+1}}_{a}$ is given by $I^{+}_{\SO_{2n+1}}\langle\varphi_{a^{-1}}^{\SO_{2n+1}}\rangle$, where $\varphi_{a^{-1}}^{\SO_{2n+1}}$ is an element of order $2$ given by
\[
\varphi_{a^{-1}}^{\SO_{2n+1}}:=
-\begin{pmatrix}
&&a2^{-1}\varpi^{-1}\\
&I_{2n-1}&\\
a^{-1}2\varpi&&
\end{pmatrix}
\in\SO_{2n+1}(F).
\]
Hence, for any $\zeta\in\{\pm1\}$, we can extend $\chi^{\SO_{2n+1}}_{a}$ to a character $\tilde{\chi}^{\SO_{2n+1}}_{a,\zeta}$ on $I^{+}_{\SO_{2n+1}}\langle\varphi_{a^{-1}}^{\SO_{2n+1}}\rangle$ by putting $\tilde{\chi}^{\SO_{2n+1}}_{a,\zeta}(\varphi_{a^{-1}}^{\SO_{2n+1}}):=\zeta$.

Let $\pi_{a,\zeta}^{\SO_{2n+1}}$ be the representation of $\SO_{2n+1}(F)$ defined by
\[
\pi^{\SO_{2n+1}}_{a,\zeta}:=\cInd^{\SO_{2n+1}(F)}_{I^{+}_{\SO_{2n+1}}\langle\varphi_{a^{-1}}^{\SO_{2n+1}}\rangle} \tilde{\chi}^{\SO_{2n+1}}_{a,\zeta}.
\]
Then
\[
\{
\pi^{\SO_{2n+1}}_{a,\zeta}
\mid
a\in k^{\times}, \zeta\in\{\pm1\}
\}
\]
gives a complete set of representatives of the set of equivalence classes of simple supercuspidal representations of $\SO_{2n+1}(F)$ by \cite[Proposition 2.7]{Oi24}.

\begin{rem}\label{rem:ssc-odd-2}
We caution that the above parametrization of simple supercuspidal representations differs from the one in \cite{Oi19-AJM}, where $p$ is supposed to be odd, by an extra factor $2$.
To be more precise, for any $a\in k^{\times}$, the simple supercuspidal representation denoted by ``$\pi_{a,\zeta}'$'' in \cite[Section 2.4]{Oi19-AJM} is equal to $\pi_{2a,\zeta}^{\SO_{2n+1}}$ defined in this paper.
\end{rem}

\subsection{Simple supercuspidal representations of \texorpdfstring{$\SO_{2n}$}{SO(2n)}}\label{subsec:SSC-even}
%
We finally consider the case of
\[
\SO_{2n}:=\{g\in\SL_{2n} \mid {}^{t}\!gJ'_{2n}g=J'_{2n}\},
\]
where $J'_{2n}$ is the anti-diagonal matrix of size $2n$ whose anti-diagonal entries are given by $1$.
Here, we suppose that $n\geq2$ because $\SO_{2}$ is abelian.

We let $I_{\SO_{2n}}$ be the Iwahori subgroup of $\SO_{2n}(F)$ consisting of the elements of $\SO_{2n}(F)$ belonging to
\[
\begin{pmatrix}
\mcO^{\times}&&\multicolumn{1}{c:}{\mcO}&\mcO&\multicolumn{1}{c:}{\mcO}&&&\\
 &\ddots&\multicolumn{1}{c:}{}&\vdots&\multicolumn{1}{c:}{\vdots}&&\mcO&\\
 \mfp&&\multicolumn{1}{c:}{\mcO^{\times}}&\mcO&\multicolumn{1}{c:}{\mcO}&&&\\
 \cdashline{1-10}
 \mfp&\cdots&\multicolumn{1}{c:}{\mfp}&\mcO^{\times}&\multicolumn{1}{c:}{\mfp}&\mcO&\cdots&\mcO\\
\mfp&\cdots&\multicolumn{1}{c:}{\mfp}&\mfp&\multicolumn{1}{c:}{\mcO^{\times}}&\mcO&\cdots&\mcO\\
 \cdashline{1-10}
 &&\multicolumn{1}{c:}{}&\mfp&\multicolumn{1}{c:}{\mfp}&\mcO^{\times}&&\mcO\\
 &\mfp&\multicolumn{1}{c:}{}&\vdots&\multicolumn{1}{c:}{\vdots}&&\ddots&\\
 &&\multicolumn{1}{c:}{}&\mfp&\multicolumn{1}{c:}{\mfp}&\mfp&&\mcO^{\times}\\
\end{pmatrix}
\]
(we give some details in Section \ref{subsec:Iwahori}).
Similarly to the case of $\GL_{N}$, we let $I^{+}_{\SO_{2n}}$ be the pro-unipotent radical of $I_{\SO_{2n}}$ and $I^{++}_{\SO_{2n}}$ the next step of the Moy--Prasad filtration with respect to the barycenter of the alcove corresponding to $I_{\SO_{2n}}$.
Then,
\[
I^{+}_{\SO_{2n}}/I^{++}_{\SO_{2n}}\cong k^{\oplus n+1}\colon
(g_{ij})_{ij}
\mapsto
(\overline{g_{12}},\ldots,\overline{g_{n-1,n}},\overline{g_{n-1,n+1}},\overline{g_{2n-1,1}\varpi^{-1}}).
\]
Let $Z_{\SO_{2n}}=\{\pm I_{2n}\}$ be the center of $\SO_{2n}(F)$.
For $(\xi,\kappa,a)\in\{\pm1\}\times\{0,1\}\times k^{\times}$, we define an affine generic character $\chi^{\SO_{2n}}_{\xi,\kappa,a}$ on $Z_{\SO_{2n}}I_{\SO_{2n}}^{+}$ by
\begin{itemize}
\item
$\chi^{\SO_{2n}}_{\xi,\kappa,a}(-I_{2n})=\xi$, and
\item
$\chi^{\SO_{2n}}_{\xi,\kappa,a}|_{I_{\SO_{2n}}^{+}}$ is the pull back of the character on $k^{\oplus n+1}$ given by
\[
(x_{1},\ldots,x_{n-1},x_{n},x_{n+1})\mapsto \psi(x_{1}+\cdots+x_{n-1}+\epsilon^{\kappa}x_{n}+ax_{n+1})
\]
via the map $I^{+}_{\SO_{2n}}\twoheadrightarrow I^{+}_{\SO_{2n}}/I^{++}_{\SO_{2n}}\cong k^{\oplus n+1}$.
\end{itemize}
Then the stabilizer group
of $\chi^{\SO_{2n}}_{\xi,\kappa,a}$ is given by $Z_{\SO_{2n}} I_{\SO_{2n}}^{+} \lan\varphi_{\epsilon^{\kappa},-a^{-1}}^{\SO_{2n}}\ran$, where we put
\[
\varphi_{\alpha,\beta}^{\SO_{2n}}:=
\begin{pmatrix}
&&&&&(\beta\varpi)^{-1}\\
&I_{n-2}&&&&\\
&&&\alpha^{-1}&&\\
&&\alpha&&&\\
&&&&I_{n-2}&\\
\beta\varpi&&&&&
\end{pmatrix}
\]
for any $\alpha,\beta\in k^{\times}$.
Thus, for any $\zeta\in\{\pm1\}$, we can extend $\chi^{\SO_{2n}}_{\xi,\kappa,a}$ to a character $\tilde{\chi}_{\xi,\kappa,a,\zeta}^{\SO_{2n}}$ on $Z_{\SO_{2n}} I_{\SO_{2n}}^{+} \lan\varphi_{\epsilon^{\kappa},-a^{-1}}^{\SO_{2n}}\ran$ by putting $\tilde{\chi}_{\xi,\kappa,a,\zeta}^{\SO_{2n}}(\varphi_{\epsilon^{\kappa},-a^{-1}}^{\SO_{2n}}):=\zeta$.

Let $\pi^{\SO_{2n}}_{\xi,\kappa,a,\zeta}$ be the representation of $\SO_{2n}(F)$ defined by
\[
\pi^{\SO_{2n}}_{\xi,\kappa,a,\zeta}
:=
\cInd^{\SO_{2n}}_{Z_{\SO_{2n}} I_{\SO_{2n}}^{+}\lan\varphi_{\epsilon^{\kappa},-a^{-1}}^{\SO_{2n}}\ran} \tilde{\chi}^{\SO_{2n}}_{\xi,\kappa,a,\zeta}.
\]

When $p\neq2$,
\[
\{
\pi^{\SO_{2n}}_{\xi,\kappa,a,\zeta}
\mid
\xi\in\{\pm1\}, \kappa\in\{0,1\}, a\in k^{\times}, \zeta\in\{\pm1\}
\}
\]
gives a complete set of representatives of the set of equivalence classes of simple supercuspidal representations of $\SO_{2n}(F)$.
Moreover, any $\pi^{\SO_{2n}}_{\xi,\kappa,a,\zeta}$ is stable under the action of $\mathrm{O}_{2n}(F)$ (see \cite[Section 2.6]{Oi24}).

When $p=2$, since $Z_{\SO_{2n}}$ is contained in $I^{++}_{\SO_{2n}}$, the parameter $\xi$ must be $1$.
Furthermore, since $\epsilon=1$, obviously $\pi^{\SO_{2n}}_{\xi,1,a,\zeta}=\pi^{\SO_{2n}}_{\xi,0,a,\zeta}$.
By noting these points, we can check that
\[
\{
\pi^{\SO_{2n}}_{1,0,a,\zeta}
\mid
a\in k^{\times}, \zeta\in\{\pm1\}
\}
\]
gives a complete set of representatives of the set of equivalence classes of simple supercuspidal representations of $\SO_{2n}(F)$.
In the following, when $p=2$, we write $\pi^{\SO_{2n}}_{a,\zeta}$ (resp.\ $\tilde{\chi}^{\SO_{2n}}_{a,\zeta}$) instead of $\pi^{\SO_{2n}}_{1,0,a,\zeta}$ (resp.\ $\tilde{\chi}^{\SO_{2n}}_{1,0,a,\zeta}$), for short.
Note that any $\pi^{\SO_{2n}}_{a,\zeta}$ is stable under the action of $\mathrm{O}_{2n}(F)$ (the same argument as in \cite[Section 2.6]{Oi24} works also in the case where $p=2$).

\section{Local Langlands correspondence for \texorpdfstring{$\SO_{N}$}{SO(N)}}\label{sec:LLC}

\subsection{Local Langlands correspondence for \texorpdfstring{$\SO_{N}$}{SO(N)}}\label{subsec:LLC}

For any connected reductive group $\G$ over $F$, we let $\hat{\G}$ denote the Langlands dual group and put $\L\G:=\hat{\G}\rtimes W_{F}$.
We say that a homomorphism $\phi\colon W_{F}\times\SL_{2}(\C)\rightarrow \L\G$ is an \textit{$L$-parameter} of $\G$ if $\phi$ is smooth on $W_{F}$, $\phi$ is compatible with the projections to $W_{F}$, and the restriction $\phi|_{\SL_{2}(\C)}\colon\SL_{2}(\C)\rightarrow\hat{\G}$ is algebraic.
We let
\begin{itemize}
\item
$\Pi_{\temp}(\G)$ be the set of equivalence classes of irreducible tempered representations of $\G(F)$, and
\item
$\Phi_{\temp}(\G)$ be the set of $\hat{\G}$-conjugacy classes of tempered (i.e., the image of $W_{F}$ is bounded in $\hat{\G}$) $L$-parameters of $\G$.
\end{itemize}

We are interested in the case where $\G$ is the split special orthogonal group $\SO_{N}$.
The Langlands dual group of $\SO_{N}$ is given by
\[
\begin{cases}
\Sp_{2n}(\C)& \text{if $N=2n+1$,}\\
\SO_{2n}(\C)& \text{if $N=2n$.}
\end{cases}
\]
In both cases, the Galois action on $\hat{\G}$ is trivial.

Hence, an $L$-parameter of $\SO_{2n+1}$ can be regarded as a $2n$-dimensional symplectic representation of $W_{F}\times\SL_{2}(\C)$.
It is known that two $L$-parameters of $\SO_{2n+1}$ are conjugate by $\Sp_{2n}(\C)$ if and only if they are equivalent as $2n$-dimensional representations of $W_{F}\times\SL_{2}(\C)$ (\cite[Theorem 8.1 (ii)]{GGP12-1}).
Thus the set $\Phi_{\temp}(\SO_{2n+1})$ can be identified with the set of isomorphism classes of $2n$-dimensional symplectic representations of $W_{F}\times\SL_{2}(\C)$ which are bounded on $W_{F}$.

Similarly, an $L$-parameter of $\SO_{2n}$ can be regarded as a $2n$-dimensional orthogonal representation of $W_{F}\times\SL_{2}(\C)$ with trivial determinant.
It is known that two $L$-parameters of $\SO_{2n}$ are conjugate by $\mathrm{O}_{2n}(\C)$ if and only if they are equivalent as $2n$-dimensional representations of $W_{F}\times\SL_{2}(\C)$ (\cite[Theorem 8.1 (ii)]{GGP12-1}).
By noting this, we let $\tilde{\Phi}_{\temp}(\SO_{2n})$ be the set of $\mathrm{O}_{2n}(\C)$-conjugacy classes of tempered $L$-parameters of $\SO_{2n}$.
Thus the set $\tilde{\Phi}_{\temp}(\SO_{2n})$ can be identified with the set of isomorphism classes of $2n$-dimensional orthogonal representations of $W_{F}\times\SL_{2}(\C)$ which are bounded on $W_{F}$ and have trivial determinant.
We also define $\tilde{\Pi}_{\temp}(\SO_{2n})$ to be the set of $\mathrm{O}_{2n}(F)$-orbits in $\Pi_{\temp}(\SO_{2n})$.

For the notational convenience, let us put $\tilde{\Phi}_{\temp}(\SO_{2n+1}):=\Phi_{\temp}(\SO_{2n+1})$ and $\tilde{\Pi}_{\temp}(\SO_{2n+1}):=\Pi_{\temp}(\SO_{2n+1})$ in the case where $N=2n+1$.

For any $\phi\in\tilde{\Phi}_{\temp}(\SO_{N})$, we define a finite group $\overline{S}_{\phi}$ as follows:
\begin{align*}
S_{\phi}&:=\Cent_{\hat{\G}}(\mathrm{Im}(\phi)),\\
\overline{S}_{\phi}&=S_{\phi}/(S_{\phi}^{\circ}Z_{\hat{\G}}),
\end{align*}
where $S_{\phi}^{\circ}$ denotes the identity component of $S_{\phi}$ and $Z_{\hat{\G}}$ denotes the center of $\hat{\G}$.
Here, we implicitly fix a representative of the equivalence class $\phi$ and again write $\phi$ for it by abuse of notation.

Recall that a Whittaker datum of $\G$ is a pair $\mfw=(\bfB,\lambda)$ of an $F$-rational Borel subgroup $\bfB$ of $\G$ and a generic character $\lambda$ of $U=\bfU(F)$, where $\bfU$ is the unipotent radical of $\bfB$.
In the following, let us fix a Whittaker datum $\mfw$ of $\G=\SO_{N}$ defined as follows:
\begin{enumerate}
\item
When $N=2n+1$, we take $\bfB$ to be the upper-triangular Borel subgroup of $\SO_{2n+1}$:
\[
\bfB
=
\begin{pmatrix}
\ast&\cdots&\ast\\
&\ddots&\vdots\\
&&\ast
\end{pmatrix}.
\]
We define a generic character $\lambda$ of $U$ by
\[
\lambda(y)=
\psi(y_{1,2}+\cdots+y_{n-1,n}+y_{n,n+1})
\]
for any $y=(y_{ij})\in U$.
\item
When $N=2n$, we take $\bfB$ to be the upper-triangular Borel subgroup of $\SO_{2n}$:
\[
\mathbf{B}:=\begin{pmatrix}
\ast&\cdots&\multicolumn{1}{c:}{\ast}&\ast&\multicolumn{1}{c:}{\ast}&&&\\
 &\ddots&\multicolumn{1}{c:}{\vdots}&\vdots&\multicolumn{1}{c:}{\vdots}&&\ast&\\
 &&\multicolumn{1}{c:}{\ast}&\ast&\multicolumn{1}{c:}{\ast}&&&\\
 \cdashline{1-8}
 &&\multicolumn{1}{c:}{}&\ast&\multicolumn{1}{c:}{\ast}&\ast&\cdots&\ast\\
 &&\multicolumn{1}{c:}{}&\ast&\multicolumn{1}{c:}{\ast}&\ast&\cdots&\ast\\
 \cdashline{1-8}
 &&\multicolumn{1}{c:}{}&&\multicolumn{1}{c:}{}&\ast&\cdots&\ast\\
 &&\multicolumn{1}{c:}{}&&\multicolumn{1}{c:}{}&&\ddots&\vdots\\
 &&\multicolumn{1}{c:}{}&&\multicolumn{1}{c:}{}&&&\ast
\end{pmatrix}.
\]
We define a generic character $\lambda$ of $U$ by
\[
\lambda(y)=
\psi(y_{1,2}+\cdots+y_{n-1,n}+y_{n-1,n+1})
\]
for any $y=(y_{ij})\in U$.
\end{enumerate}

The \textit{local Langlands correspondence for tempered representations of $\SO_{N}$}, which was established by Arthur (\cite[Theorems 1.5.1 and 2.2.1]{Art13}), asserts that there exists a natural map
\[
\LLC_{\SO_{N}}\colon\tilde{\Pi}_{\temp}(\SO_{N})\rightarrow\tilde{\Phi}_{\temp}(\SO_{N}),
\]
which is surjective and with finite fibers.
In other words, by letting $\tilde{\Pi}_{\phi}^{\SO_{N}}$ be the fiber of the map $\LLC_{\SO_{N}}$ at an $L$-parameter $\phi$, we obtain a natural partition
\[
\tilde{\Pi}_{\temp}(\SO_{N})
=
\bigsqcup_{\phi\in\tilde{\Phi}_{\temp}(\SO_{N})}\tilde{\Pi}_{\phi}^{\SO_{N}},
\]
where each $\tilde{\Pi}_{\phi}^{\SO_{N}}$ is finite.
For any $\phi\in\tilde{\Phi}_{\temp}(\SO_{N})$, the finite set $\tilde{\Pi}_{\phi}^{\SO_{2n}}$ is called an \textit{$L$-packet} and is equipped with a bijective map depending on the chosen Whittaker datum $\mfw$ of $\SO_{N}$:
\[
\iota_{\mfw}\colon \tilde{\Pi}_{\phi}^{\SO_{N}}\xrightarrow{1:1}\overline{S}_{\phi}^{\vee}
\]
 to the set $\overline{S}_{\phi}^{\vee}$ of irreducible characters of $\overline{S}_{\phi}$.

Note that there exists a unique representation in each tempered $L$-packet which is mapped to the trivial character of $\overline{S}_{\phi}$ under the map $\iota_{\mfw}$.
In general, Shahidi's \textit{generic packet conjecture} predicts that each tempered $L$-packet contains a unique $\mfw$-generic representation $\pi_{\mfw}$ and that $\iota_{\mfw}(\pi_{\mfw})=\mathbbm{1}$.
In the present setting, Shahidi's conjecture has been verified by \cite[Proposition 8.3.2]{Art13}, \cite{Var17}, and \cite{Ato17}.

\subsection{Result of M{\oe}glin and Xu}\label{subsec:Moeglin-Xu}
We say that a tempered $L$-parameter $\phi\in\tilde{\Phi}_{\temp}(\SO_{N})$ is \textit{discrete} if its centralizer group $S_{\phi}$ is finite, or equivalently (see \cite[Section 4]{GGP12-1}), $\phi$ is the direct sum of pairwise inequivalent irreducible symplectic (resp.\ orthogonal) representations as a $2n$-dimensional representation of $W_{F}\times\SL_{2}(\C)$ when $N=2n+1$ (resp.\ when $N=2n$).
Arthur's result assures that $\phi$ is discrete if and only if $\tilde{\Pi}_{\phi}^{\SO_{N}}$ contains a discrete series representation of $\SO_{N}(F)$, and that, in this case, every member of $\tilde{\Pi}_{\phi}^{\SO_{N}}$ is a discrete series.
(We refer the reader to \cite[Section 2]{Xu17-MM} for details.)

In general, it is possible that the $L$-packet $\tilde{\Pi}_{\phi}^{\SO_{N}}$ for a discrete $L$-parameter $\phi$ contains both a supercuspidal representation and a non-supercuspidal discrete series representation.
In \cite{Moeg11} and \cite{Xu17-MM}, M{\oe}glin and Xu gave a parametrization of the supercuspidal members in a given discrete series $L$-packet in terms of $\overline{S}_{\phi}^{\vee}$.

As an easy consequence of their result, we deduce the following (see \cite[Corollary 4.9]{Oi24} for the details):

\begin{prop}\label{prop:Moeglin-Xu}
For any discrete $L$-parameter $\phi\in\tilde{\Phi}_{\temp}(\SO_{N})$, the following are equivalent:
\begin{enumerate}
\item
$\phi$ is trivial on $\SL_{2}(\C)$;
\item
the unique $\mfw$-generic representation in $\tilde{\Pi}_{\phi}^{\SO_{N}}$ is supercuspidal;
\item
every member of $\tilde{\Pi}_{\phi}^{\SO_{N}}$ is supercuspidal.
\end{enumerate}
\end{prop}

\begin{rem}\label{rem:Moeglin-Xu}
\begin{enumerate}
\item
What follows from \cite{Moeg11} and \cite{Xu17-MM} directly is the equivalence between (1), (3) and
\begin{quote}
(2') the member of $\tilde{\Pi}_{\phi}^{\SO_{N}}$ which is mapped to $\mathbbm{1}$ under $\iota_{\mfw}$ is supercuspidal.
\end{quote}
Hence, in order to rephrase this condition as in (2), we need to appeal to Shahidi's conjecture.
\item
The logical relationship between \cite{Moeg11} and \cite{Xu17-MM} is as follows.
In \cite{Moeg11}, M{\oe}glin established an explicit construction of discrete series $L$-packets (more generally, $A$-packets) of classical groups modulo the local Langlands correspondence for $L$-packets consisting only of supercuspidal representations.
In particular, her result gives the above-mentioned parametrization of supercuspidal members in each discrete series $L$-packet.
However, a priori, it is not obvious at all whether her construction is consistent with Arthur's one \cite{Art13}.
What Xu did in \cite{Xu17-MM} is to prove that these two constructions indeed coincide.
\end{enumerate}
\end{rem}

\subsection{Formal degree conjecture of Hiraga--Ichino--Ikeda}\label{subsec:FDC}

In \cite[Conjecture 1.4]{HII08}, Hiraga--Ichino--Ikeda proposed the following conjecture (here we state the conjecture according to a reformulation by Gross--Reeder, \cite[Conjecture 7.1 (5)]{GR10}):

\begin{conj}[Formal degree conjecture]\label{conj:FDC}
Let $\phi\in\tilde{\Phi}_{\temp}(\SO_{N})$ be a discrete $L$-parameter.
Then, for any $\pi\in\tilde{\Pi}_{\phi}^{\SO_{N}}$,
\[
|\deg_{\mu}(\pi)|
=
\frac{1}{|\overline{S}_{\phi}|}\cdot\frac{|\gamma(0,\Ad\circ\phi,\psi_{0})|}{|\gamma(0,\Ad\circ\phi_{\mathrm{pr}},\psi_{0})|}.
\]
Here,
\begin{itemize}
\item
$\deg_{\mu}(\pi)$ is the formal degree of $\pi$ with respect to the Euler--Poincare measure $\mu$ of $\SO_{N}(F)$ (see \cite[Section 7.1]{GR10}),
\item
$\Ad$ is the adjoint representation of $\hat{\G}$ on its Lie algebra $\Lie\hat{\G}$,
\item
$\gamma(s,-,\psi_{0})$ is the $\gamma$-factor for representations of $W_{F}\times\SL_{2}(\C)$ (see \cite[Section 2.2]{GR10}) with respect to a nontrivial additive character $\psi_{0}$ of $F$ of level $0$, i.e., $\psi_{0}$ is trivial on $\mcO$ but not on $\mfp^{-1}$, and
\item
$\phi_{\mathrm{pr}}$ denotes the principal parameter in the sense of Gross--Reeder (see \cite[Section 3.3]{GR10}).
\end{itemize}
\end{conj}

\begin{rem}
In \cite{HII08}, the formal degree conjecture is formulated for any quasi-split connected reductive group $\G$.
In general, the right-hand side of the identity of Conjecture \ref{conj:FDC} must contain one more term ``$\langle1,\pi\rangle$'' (see \cite[Conjecture 1.4]{HII08}).
Here $\langle-,\pi\rangle$ denotes the irreducible character of $\overline{S}_{\phi}$ corresponding to $\pi$ under the map $\iota_{\mfw}\colon \tilde{\Pi}_{\phi}^{\SO_{N}}\rightarrow\overline{S}_{\phi}^{\vee}$.
In fact, the group $\overline{S}_{\phi}$ is always abelian when $\G=\SO_{N}$.
Accordingly, $\langle1,\pi\rangle$ is always given by $1$.
\end{rem}

The formal degree conjecture for the odd special orthogonal group $\SO_{2n+1}$ was proved by Ichino--Lapid--Mao \cite{ILM17}.
For the even special orthogonal group $\SO_{2n}$, recently Beuzart-Plessis announced that he has proved the formal degree conjecture (\cite{BP21-OWR}).

\section{Analysis of symmetric and exterior square $L$-factors}\label{sec:L}

In this section, we prove several results on the symmetric and exterior square $L$-factors of self-dual irreducible Galois representations, which will be needed later.

The following lemma is proved in \cite[Lemma 4.11]{HO22}.

\begin{lem}\label{lem:HO-4.1}
Let $\rho$ be a finite-dimensional irreducible smooth representation of $W_{F}$.
\begin{enumerate}
\item
The number of irreducible components of $\rho|_{I_{F}}$ is equal to the degree of the maximal unramified extension $E$ of $F$ from which $\rho$ is induced, where $I_{F}$ denotes the inertia subgroup of $W_{F}$.
\item
If we let $\sigma$ be a representation of $W_{E}$ such that $\rho\cong\Ind_{W_{E}}^{W_{F}}\sigma$, then the restriction of $\sigma$ to $I_{F}$ is irreducible and $\Gal(E/F)$-regular, i.e., $(\sigma|_{I_{F}})^{\gamma}\ncong\sigma|_{I_{F}}$ for any $\gamma\in\Gal(E/F)\smallsetminus\{1\}$.
\item
An unramified character $\omega$ of $F^{\times}$ satisfies $\rho\otimes\omega\cong\rho$ if and only if $\omega^{d}=\mathbbm{1}$ for $d:=[E:F]$, or equivalently, $\omega|_{W_{E}}=\mathbbm{1}$.
\end{enumerate}
\end{lem}

\begin{prop}\label{prop:par-L}
Let $\rho$ be a self-dual finite-dimensional irreducible smooth representation of $W_{F}$.
Let $E$ and $\sigma$ be as in Lemma \ref{lem:HO-4.1}, thus $\rho\cong\Ind_{W_{E}}^{W_{F}}\sigma$.
We put $d:=[E:F]$.
\begin{enumerate}
\item
When $\rho$ is orthogonal,
\[
L(s,\wedge^{2}\rho)
=
\begin{cases}
1 & \text{if $\sigma$ is self-dual,} \\
(1+q^{-es})^{-1} &  \text{if $\sigma$ is not self-dual.}
\end{cases}
\]
\item
When $\rho$ is symplectic,
\[
L(s,\Sym^{2}\rho)
=
\begin{cases}
1 & \text{if $\sigma$ is self-dual,} \\
(1+q^{-es})^{-1} &  \text{if $\sigma$ is not self-dual.}
\end{cases}
\]
\end{enumerate}
Here, the degree $d$ must be even when $\sigma$ is not self-dual in both cases, hence we put $d=2e$.
\end{prop}

\begin{proof}
The proof can proceed in the same way in both cases (1) and (2) (by swapping $\Sym^{2}$ and $\wedge^{2}$), so let us consider only the case (1).

We assume that $\rho$ is orthogonal.
Recall that
\[
L(s,\wedge^{2}\rho)
=
\det\bigl(1-q^{-s}\cdot\Frob \,\big\vert\, (\wedge^{2}\rho)^{I_{F}}\bigr)^{-1}.
\]
We note that $(\wedge^{2}\rho)^{I_{F}}\subset (\rho\otimes\rho)^{I_{F}}$.
Let us investigate the unramified characters appearing in $\rho\otimes\rho$.
If we let $\omega$ be such a character of $F^{\times}$, then
\[
\Hom_{W_{F}}(\rho\otimes\omega,\rho)
\cong
\Hom_{W_{F}}(\omega,\rho\otimes\rho)\neq0
\]
by the self-duality of $\rho$.
Thus $\omega$ must satisfy $\omega^{d}=\mathbbm{1}$, or equivalently, $\omega|_{W_{E}}=\mathbbm{1}$ by Lemma \ref{lem:HO-4.1} (3).
Note that such a character $\omega$ occurs only once in $\rho\otimes\rho$ since $\Hom_{W_{F}}(\rho\otimes\omega,\rho)$ is at most one-dimensional by the irreducibility of $\rho$.

We first consider the case where $\sigma$ is self-dual.
Note that then $\sigma$ must be orthogonal since its induction $\rho\cong\Ind_{W_{E}}^{W_{F}}\sigma$ is orthogonal.
If $\omega$ is an unramified character of $W_{F}$ satisfying $\omega|_{W_{E}}=\mathbbm{1}$, then
\[
\Hom_{W_{F}}(\omega,\Ind_{W_{E}}^{W_{F}}(\Sym^{2}\sigma))
\cong\Hom_{W_{E}}(\mathbbm{1},\Sym^{2}\sigma)
\neq0
\]
as $\sigma$ is irreducible and orthogonal.
Since $\Ind_{W_{E}}^{W_{F}}(\Sym^{2}\sigma)$ is contained in $\Sym^{2}\rho$, we get $\Hom_{W_{F}}(\omega,\Sym^{2}\rho)\neq0$ and thus $\Hom_{W_{F}}(\omega,\wedge^{2}\rho)=0$.

We next consider the case where $\sigma$ is not self-dual.
In this case, there is a nontrivial element $\gamma$ of $\Gal(E/F)$ such that $\sigma^\gamma$ is equivalent to the contragredient of $\sigma$ (compare the Mackey decompositions of $\rho$ and $\rho^{\vee}$).
This element $\gamma$ necessarily has order $2$. In particular, this implies that $d$ is even and we put $d=2e$.
Let $E'$ be the subextension of $E/F$ fixed by $\gamma$.
We put $\sigma':=\Ind_{W_{E}}^{W_{E'}}\sigma$.
Then $\sigma'$ is irreducible and self-dual since $\sigma^{\gamma}$ is the contragredient of $\sigma$ and not equivalent to $\sigma$.
Note that this implies that $\sigma'$ is orthogonal since so is its induction $\rho$ to $W_{F}$.
Furthermore, by noting that $\sigma|_{I_{F}}$ is $\Gal(E/F)$-regular, $\sigma'|_{I_{F}}$ is $\Gal(E'/F)$-regular.

Now let $\omega$ be an unramified character of $W_{F}$ which occurs in $\rho\otimes\rho$.
If $\omega$ satisfies $\omega^{e}=\mathbbm{1}$ (or, equivalently, $\omega|_{W_{E'}}=\mathbbm{1}$), then the same argument as in the case where $\sigma$ is self-dual shows that $\omega$ must occur in $\Ind_{W_{E'}}^{W_{F}}(\Sym^{2}\sigma')$.
Since $\Ind_{W_{E'}}^{W_{F}}(\Sym^{2}\sigma')$ is contained in $\Sym^{2}\rho$, we get $\Hom_{W_{F}}(\omega,\Sym^{2}\rho)\neq0$ and $\Hom_{W_{F}}(\omega,\wedge^{2}\rho)=0$.

Let us consider the case where $\omega^{e}\neq\mathbbm{1}$.
In this case, the restriction $\omega|_{W_{E'}}=:\omega'$ is necessarily the nontrivial quadratic character of $W_{E'}$ with kernel $W_{E}$.
Let us write $V$ for the representation space of $\sigma$, and $\gamma V$ for that of $\sigma^\gamma$; more precisely, we choose a lift $\gamma$ in $W_{E'}$ of the nontrivial element of $\Gal(E/E')$, and the action of $w\in W_E$ on $\gamma v$ in $\gamma V$ gives $w\gamma v=\gamma(\gamma^{-1}w\gamma)v$.
That gives the action of $W_E$ on the space $X=V\oplus\gamma V$ of $\sigma'=\Ind_{W_{E}}^{W_{E'}}\sigma$, and $\gamma$ acts on that space via $\gamma(v,\gamma w)=(\gamma^{2}w,\gamma v)$.
The action of $W_{E'}$ on $X\otimes X$ is the direct sum of
\[
(V\otimes V)\oplus(\gamma V\otimes\gamma V)= \Ind_{W_{E}}^{W_{E'}}(V\otimes V)
\]
and
\[(
V\otimes\gamma V)\oplus(\gamma V\otimes V)=\Ind_{W_{E}}^{W_{E'}}(V\otimes\gamma V).
\]
The first summand cannot contain $\omega'$ since $\sigma$ is not self-dual.
We investigate the second summand $(V\otimes\gamma V)\oplus(\gamma V\otimes V)$.
On the space $(V\otimes\gamma V)\oplus(\gamma V\otimes V)$ we have the involution sending $(v_{1}\otimes\gamma v_{2},\gamma v'_{1}\otimes v'_{2})$ to $(v'_{2}\otimes\gamma v'_{1}, \gamma v_{2}\otimes v_{1})$. The decomposition $X\otimes X=\Sym^{2}(X)\oplus \wedge^{2}(X)$ induces a decomposition of $(V\otimes\gamma V)\oplus(\gamma V\otimes V)$: under this decomposition, the $\Sym^{2}$-part is the subspace where that involution acts trivially, and the $\wedge^{2}$-part is the subspace where it acts as $-1$.
From this, it can be easily checked that the $\wedge^{2}$-part of $(V\otimes\gamma V)\oplus(\gamma V\otimes V)$ is, as a representation of $W_{E'}$, the twist of the $\Sym^{2}$-part by the character
$\omega'$.
Indeed, as representations of $W_{E}$, both parts are isomorphic to $V \otimes \gamma V$ (by projecting onto that component).  However, in the symmetric part, $\gamma$ acts by sending $v \otimes \gamma w$ to $\gamma^2 w \otimes \gamma v$, whereas in the alternating part it acts further multiplying by $-1$. The assertion follows.
Since $\sigma'$ is orthogonal, the trivial character occurs in the $\Sym^{2}$-part.
Thus the character $\omega'$ must occur in the $\wedge^{2}$-part.

In summary, we see that an unramified character $\omega$ of $W_{F}$ is contained in $\wedge^{2}\rho$ if and only if $\omega|_{W_{E'}}\neq\mathbbm{1}$.
By noting that $\Ind_{W_{E'}}^{W_{F}}\omega'$ is isomorphic to the direct sum of all such characters of $W_{F}$, we conclude that
\[
L(s,\wedge^{2}\rho)
=L(s,\Ind_{W_{E'}}^{W_{F}}\omega')
=L(s,\omega')
=(1+q^{-es})^{-1}.\qedhere
\]
\end{proof}

Recall that, for any finite-dimensional irreducible smooth representation $\rho$ of $W_{F}$, we can consider its \textit{Swan exponent} $\Swan(\rho)$.
We refer the reader to \cite[Section 2]{GR10} for the definition of the Swan exponent; note that $\Swan(\rho)$ is denoted by $b(\rho)$ in \cite{GR10}.
We also need the following:

\begin{lem}\label{lem:tame-rep}
Let $\rho$ be a self-dual finite-dimensional irreducible smooth representation of $W_F$ which is not a character and satisfies $\Swan(\rho)=0$. Denote the dimension of $\rho$ by $d$. Then $d$ is even and $\rho\cong\Ind_{W_{E}}^{W_{F}}\chi$.
Here,
\begin{itemize}
\item
$E$ is the degree $d$ unramified extension of $F$, and
\item
$\chi$ is a $\Gal(E/F)$-regular character of $E^{\times}$ satisfying $\chi^{2}\neq\mathbbm{1}$ and
\[
\chi|_{E^{\prime\times}}
=
\begin{cases}
\mathbbm{1} & \text{if $\rho$ is orthogonal,}\\
\omega' & \text{if $\rho$ is symplectic,}
\end{cases}
\]
where $E'$ is the subextension of $E/F$ such that $[E:E']=2$ and $\omega'$ is the nontrivial quadratic unramified character of $E'^{\times}$.
\end{itemize}
\end{lem}

\begin{proof}
Since $\Swan(\rho)=0$, $\rho$ is trivial on the wild inertia subgroup $P_{F}$.
Thus, by noting that $I_{F}/P_{F}$ is abelian, the restriction $\rho|_{I_{F}}$ of $\rho$ to the inertia subgroup $I_{F}$ decomposes into the direct sum of $d$ characters.
Hence, by Lemma \ref{lem:HO-4.1}, we obtain $\rho\cong\Ind_{W_{E}}^{W_{F}}\chi$, where $E$ is the degree $d$ unramified extension of $F$ and $\chi$ is a $\Gal(E/F)$-regular character of $E^{\times}$ (recall that even its restriction to $\mcO_{E}^{\times}$ is $\Gal(E/F)$-regular).
Note that $\chi$ is tamely ramified by the assumption that $\Swan(\rho)=0$.

Since $\rho$ is self-dual, $\Hom_{W_{F}}(\rho,\rho^{\vee})\neq0$.
Hence, by the Frobenius reciprocity and the Mackey decomposition, $\chi^{-1}$ has to equal $\chi^{g}$ for some $g$ in $\Gal(E/F)$.
Then $g^{2}$ fixes $\chi$, so either $g$ is trivial or $g$ has order $2$ by the $\Gal(E/F)$-regularity of $\chi$.
If $g$ is trivial, then $\chi^{2}=\mathbbm{1}$.
Since $\chi$ is tamely ramified, this implies that $\chi|_{\mcO_{E}^{\times}}$ is fixed by $\Gal(E/F)$, which contradicts the $\Gal(E/F)$-regularity of $\chi|_{\mcO_{E}^{\times}}$.
Thus $g$ must have order $2$, hence $d$ is even.
Note that then $\chi^{2}\neq\mathbbm{1}$.

Let $E'$ be the fixed field of $g$, so that $E/E'$ is a quadratic extension.
Then $\Ind_{W_{E}}^{W_{E'}}\chi$ is self-dual.
Moreover, since $\rho\cong\Ind_{W_{E'}}^{W_{F}}(\Ind_{W_{E}}^{W_{E'}}\chi)$, $\Ind_{W_{E}}^{W_{E'}}\chi$ is irreducible.
In particular, $\Ind_{W_{E}}^{W_{E'}}\chi$ is either symplectic or orthogonal.
Note that $\wedge^{2}(\Ind_{W_{E}}^{W_{E'}}\chi)\cong\det(\Ind_{W_{E}}^{W_{E'}}\chi)$ as $\Ind_{W_{E}}^{W_{E'}}\chi$ is two dimensional.
Hence, $\Ind_{W_{E}}^{W_{E'}}\chi$ is symplectic if and only if $\det(\Ind_{W_{E}}^{W_{E'}}\chi)$ is trivial.
By \cite[29.2, Proposition]{BH06},
\[
\det(\Ind_{W_{E}}^{W_{E'}}\chi)
\cong
\det(\Ind_{W_{E}}^{W_{E'}}\mathbbm{1})\otimes(\chi|_{E^{\prime\times}}).
\]
The character $\det(\Ind_{W_{E}}^{W_{E'}}\mathbbm{1})$ equals the nontrivial unramified quadratic character of $W_{E'}$.
Indeed, $\det(\Ind_{W_{E}}^{W_{E'}}\mathbbm{1})$ is obviously trivial on $W_{E}$, hence equals either the trivial character or the nontrivial unramified quadratic character of $W_{E'}$.
If we take a realization of the representation space of $\Ind_{W_{E}}^{W_{E'}}\mathbbm{1}$ as in the proof of Prop \ref{prop:par-L}, we can see that the action of any $\gamma\in W_{E'}\smallsetminus W_{E}$ is represented by a matrix $\begin{pmatrix}0&1\\1&0\end{pmatrix}$, whose determinant equals $-1$.
Hence $\det(\Ind_{W_{E}}^{W_{E'}}\mathbbm{1})$ is not the trivial character.
On the other hand, since $\chi\cdot\chi^{g}=\mathbbm{1}$ (i.e., $\chi$ is trivial on norms from $E^{\times}$ to $E^{\prime\times}$), the restriction $\chi|_{E^{\prime\times}}$ is either trivial or the nontrivial unramified quadratic character.
Therefore, $\det(\Ind_{W_{E}}^{W_{E'}}\chi)$ is trivial if and only if $\chi|_{E^{\prime\times}}$ is the nontrivial unramified quadratic character.
This completes the proof.
\end{proof}

\section{Twisted gamma factor for simple supercuspidal representations of \texorpdfstring{$\SO_{2n}$}{SO(2n)}}\label{sec:AK}

This section concerns the case where $\bfG:=\SO_{2n}$, $n>1$.
Let $\pi$ be a simple supercuspidal representation of $G=\bfG(F)$, and $\tau$ a tamely ramified character of $F^{\times}$.
We compute the Rankin--Selberg $\gamma$-factor $\gamma(s,\pi\times\tau, \psi)$ defined in \cite{Kap13-MM, Kap15,AK21}, thus completing \cite{AK21} which treated the case where $\tau$ is quadratic.
Then we deduce consequences for the parameter of $\pi$, as in \cite[Section 5]{AK21} (see \cite{Adr16,AK19} for similar results in the $\SO_{2n+1}$ and $\Sp_{2n}$ cases).

\subsection{Simple supercuspidal representations of \texorpdfstring{$\SO_{2n}$}{SO(2n)}}\label{subsec:AK-ssc}

Following \cite{AK21}, we let $\pi$ be the simple supercuspidal representation of $G$ given as follows.
Fix $\alpha \in \padicfield^{\times}$ and a uniformizer $\varpi'$ for the ring of integers $\mcO$ of $F$.
(Here, we use the symbol $\varpi'$ rather than $\varpi$ because we want to reserve the symbol $\varpi$ for ``the'' uniformizer fixed at the beginning of this paper.) In addition let $\omega$ be a character of $Z_{\SO_{2n}}$ and $\zeta\in\{\pm1\}$.
Let us write $I$ and $I^{+}$ instead of $I_{\SO_{2n}}$ and $I_{\SO_{2n}}^{+}$ for short, respectively.
Similarly, we write $Z=Z_{\SO_{2n}}$.
Define an affine generic character $\chi$ of $I^{+}$ by
\begin{align*}
\chi(y)= \psi(\sum_{i=1}^{n-1}y_{i,i+1} + \alpha y_{n-1,n+1} + \varpi^{\prime-1}y_{2n-1,1}).
\end{align*}
Then the stabilizer of $\chi$ in $I$ is given by $Z\langle g_{\chi}\rangle I^+$, where
\begin{align*}
g_{\chi} =
\left(\begin{smallmatrix}
 & & & & & -\varpi^{\prime-1} \\
& I_{n-2} &  &  & \\
&  &   & \alpha^{-1} & &\\
&  & \alpha &  & \\
&  &  &  & I_{n-2}\\
-\varpi' & & & & &
\end{smallmatrix}\right)\in G.
\end{align*}
Extend $\chi$ to a character of $Z \langle g_{\chi} \rangle I^+$, also denoted by $\chi$,
by defining $\chi(g_{\chi})=\zeta$ ($g_{\chi}$ is of order $2$) and for $z\in Z$, $\chi(z)=\omega(z)$.
We put $\pi := \mathrm{Ind}_{Z \langle g_{\chi} \rangle I^+}^{G} \chi$.

Note that this construction exhausts all (equivalence classes of) simple supercuspidal representations of $G$.
See Section \ref{subsubsec:SO-even} for the comparison of this parametrization of simple supercuspidal representations with the one given in Section \ref{subsec:SSC-even}.

\subsection{The twisted \texorpdfstring{$\gamma$}{gamma}-factors}\label{twisted gamma factors}\label{subsec:AK-twisted-gamma}
Let $\tau$ be any quasi-character of $F^{\times}$.
We recall the definition of the $\gamma$-factor $\gamma(s,\pi\times\tau,\psi)$ of \cite{Kap13-MM} via the theory of Rankin--Selberg integrals, with the minor change in conventions introduced in \cite{AK21} (see Remark \ref{remark difference Kap15 vs AK21} below).
Fix Haar measures $dx$ on $F$ and $d^{\times}x$ on $F^{\times}$ by requiring $\int_{\mcO}\,dx=q^{1/2}$ and
$d^{\times}x=\frac{q^{1/2}}{q-1}|x|^{-1}dx$. For a measurable subset $X\subset F$ put $\vol(X)=\int_X\,dx$ and similarly
for $X\subset F^{\times}$, $\vol^{\times}(X)=\int_X\,d^{\times}x$. In order to define the integral we introduce an auxiliary classical group.
Let $\Param\in F^{\times}$ be a given element\footnote{The parameter $\Param$ was denoted by $\gamma$ in \cite{AK21}.}. Define
\begin{align*}
\bfH=\{g\in\SL_{3} \mid {}^tgJ_{3,\Param}g=J_{3,\Param}\},\quad J_{3,\Param}=\begin{pmatrix}&&1\\ &\Param/2&\\ 1&&\end{pmatrix}.
\end{align*}

Let $\bfB_{\bfH}=\bfT_{\bfH}\ltimes \bfU_{\bfH}$ denote the Borel subgroup of upper triangular invertible matrices in $\bfH$, and
similarly denote by $\overline{\bfB}_{\bfH}$ the Borel subgroup of lower triangular invertible matrices.
For $s \in \C$, let $V(\tau,s)$ be the space of the representation $\nInd_{B_{H}}^{H} (|\det|^{s-1/2}\tau)$ (normalized parabolic induction).
The vectors in $V(\tau,s)$ are regarded as complex-valued functions $H\times F^{\times}\to\C$ and the $H$-action on $f_s\in V(\tau,s)$ is denoted by $h\cdot f_s$ ($h\in H$).
Using the Iwasawa decomposition, the representations $\nInd_{B_{H}}^{H} (|\det|^{s-1/2}\tau)$ can be realized on the same space $V(\tau)=V(\tau,0)$, then a standard section $f_s\in V(\tau,s)$ is the image of $f\in V(\tau)$. A holomorphic (resp., meromorphic) section is then an element in $\C[q^{\mp s}]\otimes V(\tau)$ (resp., $\C(q^{-s})\otimes V(\tau)$).

We embed $\bfH$ in $\bfG$ as the stabilizer of the vectors \[e_1,\ldots,e_{n-2},\tfrac{1}{4}e_n-\Param e_{n+1},e_{n+3},\ldots,e_{2n}\]
(for $n=2$, the stabilizer of $\tfrac{1}{4}e_n-\Param e_{n+1}$), where $(e_1,\ldots,e_{2n})$ is the standard basis of the column space $F^{2n}$. In coordinates, an element $(h_{ij})_{1\leq i,j\leq 3}\in \bfH$ is mapped to
\begin{align*}
\diag(I_{n-2}, \left(\begin{smallmatrix}1\\
&\tfrac14 & \tfrac14\\
&-\Param & \Param\\&&&1
\end{smallmatrix}\right)\left(
\begin{smallmatrix}
h_{11} &  & h_{12} & h_{13}\\
& 1 & & \\
h_{21} & & h_{22} & h_{23}\\
h_{31} & & h_{32} & h_{33}
\end{smallmatrix}\right)\left(\begin{smallmatrix}1\\
&2 & -\tfrac12\Param^{-1}\\
&2 & \tfrac12\Param^{-1}\\&&&1
\end{smallmatrix}\right), I_{n-2})\in \bfG
\end{align*}
(cf.\ \cite{Kap15}, see \cite{AK21}).

Also define
\begin{align*}
R^{n,1} =\left\{\left(\begin{smallmatrix}
1 & & & & \\
r & I_{n-2} & & & \\
& & I_2 & &\\
& & & I_{n-2} &\\
& & & r' & 1
\end{smallmatrix}\right) \in \bfG\right\},\qquad
w^{n,1}=\left(\begin{smallmatrix}
 &1 & & & \\
 I_{n-2} & & & \\
& & I_2 & &\\
& & & & I_{n-2} \\
& & & 1 &
\end{smallmatrix}\right) \in \bfG.
\end{align*}

Recall that $n>1$. Let $\bfU$ be the unipotent radical of the upper triangular Borel subgroup $\bfB$ of $\bfG$ (see Section \ref{subsec:LLC}), and let $\bfT<\bfB$ be the diagonal torus, then $\bfB=\bfT\ltimes\bfU$.
Define a generic character $\psi_U$ of $U$ by
\begin{align}\label{def:character of U for integral}
\psi_U(u)=\psi (\sum_{i=1}^{n-2} u_{i,i+1} + \tfrac{1}{4} u_{n-1,n} - \Param u_{n-1,n+1}).
\end{align}
(This generic character is different from $\lambda$ fixed in Section \ref{subsec:LLC}.)
Any generic character of $U$ is of the form \eqref{def:character of U for integral} up to conjugation by an element of $T$, since $\Param$ is arbitrary.

Now we can define the Rankin--Selberg integral.
Let $\pi$ be an irreducible $\psi_U^{-1}$-generic representation of $G$, and denote the corresponding Whittaker model of $\pi$
by $\mathcal{W}(\pi, \psi_U^{-1})$. For any $W\in\mathcal{W}(\pi, \psi_U^{-1})$ and a holomorphic section $f_s$, the integral is defined for $\Re(s)\gg0$ by
\begin{align}\label{zeta}
\Psi(W, f_s)=\int_{U_{H}\backslash H}\int_{R^{n,1}}W(rw^{n,1} h) f_s(h,1)  \,dr\,dh.
\end{align}
This integral extends to a meromorphic function in $\C(q^{-s})$.

In order to define the $\gamma$-factor we consider the intertwining operator $M(\tau,s):V(\tau,s)\rightarrow V(\tau^{-1},1-s)$
defined in $\Re(s)\gg0$ by
\begin{align*}
M(\tau,s)f_s(h,a) = \int_{U_H} f_s(w_1 u h, -a^{-1}) d u,\qquad w_1=\left(\begin{smallmatrix}
&  & 1 \\& -1 &  \\1 &  & \end{smallmatrix}\right),
\end{align*}
and in general by meromorphic continuation. Here the measure is defined by identifying the coordinate $u_{1,2}$ of $u\in U_H$ with $F$ and then $du$ is our fixed measure $dx$.
The intertwining operator is normalized by the Langlands--Shahidi local coefficient which we denote by $C(s,\tau,\psi)$.
This factor was computed in \cite{AK21}, but the fourth named author found a typo in the computation, which we hereby correct:
the factor $\tau(4/\Param)|4/\Param|^s$ in equation~\cite[(6.4)]{AK21} should be replaced by its inverse (the effect of the change $z\mapsto\tfrac{\Param}4z$ was computed incorrectly). Consequently the formula reads
\begin{align}\label{eq:C factor}
C(s,\tau,\psi)&=\tau(\Param)|\Param|^{s-1}\gamma(2s-1,\tau^2,\psi).
\end{align}
Here $\gamma(s,\tau^2,\psi)$ is Tate's $\gamma$-factor \cite{Tat67}.
Then we define
\[
M^*(\tau,s)=C(s,\tau,\psi)M(\tau,s)
\quad\text{and}\quad
\Psi^*(W, f_s)=\Psi(W, M^*(\tau,s)f_s)
\]
(which is absolutely convergent for $\Re(s)\ll0$).
The $\gamma$-factor is now defined by the functional equation
\begin{align}\label{gamma def}
\gamma_{\Param}(s,\pi\times\tau,\psi)\Psi(W, f_s)=\pi(-I_{2n})\tau(-1)^{n}\tau(4\Param^{-2})|\Param/2|^{-2s+1}\Psi^*(W, f_s).
\end{align}
When $\Param$ is fixed, we denote 
\begin{align*}
\gamma(s,\pi\times\tau,\psi)=\gamma_{\Param}(s,\pi\times\tau,\psi).
\end{align*}
\begin{rem}
The correction to $C(s,\tau,\psi)$ does not affect \cite[Theorem~1.1]{AK21} because it does not affect the poles of $\gamma(s,\pi\times\tau,\psi)$, and
it also has no impact on \cite[Corollary~4.7]{AK21} and \cite[Proposition~5.1]{AK21} because $\tau$ was taken to be quadratic and
$|\Param/4|=1$ (see \cite[Section 3]{AK21} and recall that $\Param$ was denoted by $\gamma$ in \textit{ibid.}).
\end{rem}
\begin{rem}\label{remark difference Kap15 vs AK21}
The Rankin--Selberg integral is defined in \cite{Kap13-MM,Kap15} for a pair of generic representations of $\SO_{2n}\times\GL_k$ for any $n$ and $k$. However, the definition in \textit{ibid.}~ depends on the character $\psi_U$ in the sense that when $\SO_{2n}$ is split, it is assumed that $2\Param$ is a square in $F^{\times}$. The definition in \cite{AK21} relaxes the construction by removing this assumption, but only for $n>k$ and details are given only for $k=1$. While it is not clear that for the case $n\leq k$ the definition can also be given independently of the square class of $\Param$, this is in fact not needed, not here and not in \cite{AK21}. We only need to know that the $\gamma$-factor defined above coincides with Shahidi's $\gamma$-factor (defined in \cite{Sha90}) for generic supercuspidal representations. This is guaranteed by Corollary \ref{cor:gamma factor OK} below.
\end{rem}

For the following theorem and corollary we replace $F$ by an arbitrary local field $F_0$, $p$-adic or Archimedean (of characteristic $0$).
For details specific to the Archimedean integrals see \cite{Kap15}.
Let $\mathrm{Rep}(\GL_n(F_0))$ denote the set of isomorphism classes of (complex, smooth) representations $\sigma$ of $\GL_n(F_0)$ with the following properties:
\begin{itemize}
\item $\sigma$ has a unique Whittaker model.
\item If $F_0$ is $p$-adic: $\sigma$ is either irreducible, or reducible but then $\sigma$ is a quotient of a representation parabolically induced from an irreducible supercuspidal representation of a Levi subgroup (this includes the case that $\sigma$ is a reducible principal series).
\item If $F_0$ is Archimedean: $\sigma$ is either irreducible, admissible, Fr\'{e}chet, of moderate growth, or $\sigma$ is a principal series representation.
\end{itemize}

\begin{thm}\label{theorem:properties of the gamma factor}
Let $\pi$ be an irreducible admissible (when $F_0$ is Archimedean: Fr\'{e}chet, of moderate growth) $\psi_U^{-1}$-generic representation of $\SO_{2n}(F_0)$, $n>1$, and let $\tau$ be a quasi-character of $F_0^{\times}$. The $\gamma$-factor $\gamma(s,\pi\times\tau,\psi)$ satisfies the following properties.
\begin{enumerate}[label=(\alph*)]
\item 
(unramified twisting) For $s_0\in\C$, $\gamma(s,\pi\times\tau|\cdot|^{s_0},\psi)=\gamma(s+s_0, \pi\times\tau,\psi)$.
\item 
(multiplicativity) Let $\sigma\in \mathrm{Rep}(\GL_n(F_0))$ and assume that $\pi$ is the irreducible $\psi_U^{-1}$-generic quotient of a representation parabolically induced from $\sigma$. Then,
    \[\gamma(s,\pi\times\tau,\psi)=\gamma(s,\sigma\times\tau,\psi)\gamma(s,\sigma^{\vee}\times\tau,\psi).\]
    Here the $\gamma$-factors on the right-hand side are the Rankin--Selberg $\GL_n\times\GL_1$ $\gamma$-factors of \cite{JPSS83,JS3,Jac5}.
\item \label{enum:unramified factors} (unramified factors) When all data are unramified, in particular $F_0$ is $p$-adic and $|2|=|\Param|=1$ in $F_0$,
    \[ \gamma(s,\pi\times\tau,\psi)=\frac{L(1-s,\pi\times\tau^{-1})}{L(s,\pi\times\tau)}.    \]
\item
(dependence on $\psi$) For any $b\in F_0^{\times}$, let $\psi_b$ be the character of $F_0$ given by $\psi_b(x)=\psi(bx)$. Then,
    \[
    \gamma(s,\pi\times\tau,\psi_b)=\tau(b)^{2n}|b|^{2n(s-1/2)}\gamma(s,\pi\times\tau,\psi). \]
\item \label{enum:dependence on psi}
(smooth with respect to $\Param$) Assume that $F_0$ is $p$-adic. Let $y\in F_0^{\times}$ be such that $|y\Param^{-1}-1|$ is sufficiently small with respect to $\tau$. Then,
    \[    \gamma_{\Param}(s,\pi\times\tau,\psi)=\gamma_{y}(s,\pi\times\tau,\psi). \]
\item 
(Archimedean property) When $F_0$ is Archimedean, $\gamma(s,\pi\times\tau,\psi)$ is the Artin $\gamma$-factor $\gamma^{\mathrm{Artin}}(s,\pi\times\tau,\psi)$ under the local Langlands correspondence (\cite{Lan89}).
\item 
(global property) Let $\mff$ be a number field with adeles ring $\bbA_{\mff}$, and $\eta$ be a nontrivial character of $\mff\backslash\bbA_{\mff}$. Let $\Pi$ be a globally generic cuspidal automorphic representation of $\bfG(\bbA_{\mff})$. Then, $\Pi$ is
(necessarily) globally generic with respect to a character of $U(\bbA_{\mff})$ trivial on $U(\mff)$. We can assume that this character is given by \eqref{def:character of U for integral} where $(\psi,\Param)$ is replaced with $(\eta,y)$, $y\in \mff^{\times}$. Let $\Xi$ be an automorphic character of $\bbA_{\mff}^{\times}$. Let $S$ be a finite set of places of $\mff$ such that for any $v\notin S$, all data are unramified. Then
    \begin{align*}
    L_S(s,\Pi\times\Xi)=\prod_{v\in S}\gamma_{y_v}(s,\Pi_v\times\Xi_v,\eta_v)L_S(1-s,\Pi\times\Xi^{-1}).
    \end{align*}
    Here $L_S(s,\Pi\times\Xi)$ is the partial $L$-function with respect to $S$.
\end{enumerate}
\end{thm}
\begin{proof}
For the proofs of all of these properties except \ref{enum:dependence on psi} refer to \cite{Kap15} and the references therein. In particular, the multiplicativity property is proved in \cite[Section 5.4]{Kap13-MM} and the proof is adapted to the conventions of \cite{AK21} (i.e., no assumption on $\Param$) in \cite[Section 6]{AK21}. Note that there is no need to assume that $\sigma$ is irreducible (cf.\ \cite[Theorem 1]{Kap15}), but if $F_0$ is Archimedean and $\sigma$ is reducible, then by our definition $\sigma$ is a principal series. This is the generality for which the $\GL_n\times\GL_1$ $\gamma$-factor was treated in \cite{Jac5}.

The important observation is the following. Only one part of the theory of \cite{Kap13-MM,Kap15} of the Rankin--Selberg integrals for representations of $\SO_{2n}\times\GL_m$ with $n>m$ depends on the theory for $n\leq m$. This part is the multiplicativity property when $\pi$ is a quotient of a representation parabolically induced from a representation of $\GL_r\times\SO_{2(n-r)}$ with $0<r<n$ (see \cite[Section 5.1]{Kap15}). Other than that, the theory for $n>m$ can be developed independently of the theory for $n\leq m$. (Here $m=1$.)

Regarding the smoothness property denote the character \eqref{def:character of U for integral} by $\psi_{U,\Param}$, to explicate the dependence on $\Param$. Let $t=\diag(dI_n,d^{-1}I_n)\in T$ with $d\in F_0^{\times}$. By the assumption on $y$, we can take $d$ such that $\Param d^{2}=y$. For
$W\in\mathcal{W}(\pi, \psi_{U,\Param}^{-1})$, the function $[\ell(t)W](g)=W(tg)$ on $G$ belongs to $\mathcal{W}(\pi, \psi_{U,y}^{-1})$.
Let $\Psi_{y}(\ell(t)W,f_s)$ denote the Rankin--Selberg integral constructed using $y$ instead of $\Param$, and let $t\cdot W$ denote the right-translation of $W$ by $t$. We claim that
\[\Psi_{y}^*(\ell(t)W,f_s)/\Psi_{y}(\ell(t)W,f_s)=\Psi^*(t\cdot W,f_s)/\Psi(t\cdot W,f_s).\]
To see this observe the following:
\begin{itemize}
\item $t$ commutes with $w^{n,1}$, normalizes $R^{n,1}$, and $\psi_{U,y}(t^{-1}ut)=\psi_{U,\Param}(u)$.
\item The conjugation of $R^{n,1}$ by $t$ changes the measure, but this change is cancelled when we consider the quotient of integrals.
\item We may write the $dh$-integral of $\Psi_y(\ell(t)W,f_s)$ over $\overline{B}_{H}$. For $b\in \overline{B}_{H}$, whose image in $G$ is defined using the embedding with respect to $y$, $tbt^{-1}$ is the image of $b$ in $G$ under the embedding of $\bfH$ in $\bfG$ defined with $\Param$ (direct verification).
\end{itemize}
Looking at the definition of the $\gamma$-factor \eqref{gamma def}, it remains to note that under the given condition on $y$, we can assume that $|y|=|\Param|$ and $\tau(y)=\tau(\Param)$. This completes the proof of \ref{enum:dependence on psi}.
\end{proof}
\begin{rem}
The multiplicativity property in Theorem \ref{theorem:properties of the gamma factor} is only partial, cf.\ \cite[Theorem 1]{Kap15}, but suffices for what we require.
\end{rem}
\begin{cor}\label{cor:gamma factor OK}
Let $F_0$ be a local $p$-adic field. Let $\pi$ be an irreducible $\psi_U^{-1}$-generic supercuspidal representation of $\SO_{2n}(F_0)$, $n>1$, and $\tau$ be a quasi-character of $F_0^{\times}$. Then
$\gamma(s,\pi\times\tau,\psi)$ and Shahidi's $\gamma$-factor $\gamma^{\mathrm{Sh}}(s,\pi\times\tau,\psi)$ of \cite{Sha90} coincide.
\end{cor}
\begin{proof}
Since $\pi$ is irreducible supercuspidal, we can argue using globalization. Briefly, we can choose the following data (see \cite[Section 5]{Sha90} for more details):
\begin{itemize}
  \item A number field $\mff$ such that for some place $v$ of $\mff$, $\mff_v=F_0$.
  \item A nontrivial character $\eta$ of $\mff\backslash\bbA_{\mff}$ and an element $y\in \mff^{\times}$, such that in $F_0$, $|y_v\Param^{-1}-1|$ is sufficiently small with respect to $\tau$. Denote by $\eta_U$ the character of $U(\bbA_{\mff})$ that is trivial on $U(\mff)$, which is defined by \eqref{def:character of U for integral} with $(\eta,y)$ instead of $(\psi,\Param)$.
  \item A globally $\eta_U^{-1}$-generic cuspidal automorphic representation $\Pi$ of $\bfG(\bbA_{\mff})$ with $\Pi_v=\pi$, and such that for all places $v'\ne v$, $\Pi_{v'}$ is a quotient of a principal series representation.
  \item An automorphic character $\Xi$ of $\bbA_{\mff}^{\times}$ such that $\Xi_v=\tau$.
\end{itemize}
Note that $\bfG((\bbA_{\mff})_{v'})$ is split over $(\bbA_{\mff})_{v'}$ for all places $v'$, because we defined the orthogonal group using $J_{2n}'$ (see Section \ref{subsec:SSC-even}), and $\bfH((\bbA_{\mff})_{v'})$ is also split.

By transitivity of induction we can assume that for each $v'\ne v$, $\Pi_{v'}$ is a quotient of some $\sigma\in\mathrm{Rep}(\GL_n((\bbA_{\mff})_{v'}))$.
Thus by the multiplicativity and the Archimedean properties of Theorem \ref{theorem:properties of the gamma factor}, and the similar properties of Shahidi's $\gamma$-factor, for all $v'\ne v$,
\[\gamma_{y_{v'}}(s,\Pi_{v'}\times\Xi_{v'},\eta_{v'})=\gamma^{\mathrm{Sh}}(s,\Pi_{v'}\times\Xi_{v'},\eta_{v'}).\]
Note that the right-hand side does depend on $y_{v'}$ through the Whittaker datum of $\Pi_{v'}$. 
Now, the unramified property, the dependence on $\psi$ property, and the global property imply that
\[\gamma_{y_{v}}(s,\pi\times\tau,\psi)=\gamma^{\mathrm{Sh}}(s,\pi\times\tau,\psi).\]
Then, by the smoothness property \ref{enum:dependence on psi} of $\gamma(s,\pi\times\tau,\psi)$ and the analogous property of Shahidi's $\gamma$-factor (which is straightforward to check), we deduce that the same equality holds with $\Param$ instead of $y_{v}$.
\end{proof}

\subsection{The computation of \texorpdfstring{$\gamma(s,\pi\times\tau,\psi)$}{the gamma factor of pi x tau}}\label{subsec:AK-easyside}
Until the end of this section, we put $\pi = \mathrm{Ind}_{Z \langle g_{\chi} \rangle I^+}^{G} \chi$ and $\tau$ is a tamely ramified character of $F^{\times}$.
Recall that
the only difference here compared to \cite{AK21} is that in \textit{ibid.}\ we assume $\tau$ is also quadratic. Because of this,
most of the computations from \textit{ibid.}\ extend trivially to the case here; for the remaining arguments we provide full details.

Recall that $\alpha\in k^{\times}$ was taken in Section \ref{subsec:AK-ssc} in as a parameter of the simple supercuspidal representation $\pi$.
We define
\begin{align*}
\Param=-4\alpha.
\end{align*}
The parameter $\Param$ is now used for the definition of the $\gamma$-factor as in Section \ref{twisted gamma factors}.
We define a generic character $\psi_{\alpha}$ of $U$ by
\begin{align*}
\psi_{\alpha}(u) = \psi(\sum_{i=1}^{n-2}u_{i,i+1} + u_{n-1,n} + \alpha u_{n-1,n+1}).
\end{align*}
In particular $\psi_{1}=\lambda$ of Section \ref{subsec:LLC}.
\begin{lem}\label{lem:generic}
The simple supercuspidal representation $\pi$ is generic with respect to $\psi_{\alpha}$ and also $\psi_U^{-1}$.
\end{lem}

\begin{proof}
Recall that $\pi$ is said to be generic with respect to $\psi_{\alpha}$ of $U$ if $\Hom_{U}(\pi,\psi_{\alpha})\neq0$.
By the Frobenius reciprocities for the compact and smooth inductions,
\begin{align*}
\Hom_{U}(\pi,\psi_{\alpha})
\cong
\Hom_{G}(\pi, \Ind_{U}^{G}\psi_{\alpha})
\cong
\Hom_{Z \langle g_{\chi} \rangle I^{+}}(\chi, \Ind_{U}^{G}\psi_{\alpha}).
\end{align*}
Thus it suffices to construct a nonzero homomorphism from $\chi$ to $\Ind_{U}^{G}\psi_{\alpha}$ which is $Z \langle g_{\chi} \rangle I^{+}$-equivariant.
Since the character $\psi_{\alpha}$ of $U$ coincides with $\chi$ on $U\cap Z \langle g_{\chi} \rangle I^{+}$, we can define a non-zero element $W$ of $\Ind_{U}^{G}\psi_{\alpha}$ by
\[
W(g)=
\begin{cases}
 \psi_{\alpha}(u)\chi(x) & \text{if $g=ux$ for $u\in U, x\in Z\langle g_{\chi} \rangle I^{+}$,}\\
 0 & \text{otherwise}.
\end{cases}
\]
We define a $\C$-linear map $f\colon\chi\rightarrow\Ind_{U}^{G}\psi_{\alpha}$ by $f(1):=W$.
Then $f$ is $Z\langle g_{\chi} \rangle I^{+}$-equivariant.
Hence $\pi$ is $\psi_{\alpha}$-generic.

We note that $\psi_{\alpha}$ is rationally conjugate to $\psi_U^{-1}$.
Indeed, for example, by putting
\[
t:=\diag(\underbrace{-1,1,-1,\ldots,(-1)^{n-1}}_{n-1},(-1)^{n}4, (-1)^{n}4^{-1}, \underbrace{(-1)^{n-1},\ldots,-1,1,-1}_{n-1}),
\]
we get
\begin{align*}
\psi_{\alpha}(tut^{-1})
&=\psi(-u_{12}-\cdots -u_{n-2,n-1}-4^{-1}u_{n-1,n}-4\alpha u_{n-1,n+1})\\
&=\psi^{-1}(u_{12}+\cdots+ u_{n-2,n-1}+4^{-1}u_{n-1,n}-\Param u_{n-1,n+1})\\
&=\psi_U^{-1}(u)
\end{align*}
for any $u=(u_{ij})\in U$ (recall that $\Param=-4\alpha$).
Hence the $\psi_{\alpha}$-genericity of $\pi$ is equivalent to the $\psi_U^{-1}$-genericity of $\pi$.
\end{proof}

Put $\iota=\diag(I_{n-1},1/4,4,I_{n-1})\in T$ and define for $g\in G$, ${}^{\iota}g=\iota^{-1}g\iota$. The representation
$\pi^{\iota}$ acts on the same space as $\pi$, by $\pi^{\iota}(g)=\pi({}^{\iota}g)$. Clearly $\pi$ and $\pi^{\iota}$ are isomorphic, thus
\begin{align*}
\gamma(s,\pi\times\tau,\psi)=\gamma(s,\pi^{\iota}\times\tau,\psi).
\end{align*}

We turn to compute this $\gamma$-factor using a specific choice of data $(W, f_s)$, which is the same data taken for a quadratic $\tau$.

We realize the Whittaker model $\mathcal{W}(\pi^{\iota},\psi_U^{-1})$ via the isomorphism $\mathcal{W}(\pi,\psi_{-\Param/4}^{-1})\to \mathcal{W}(\pi^{\iota},\psi_U^{-1})$ defined by $W\mapsto W^{\iota}$, where
$W^{\iota}(g)=W({}^{\iota}g)$. Thus for the Whittaker function in $\mathcal{W}(\pi^{\iota},\psi_U^{-1})$ we can take $W^{\iota}$ for
$W\in \mathcal{W}(\pi,\psi_{-\Param/4}^{-1})$.

Define $W_0\in \mathcal{W}(\pi,\psi_{-\Param/4}^{-1})$ by
\begin{align*}
W_0(g) =
\begin{cases}
\psi_{-\Param/4}^{-1}(u) \chi(x) & g = u x,\quad u \in U, x\in Z \langle g_{\chi} \rangle I^+ ,\\
0  &  \text{otherwise.}
\end{cases}
\end{align*}
Take $W=(w^{n,1})^{-1}\cdot W_0$, the right-translation of $W_0$ by $(w^{n,1})^{-1}$.
For the section, let $I_{H}^+$ be the pro-unipotent part of the Iwahori subgroup of $H$ corresponding to $B_H$.
Define $f_s$ by
\begin{align*}
 f_s(g,a)= \begin{cases}
|m|^{s} \tau(am)  & g = \diag(m,1,m^{-1})uy,\quad m\in F^{\times}, u\in U_{H}, y\in I_{H}^+,\\
0 &  \text{otherwise.}
\end{cases}
\end{align*}

The computation involves writing the integral $dh$ of \eqref{zeta} over $\overline{B}_{H}$.
Write $b\in\overline{B}_{H}$ in the form
\begin{align*}
b=
\left(\begin{smallmatrix}
    a &  &  \\
     & 1 &  \\
    && a^{-1}
  \end{smallmatrix}\right)\left(\begin{smallmatrix}
    1 &  &  \\
    x & 1 &  \\
    -\tfrac\Param4x^2 & -\tfrac\Param2x& 1
  \end{smallmatrix}\right),\qquad a\in F^{\times}, x\in F.
\end{align*}
Since $H$ is defined with respect to $J_{3,\Param}$, if $a=1$, then $b\in I_H^+$ if and only if $|x|<1$.

The computations related to the inner $dr$-integral of $W^{\iota}$ and $\Psi(W^{\iota},f_s)$ follow just as in \cite{AK21} so that we only cite the statements:

\begin{lem}[{\cite[Lemma 4.1]{AK21}}\footnote{The notation $U$ in \cite[Lemma~4.1]{AK21} means $U_{\SO_{2n}}$.}]\label{support}
Assume $b$ as above.
\begin{enumerate}
\item\label{it:support 1}
If ${}^{\iota}(rw^{n,1}b(w^{n,1})^{-1})\in UZI^+$,
then $a \in 1 + \mathfrak{p}$, $|x|<1$, $r\in\mathfrak{p}^{n-2}$.
\item\label{it:support 2}
If ${}^{\iota}(rw^{n,1}b(w^{n,1})^{-1})\in Ug_{\chi}ZI^+$, then there is some $k\geq0$ such that
$|a|=q^{2k+1}$, $|x|=q^k$ and $\tfrac{\Param}4x^2a^{-1}\in  \varpi' \cdot (1 + \mathfrak{p})$, and also $r\in\mathfrak{p}^{n-2}$.\qed
\end{enumerate}
\end{lem}

\begin{cor}[{\cite[Corollary 4.2]{AK21}}]\label{cor:psi W f}
$\Psi(W^{\iota},f_s)= \vol^{\times}(1 + \mathfrak{p})\vol(\mathfrak{p})^{n-1}$.\qed
\end{cor}

\begin{lem}[{\cite[Lemma 4.5]{AK21}}]\label{W on other coset}
Assume ${}^{\iota}(rw^{n,1}b(w^{n,1})^{-1})\in Ug_{\chi}ZI^+$. Then
\begin{align*}
\int_{R^{n,1}} W_0^{\iota}(r w^{n,1} b (w^{n,1})^{-1}) \,dr = \chi(g_{\chi})\vol(\mathfrak{p})^{n-2}.\qed
\end{align*}
\end{lem}

Next, we compute $M(\tau,s)f_s$ on the support of $W^{\iota}$.
In order to obtain uniform formulas define
\begin{align*}
A_{\tau}=\int_{\ringofintegers^{\times}}\tau^2(o)d^{\times}o=\begin{cases}1&\text{$\tau^2$ is unramified,}\\0&\text{otherwise.}\end{cases}
\end{align*}

\begin{lem}[cf.\ {\cite[Lemma 4.3]{AK21}}]\label{lem:intertwining operator when a is in 1+p}
Assume $a \in 1 + \mathfrak{p}$ and $|x|<1$.
\begin{align*}
M(\tau,s)f_s(b,1)=
|\Param|^s\tau(\Param)\tau(2^{-2})(q-1)|2|^{1-2s}
\frac{\tau^2(\varpi')q^{1/2-2s}}{1-\tau^2(\varpi')q^{1-2s}}A_{\tau}.
\end{align*}
\end{lem}
\begin{proof}
The proof follows exactly as in \textit{loc.\ cit.}\ except the final step. For the convenience of the reader, we
recall the beginning of the argument. To start, we can take $a=1$ because $\tau$ is tamely ramified. Then, for $u\ne I_3$,
\begin{align*}
w_1 u=\left(\begin{smallmatrix}
&  & 1 \\& -1 &  \\1 &  & \end{smallmatrix}\right)\left(\begin{smallmatrix}
1 &  v & -\Param^{-1} v^2\\
 & 1 & -2\Param^{-1} v\\
 &  & 1
\end{smallmatrix}\right)=\left(
\begin{smallmatrix}
1 & \Param v^{-1} & -\Param v^{-2}\\&1&-2v^{-1}\\&&1
\end{smallmatrix}\right)\left(
\begin{smallmatrix}
-\Param v^{-2} \\2v^{-1}&1\\1&v&-\Param^{-1}v^2
\end{smallmatrix}\right).
\end{align*}
Substituting this into the integral defining $M(\tau,s)$ we obtain
\begin{align*}
M(\tau,s) f_s(b, 1)  &= |\Param|^s\tau(\Param)\int_{F^{\times}}
\tau(v^{-2})|v|^{-2s}f_s(
\left(\begin{smallmatrix}
1\\2v^{-1}+x&1\\-\tfrac\Param4(2v^{-1}+x)^2&-\tfrac\Param2(2v^{-1}+x)&1
\end{smallmatrix}\right), 1) \,dv.
\end{align*}
Now we have reached the step where the proof is different. Next, using the fact that $f_s$ is supported in $B_{H}I_{H}^{+}$ and the assumption $|x|<1$, we deduce
\begin{align*}
M(\tau,s)f_s(b,1)&=|\Param|^s\tau(\Param)(q-1)q^{-1/2} \tau(2^{-2}) |2|^{1-2s}\int_{\{v\in F^{\times}:|v|>1\}}\tau(v^{-2})|v|^{1-2s} \,d^{\times}v.
\end{align*}
The last integral is equal to the sum $\sum_{l=1}^{\infty}q^{l(1-2s)}\tau(\varpi^{\prime2n})\int_{\ringofintegers^{\times}}\tau^2(o)d^{\times}o$, which vanishes
when $\tau^2$ is ramified, and otherwise equals ${\tau^2(\varpi')q^{1-2s}}/(1-\tau^2(\varpi')q^{1-2s})$.
\end{proof}

\begin{lem}[cf.\ {\cite[Lemma 4.4]{AK21}}]\label{lem:operatoe on second coset}
Assume $|x|=q^k$ with $k\geq0$. Then
\begin{align*}
&M(\tau,s)f_s(b,1)=|\Param|^s\tau(\Param)|a|^{1-s}\tau^{-1}(a)|2|^{1-2s}\tau(2^{-2})\tau(x^2) q^{2k(s-1)} \vol(\mathfrak{p}).
\end{align*}
\end{lem}
\begin{proof}
As in the proof of Lemma~\ref{lem:intertwining operator when a is in 1+p} but now for any $a$,
\begin{multline*}
M(\tau,s) f_s(b, 1)\\
=|\Param|^s\tau(\Param)|a|^{1-s}\tau^{-1}(a)\int_{F^{\times}}
\tau(v^{-2})|v|^{-2s}f_s(
\left(\begin{smallmatrix}
1\\2v^{-1}+x&1\\-\tfrac\Param4(2v^{-1}+x)^2&-\tfrac\Param2(2v^{-1}+x)&1
\end{smallmatrix}\right), 1) \,dv.
\end{multline*}

We change variables $v\mapsto 2v$. Then, we see that the integrand vanishes unless $v^{-1}+x\in \mathfrak{p}$, equivalently $v^{-1}\in-x(1+ x^{-1}\mathfrak{p})$. Since
$|x|\geq1$, $x^{-1}\mathfrak{p}<\mathfrak{p}$. Now because $\tau$ is tamely ramified,
$\tau(v^{-1})=\tau(-x)$. Moreover $|v|=q^{-k}$ and the $dv$-integral equals $\tau(x)^2 q^{2k(s-1)} \vol(\mathfrak{p})$.
\end{proof}
\begin{cor}[cf.\ {\cite[Corollary 4.6]{AK21}}]\label{cor:psi *}
The following identity holds:
\begin{multline*}
\frac{\Psi(W^{\iota},M(\tau,s)f_s)}{\Psi(W^{\iota},f_s)}\\
=(q-1)|2|^{1-2s} |\Param|^s \left(\frac{\tau(\Param/4)\tau^2(\varpi')q^{1/2-2s}}{1-\tau^2(\varpi')q^{1-2s}}A_{\tau}
+\frac{\chi(g_{\chi})\tau(\varpi') q^{-1/2-s}}{1-q^{-1}}\right).
\end{multline*}
\end{cor}
Note that the statement is the exact analog of \cite[Corollary 4.6]{AK21} except that here $A_{\tau}$ can be $0$ and we have the values of $\tau^2$.
\begin{proof}
The proof follows as in \textit{loc.\ cit.}, by combining the above lemmas in exactly the same way. We present the computation because we have to account for all the factors involving $\tau^2$.

In order to compute the integral we write the $dh$-integral over $\overline{B}_{H}$ using the right invariant Haar measure on $\overline{B}_{H}$ given by
$db=|a|^{-1}\,d^{\times}a\,dx$. Since $W_0$ is supported in $U\langle g_{\chi}\rangle ZI^+$ and by virtue of Lemma~\ref{support},
$\Psi(W^{\iota},M(\tau,s)f_s)$ is the sum of two integrals. The first summand,  corresponding to case \eqref{it:support 1} of the lemma, is
\begin{align*}
&\int_{a\in1+\mathfrak{p}}
\int_{x\in\mathfrak{p}}
\int_{r\in\mathfrak{p}^{n-2}}W_0^{\iota}(rw^{n,1}b(w^{n,1})^{-1})M(\tau,s) f_s(b,1)\,dr\,d^{\times}a\,dx.
\end{align*}
By Lemma~\ref{lem:intertwining operator when a is in 1+p} this integral equals
\begin{align*}
\vol^{\times}(1+\mathfrak{p})  \vol(\mathfrak{p})^{n-1} |\Param|^s\tau(\Param)\tau(2^{-2})(q-1)  |2|^{1-2s}      \frac{\tau^2(\varpi')q^{1/2-2s}}{1-\tau^2(\varpi')q^{1-2s}}A_{\tau}.
\end{align*}

The second summand corresponds to Lemma~\ref{support} \eqref{it:support 2}. It is an integral over $(a,x)$ such that
\begin{align*}
|a|=q^{2k+1},\qquad |x|=q^k, \qquad \tfrac{\Param}4x^2a^{-1}\in  \varpi' \cdot (1 + \mathfrak{p}),\qquad k\geq0.
\end{align*}
(It is the same domain as in the case where $\tau$ is quadratic.)
For each such element, by Lemma~\ref{lem:operatoe on second coset} and Lemma~\ref{W on other coset} the integrand equals
\begin{multline*}
|\Param|^s\tau(\Param)\tau^{-1}(a)|2|^{1-2s} q^{1-s} \chi(g_{\chi})\tau(2^{-2})\tau(x^2)\vol(\mathfrak{p})^{n-1}\\
=|\Param|^s\tau(\varpi')|2|^{1-2s} q^{1-s} \chi(g_{\chi})\vol(\mathfrak{p})^{n-1},
\end{multline*}
where we used $\tau^{-1}(a)=\tau(2^2)\tau(\Param)^{-1}\tau(x^{-2})\tau(\varpi')$ ($\tau$ is tamely ramified). The measure of the above set of $(a,x)$ was computed in
the proof of \cite[Corollary~4.6]{AK21} and equals
$(q-1)q^{-3/2}\vol^{\times}(1+\mathfrak{p})/(1-q^{-1})$.
Therefore
\begin{multline*}
\Psi(W^{\iota},M(\tau,s)f_s)=\vol^{\times}(1+\mathfrak{p})\vol(\mathfrak{p})^{n-1}(q-1)|2|^{1-2s} |\Param|^s
\\
\times\left(\frac{\tau(\Param/4)\tau^2(\varpi')q^{1/2-2s}}{1-\tau^2(\varpi')q^{1-2s}}A_{\tau}
+\frac{\chi(g_{\chi})\tau(\varpi') q^{-1/2-s}}{1-q^{-1}}\right).
\end{multline*}
Now the result follows with the aid of Corollary~\ref{cor:psi W f}.
\end{proof}

Collecting the results above, the immediate analog of \cite[Corollary~4.7]{AK21} reads (note that $\Param/4=-\alpha$):
\begin{cor}\label{corollary:gamma for tamely ramified}
For any tamely ramified character $\tau$ of $F^{\times}$,
\begin{multline*}
\gamma(s,\pi\times\tau,\psi)=
\pi(-I_{2n})\tau(-1)^{n}
\gamma(2s-1,\tau^2,\psi)
\\
\times\left(\frac{(q-1)\tau^2(\varpi')q^{1/2-2s}}{1-\tau^2(\varpi')q^{1-2s}}A_{\tau}
+\chi(g_{\chi})\tau^{-1}(-\alpha)\tau(\varpi') q^{1/2-s}\right).
\end{multline*}
\end{cor}

\subsection{$L$-parameter}\label{subsec:AK-Langlands-parameter}
In this section we discuss the $L$-parameter $\Langlandsparameter$ of $\pi$ in the sense of Arthur (see Section \ref{sec:LLC}).
By Appendices \ref{sec:lift} and \ref{sec:unram}, $\gamma(s,\pi\times\tau,\psi)=\gamma(s,\phi_{\pi}\otimes\tau,\psi)$ for any character $\tau$ of $F^{\times}$, so that the poles of $\gamma(s,\pi\times\tau,\psi)$ tell us when a given tame character $\tau^{-1}$ is an irreducible component of $\phi_{\pi}$. In fact in this case $\tau$ must be quadratic (see Remark~\ref{rem:self dual for poles} below), so that $\tau^{-1}=\tau$.

Note that $\phi_{\pi}$ is the direct sum of irreducible orthogonal representations of $W_{F}\times\SL_{2}(\C)$, which are inequivalent to each other as explained in Section~\ref{subsec:LLC} (see also Section~\ref{subsec:rough-form}). Hence, if $\phi_{\pi}$ contains a character, then it must necessarily be quadratic.
Let us determine those, following \cite{AK21}.

\begin{rem}\label{rem:self dual for poles}
Since $\pi$ is supercuspidal and $n>1$, by
\cite[Corollary~4.2]{Kap13-Compos} the integral $\Psi(W,f_s)$ is holomorphic, and by
\cite[Lemma~2.2.5]{Sha81} the intertwining operator $M(\tau,s)$ is holomorphic when $\Re(s)>1/2$.
Thus, even without appealing to the local Langlands correspondence, we know that if $\gamma(s,\pi\times\tau,\psi)$ has a pole at $s=1$,
then $C(s,\tau,\psi)$ must have a pole there, and by \eqref{eq:C factor} this means $\tau$ is quadratic.
\end{rem}

Assume $\eta$ is a character of $F^{\times}$. Recall that when $\eta$ is unramified, by \cite[23.4 and 23.5]{BH06},
\begin{align}\label{eq:Tate eps factor unramified}
\gamma(s, \eta, \psi) = q^{s-1/2}\eta^{-1}(\varpi')\frac{1-\eta(\varpi')q^{-s}}{1 - \eta(\varpi')^{-1}q^{s-1}}
\end{align}
(note that $\psi$ is of level $1$, i.e., trivial on $\mfp$ but not on $\mcO$), and when $\eta$ is ramified and tamely ramified, by \cite[23.6]{BH06}
\begin{align}\label{eq:Tate eps factor ramified tamely ramified}
\gamma(s,\eta,\psi)=\varepsilon(s,\eta,\psi)=q^{-1/2}G(\eta^{-1},\psi).
\end{align}
Here $G(\eta, \psi)$ is the Gauss sum of $\eta$ with respect to $\psi$ (see \cite[23.6]{BH06}), and for a
tamely ramified $\eta$, $$G(\eta, \psi) = \displaystyle\sum_{x \in \ringofintegers^{\times} / 1 + \mathfrak{p}} \eta(x) \psi(x).$$
(We use the definition of \cite{Oi24}.)

Recall $p$ is the characteristic of the residue field.

Let $\tau_1$ be the unramified character such that $\tau_1(\varpi') = \chi(g_{\chi})$. If $p>2$ we also let
$\tau_2$ be the character which restricts to the unique nontrivial quadratic character of $\ringofintegers^{\times}$ and satisfies
$\tau_2(\varpi') = \chi(g_{\chi}) \tau_2(\Param)$. Both characters are quadratic. They are also tamely ramified: $\tau_1$ trivially, and $\tau_2$ because when
$p>2$, $1+\mathfrak{p}\subset(F^{\times})^2$. Since $\tau_1(\varpi') \ne -\chi(g_{\chi}) \tau_1(\Param)$ and for $p>2$ also
$\tau_2(\varpi') \ne -\chi(g_{\chi})\tau_2(\Param)$, by
\cite[Theorem~1.1]{AK21} $\gamma(s,\pi\times\tau_i,\psi)$ has a pole at $s=1$, for both $i$.

The characters $\tau_1$ and $\tau_2$ determine $1$-dimensional summands $\Langlandsparameternorep_i$ of $\Langlandsparameter$, hence we can write
\begin{align*}
\Langlandsparameter=\Langlandsparameternorep'\oplus\Langlandsparameternorep_1[\oplus\Langlandsparameternorep_2],
\end{align*}
where until the end of this section factors in square brackets appear only when $p>2$.
Let $\Pi'$ be the endoscopic lift of $\phi'$ to a general linear group, i.e., the irreducible tempered representation of $\GL_{2n-2}(F)$ (for $p>2$) or $\GL_{2n-1}(F)$ (when $p=2$) associated with $\Langlandsparameternorep'$ by the local Langlands correspondence for general linear groups.
We let $\gamma(s,\Pi'\times\tau,\psi)$ denote the $\gamma$-factor of \cite{JPSS83} via the theory of Rankin--Selberg integrals.

\begin{thm}\label{thm:AK}
Let $\tau$ be a tamely ramified character of $F^{\times}$. Then
\[
\gamma(s,\Pi'\times\tau,\psi)
=
\pi(-I_{2n})\tau((-1)^{n+1} \alpha^{-1} \varpi')\chi(g_{\chi})[\tau^{-1}(4)\tau_2(-1)\varepsilon(s,\tau_2,\psi)]q^{1/2-s}.
\]
In particular for $p=2$,
$\gamma(s,\Pi'\times\tau,\psi)=\tau(\alpha^{-1}\varpi')\chi(g_{\chi})q^{1/2-s}$ ($\alpha = -\frac{\Param}{4}$ was defined in Section \ref{subsec:AK-easyside}).
\end{thm}
\begin{proof}
Since $\Langlandsparameter=\Langlandsparameternorep'\oplus\Langlandsparameternorep_1[\oplus\Langlandsparameternorep_2]$, the local Langlands correspondence implies
\begin{align*}
\gamma(s,\pi\times\tau,\psi)=\gamma(s, \Pi'\times\tau, \psi)\gamma(s, \tau_1\tau, \psi)[\gamma(s, \tau_2\tau, \psi)].
\end{align*}
Then by Corollary~\ref{corollary:gamma for tamely ramified},
\begin{align}\label{eq:first computation of gamma Pi'}
\gamma(s,\Pi'\times\tau,\psi)=&
\pi(-I_{2n})\tau(-1)^{n}
\gamma(2s-1,\tau^2,\psi)
\\&\times\left(  (q-1)\frac{\tau^2(\varpi')q^{1/2-2s}}{1-\tau^2(\varpi')q^{1-2s}}A_{\tau}\nonumber
+\chi(g_{\chi})\tau^{-1}(-\alpha)\tau(\varpi') q^{1/2-s}\right)
\\&\times\gamma(s, \tau\tau_1, \psi)^{-1}[\gamma(s, \tau\tau_2, \psi)^{-1}].\nonumber
\end{align}

For the computation we treat $3$ cases.
\begin{itemize}[leftmargin=*]
\item The character $\tau^2$ is ramified. Then $A_{\tau}=0$. Since $\tau$ must also be ramified and $\tau_1$ is unramified, we deduce that $\tau\tau_1$ is both ramified and tamely ramified. Next we observe that (for $p>2$) $\tau\tau_2$ is ramified, because otherwise $\tau|_{\ringofintegers^{\times}}=\tau_2^{-1}|_{\ringofintegers^{\times}}$, but by definition $\tau_2^2|_{\ringofintegers^{\times}} \equiv \mathbbm{1}$ contradicting the assumption that $\tau^2$ is ramified.
Applying \eqref{eq:Tate eps factor ramified tamely ramified} to $\tau\tau_1$ and $\tau^2$, and to $\tau\tau_2$, and noting that
$G(\tau^{-1}\tau_1^{-1},\psi)=G(\tau^{-1},\psi)$ because $\tau_1$ is unramified, identity~\eqref{eq:first computation of gamma Pi'} becomes
\begin{align}\label{eq:gamma case 1 of 3 uniform}
\gamma(s,\Pi'\times\tau,\psi)=&
\pi(-I_{2n})\tau(-1)^{n}\tau^{-1}(-\alpha)
G(\tau^{-2},\psi)
\\&\times\chi(g_{\chi})\tau(\varpi') q^{1/2-s}G(\tau^{-1},\psi)^{-1}[q^{1/2}G(\tau^{-1}\tau_2^{-1},\psi)^{-1}].\nonumber
\end{align}
When $p=2$ we can further simplify by noting that $G(\tau^{-2},\psi)=G(\tau^{-1},\psi)$, because $\psi$ is
invariant under $\Gal(\padicfield/\mathbb{F}_2)$ ($\mathbb{F}_2$ - the finite field of cardinality $2$), in particular under $x\mapsto x^2$.

The character $\tau_2|_{\mcO^{\times}}$ is the nontrivial quadratic character. Hence by the Hasse--Davenport product relation
(see e.g., \cite[Lemma~A.5(2)]{Oi24})
\begin{align}\label{eq:Hasse--Davenport}
G(\tau^{-1}\tau_2^{-1},\psi)G(\tau^{-1},\psi)=G(\tau^{-2},\psi)G(\tau_2^{-1},\psi)\tau(4).
\end{align}
Also by \cite[(23.6.3)]{BH06}, $G(\tau_2^{-1},\psi)=G(\tau_2^{-1},\psi)^{-1}\tau_2(-1)q$, and after applying \eqref{eq:Tate eps factor ramified tamely ramified} again,
\begin{align*}
q^{1/2}G(\tau^{-1}\tau_2^{-1},\psi)^{-1}&=q^{-1/2}G(\tau^{-1},\psi)G(\tau^{-2},\psi)^{-1}G(\tau_2^{-1},\psi)\tau_2(-1)\tau^{-1}(4)
\\&=\tau^{-1}(4)G(\tau^{-1},\psi)G(\tau^{-2},\psi)^{-1}\tau_2(-1)\varepsilon(s,\tau_2,\psi).
\end{align*}

Thus \eqref{eq:gamma case 1 of 3 uniform} becomes
\begin{align}\label{eq:gamma p>2 case 1}
\gamma(s,\Pi'\times\tau,\psi)=&
\pi(-I_{2n})\tau(-1)^{n}\tau^{-1}(-\alpha)\tau(\varpi')
\chi(g_{\chi})[\tau^{-1}(4)\tau_2(-1)\varepsilon(s,\tau_2,\psi)]q^{1/2-s}.
\end{align}

\item Both $\tau^2$ and $\tau\tau_1$ are unramified. Then $A_{\tau}=1$.
Because $\tau_1$ is unramified, we deduce that $\tau$ is unramified. Thus we can apply
\eqref{eq:Tate eps factor unramified} to $\tau\tau_1$ and to $\tau^2$. Note that by our choice of $\tau_1$, $\tau_1(\varpi')=\chi(g_{\chi})$, and also $\tau(-1)=1$ and $\tau(\alpha)=1$ ($|\alpha|=1$).
Then \eqref{eq:first computation of gamma Pi'} equals
\begin{align*}
\gamma(s,\Pi'\times\tau,\psi)=&
\pi(-I_{2n})
q^{2s-3/2}\tau^{-2}(\varpi')\frac{1-\tau^2(\varpi')q^{1-2s}}{1-\tau^{-2}(\varpi')q^{2s-2}}
\\&\times\left(  (q-1)\frac{\tau^2(\varpi')q^{1/2-2s}}{1-\tau^2(\varpi')q^{1-2s}}
+\chi(g_{\chi})\tau(\varpi') q^{1/2-s}\right)
\\&\times
q^{-s+1/2}\tau(\varpi')\chi(g_{\chi}) \frac{1 - \tau(\varpi')^{-1}\chi(g_{\chi})q^{s-1}}{1-\tau(\varpi')\chi(g_{\chi})q^{-s}}
[\gamma(s,\tau\tau_2,\psi)^{-1}].\nonumber
\end{align*}
Cancelling the numerator in the expression for $\gamma(2s-1,\tau^2,\psi)$ with the denominator in \eqref{eq:first computation of gamma Pi'}
and simplifying we obtain
\begin{align*}
\gamma(s,\Pi'\times\tau,\psi)=&
 \pi(-I_{2n})
q^{-1/2}\frac{1}{1-\tau^{-2}(\varpi)q^{2s-2}}
\\&\times(1+\tau(\varpi')\chi(g_{\chi})q^{1-s})(1-\tau(\varpi')\chi(g_{\chi})q^{-s})
\\&\times
\frac{1 - \tau(\varpi')^{-1}\chi(g_{\chi})q^{s-1}}{1-\tau(\varpi')\chi(g_{\chi})q^{-s}}
[\gamma(s,\tau\tau_2,\psi)^{-1}]\\=&\pi(-I_{2n})
q^{-1/2}\frac{(1+\tau(\varpi')\chi(g_{\chi})q^{1-s})
(1 - \tau(\varpi')^{-1}\chi(g_{\chi})q^{s-1})}{1-\tau^{-2}(\varpi')q^{2s-2}}\\
&\times[\gamma(s,\tau\tau_2,\psi)^{-1}].
\end{align*}
Recall that $\chi(g_{\chi})^2=1$. Since for any constant $c$ such that $c^2=1$,
\begin{align}\label{eq:identity for 1-a^2}
1-\tau^{-2}(\varpi')q^{2s-2}=(1-\tau^{-1}(\varpi')cq^{s-1})(1+\tau^{-1}(\varpi')cq^{s-1}),
\end{align}
and using $\tfrac{1+z}{1+z^{-1}}=z$, we can further simplify the expression for $\gamma(s,\Pi'\times\tau,\psi)$ and reach
\begin{align*}
\gamma(s,\Pi'\times\tau,\psi)=\pi(-I_{2n})\tau(\varpi')\chi(g_{\chi})[\gamma(s,\tau\tau_2,\psi)^{-1}]q^{1/2-s}.
\end{align*}

In addition for $p>2$, because $\tau$ is unramified,
$\tau\tau_2$ is tamely ramified and not unramified hence by \eqref{eq:Tate eps factor ramified tamely ramified} and \cite[(23.6.3)]{BH06} applied to $\tau\tau_2$,
\begin{align*}
\gamma(s, \tau\tau_2, \psi) &= q^{-1/2}G(\tau^{-1}\tau_2^{-1},\psi)=
q^{-1/2}G(\tau_2^{-1},\psi)=\tau_2(-1)\varepsilon(s, \tau_2, \psi)^{-1}.
\end{align*}
Thus
\begin{align}\label{eq:gamma p>2 case 2}
\gamma(s,\Pi'\times\tau,\psi)=\pi(-I_{2n})\tau(\varpi')\chi(g_{\chi})[\tau_2(-1)\varepsilon(s, \tau_2, \psi)]q^{1/2-s}.
\end{align}

\item $\tau^2$ is unramified and $\tau\tau_1$ is ramified. If $p=2$ we deduce $\tau$ is also unramified, hence $\tau\tau_1$ cannot be ramified. Thus we can now assume $p>2$. Again $A_{\tau}=1$.
Now $\tau$ is ramified (because $\tau_1$ is unramified) whence $\tau\tau_1$ is tamely ramified and not unramified.
In addition since $\tau^2$ is unramified, and since there is a unique nontrivial quadratic character of $\ringofintegers^{\times}$,
we deduce that $\tau|_{\ringofintegers^{\times}}=\tau_2|_{\ringofintegers^{\times}}$, in particular
$\tau\tau_2$ must be unramified. By \eqref{eq:Tate eps factor ramified tamely ramified} applied to $\tau\tau_1$ and \eqref{eq:Tate eps factor unramified} applied to $\tau^2$ and $\tau\tau_2$, and using $\tau_2(\varpi') = \chi(g_{\chi}) \tau_2(\Param)$,
\begin{align*}
\gamma(s,\Pi'\times\tau,\psi)=&
\pi(-I_{2n})\tau(-1)^{n}
q^{2s-3/2}\tau^{-2}(\varpi')\frac{1-\tau^2(\varpi')q^{1-2s}}{1-\tau^{-2}(\varpi')q^{2s-2}}
\\&\times\left(  (q-1)\frac{\tau^2(\varpi')q^{1/2-2s}}{1-\tau^2(\varpi')q^{1-2s}}
+\chi(g_{\chi})\tau^{-1}(-\alpha)\tau(\varpi') q^{1/2-s}\right)
\\&\times q^{1/2-s}\tau(\varpi')\chi(g_{\chi})\tau_2(\Param)\frac{1 - \tau(\varpi')^{-1}\chi(g_{\chi})\tau_2^{-1}(\Param)q^{s-1}}{1-\tau(\varpi')\chi(g_{\chi})\tau_2(\Param)q^{-s}}q^{1/2}G(\tau^{-1}\tau_1^{-1}, \psi)^{-1}.
\end{align*}
As in the previous case the numerator of $\gamma(2s-1,\tau^2,\psi)$ cancels with the denominator in \eqref{eq:first computation of gamma Pi'}.
We can further simplify using the fact that now $\tau_2(\Param)=\tau(\Param)$ (then $\tau_2(\Param)\tau^{-1}(-\alpha)=\tau(4)=1$) and also apply
\eqref{eq:identity for 1-a^2} with $c=\chi(g_{\chi})\tau(\Param)$. We obtain
\begin{align*}
&\gamma(s,\Pi'\times\tau,\psi)\\&=
\pi(-I_{2n})\tau(-1)^{n}
\frac1{1-\tau^{-2}(\varpi')q^{2s-2}}
G(\tau^{-1}\tau_1^{-1}, \psi)^{-1}
\\&\times\left(  (q-1)\tau(\Param)\chi(g_{\chi})q^{-s}\tau(\varpi')
+(1-q^{1-2s}\tau^2(\varpi'))\right)\frac{1 - \tau(\varpi')^{-1}\chi(g_{\chi})\tau^{-1}(\Param)q^{s-1}}{1-\tau(\varpi')\chi(g_{\chi})\tau(\Param)q^{-s}}
\\&=\pi(-I_{2n})\tau(-1)^{n}
G(\tau^{-1}\tau_1^{-1}, \psi)^{-1}
\\&\times\frac{(1-\tau(\varpi')\chi(g_{\chi})\tau(\Param)q^{-s})(1+\tau(\varpi')\chi(g_{\chi})\tau(\Param)q^{1-s})}
{(1 + \tau(\varpi')^{-1}\chi(g_{\chi})\tau^{-1}(\Param)q^{s-1})(1-\tau(\varpi')\chi(g_{\chi})\tau(\Param)q^{-s})}
\\&=\pi(-I_{2n})\tau(-1)^{n}\tau(\varpi')\tau(\Param)\chi(g_{\chi})G(\tau^{-1}\tau_1^{-1}, \psi)^{-1}q^{1-s}.
\end{align*}
Next we apply \eqref{eq:Hasse--Davenport} with $(\tau,\tau_2)$ replaced by $(\tau_1,\tau)$ (now $\tau$ is the nontrivial quadratic character of $\ringofintegers^{\times}$) and obtain
\begin{align*}
G(\tau^{-1}\tau_1^{-1},\psi)G(\tau_1^{-1},\psi)=G(\tau_1^{-2},\psi)G(\tau^{-1},\psi)\tau_1(4).
\end{align*}
Since $\tau_1$ is unramified, $G(\tau_1^{-2},\psi)=G(\tau_1^{-1},\psi)$ and $\tau_1(4)=1$ hence
\begin{align*}
G(\tau^{-1}\tau_1^{-1},\psi)=G(\tau^{-1},\psi)=G(\tau_2^{-1},\psi)=\tau_2(-1)\varepsilon(s,\tau_2,\psi)^{-1}q^{1/2}.
\end{align*}
Thus
\begin{align}\label{eq:gamma p>2 case 3}
&\gamma(s,\Pi'\times\tau,\psi)=\pi(-I_{2n})\tau(-1)^{n}\tau(\Param)\tau(\varpi')\chi(g_{\chi})\tau_2(-1)
\varepsilon(s,\tau_2,\psi)q^{1/2-s}
\end{align}
\end{itemize}
Finally note that \eqref{eq:gamma p>2 case 2} and \eqref{eq:gamma p>2 case 3} are particular cases of \eqref{eq:gamma p>2 case 1}, where in the former case $\tau(-\alpha)[\tau(4)]=1$ and in the latter (which only occurs when $p>2$) $\tau(\Param)=\tau(-\alpha^{-1})\tau^{-1}(4)$. We also mention that for $p = 2$, $-1 \in 1 + \mathfrak{p}$ whence $\pi(-I_{2n}) = 1$ ($\pi$ is a simple supercuspidal representation, thus its central character is tamely ramified)
and also $\tau(-1)=1$.
\end{proof}

\section{$L$-packets and $L$-parameters for simple supercuspidals of \texorpdfstring{$\SO_{N}$}{SO(N)}}\label{sec:SSC-LLC}

In the following, we consider the special orthogonal group $\SO_{N}$ for an integer $N$ greater than $2$.
Let $\pi$ be a simple supercuspidal representation of $\SO_{N}(F)$ and we write
\[
\pi=
\begin{cases}
\pi^{\SO_{2n+1}}_{a,\zeta}& \text{when $N=2n+1$}\\
\pi^{\SO_{2n}}_{\xi,\kappa,a,\zeta}& \text{when $N=2n$}
\end{cases}
\]
as in Section \ref{sec:SSC}.
Let $\phi$ be the $L$-parameter of $\pi$ in the sense of Arthur, thus $\pi$ is contained in the $L$-packet $\tilde{\Pi}_{\phi}^{\SO_{N}}$.
As discussed in Section \ref{subsec:LLC}, when $N=2n$ (resp.\ $N=2n+1$), $\phi$ is regarded as a $2n$-dimensional orthogonal (resp.\ symplectic) representation of $W_{F}\times\SL_{2}(\C)$.

Our aim is to determine the structure of $\tilde{\Pi}_{\phi}^{\SO_{N}}$ and describe $\phi$ explicitly as a $2n$-dimensional representation of $W_{F}\times\SL_{2}(\C)$. We let
\[
\phi=\phi_{0}\oplus\cdots\oplus\phi_{r}
\]
be the irreducible decomposition of $\phi$.

\subsection{Rough form of the $L$-parameters}\label{subsec:rough-form}
As explained in Section \ref{subsec:Moeglin-Xu}, the fact that $\pi$ is a discrete series representation implies that
$\phi_{i}$ is inequivalent to $\phi_{j}$ for any $i\neq j$ and moreover, when $N$ is even, each $\phi_{i}$ is irreducible orthogonal, and if $N$ is odd, each $\phi_{i}$ is irreducible symplectic.

Recall that the Whittaker datum $\mfw$ was specified in Section \ref{subsec:LLC}.
The following proposition can be proved in the same way as in Lemma \ref{lem:generic}.
\begin{prop}\label{prop:generic}
\begin{enumerate}
\item
When $N=2n+1$, the simple supercuspidal representation $\pi=\pi^{\SO_{2n+1}}_{a,\zeta}$ is $\mfw$-generic.
\item
When $N=2n$, the simple supercuspidal representation $\pi=\pi^{\SO_{2n}}_{\xi,\kappa,a,\zeta}$ is $\mfw$-generic if $\kappa=0$.
\end{enumerate}
\end{prop}


\begin{cor}\label{cor:generic-supercuspidal}
The $L$-parameter $\phi$ is trivial on $\SL_{2}(\C)$ and all members of $\tilde{\Pi}_{\phi}^{\SO_{N}}$ are supercuspidal.
\end{cor}

\begin{proof}
When $N=2n+1$, $\tilde{\Pi}_{\phi}^{\SO_{2n+1}}$ contains $\pi$, which is a $\mfw$-generic supercuspidal representation by Proposition \ref{prop:generic} (1).
Hence the assertion follows from Proposition \ref{prop:Moeglin-Xu}.

We next consider the case where $N=2n$.
We note that the orbit of $\pi$ with respect to the action of the adjoint group of $\SO_{2n}$ is contained in $\tilde{\Pi}_{\phi}^{\SO_{2n}}$ by the stability of the $L$-packet $\tilde{\Pi}_{\phi}^{\SO_{2n}}$ (see \cite[Corollary 4.2]{Oi24}).
It is not difficult to check that $\pi^{\SO_{2n}}_{\xi,1,a,\zeta}$ and $\pi^{\SO_{2n}}_{\xi,0,a\epsilon^{-1},\zeta}$ belong to the same orbit (see Section \cite[Section 5.1]{Oi24} for the argument in the case of symplectic groups; a similar computation works).
Thus, using Proposition \ref{prop:generic} (2), we see that $\tilde{\Pi}_{\phi}^{\SO_{2n}}$ contains a $\mfw$-generic supercuspidal representation, hence the same argument as in the previous paragraph works.
\end{proof}

\subsection{From twisted \texorpdfstring{$\gamma$}{gamma}-factor to Swan exponent}\label{subsec:AK}

The following is a key input of our proof, which follows from the computation of twisted $\gamma$-factors of simple supercuspidal representations established in \cite{AL16,Adr16,AK21} and Section \ref{sec:AK}.

\begin{prop}\label{prop:gamma}
\begin{enumerate}
\item
Suppose that $N=2n+1$.
For any tamely ramified character $\tau$ of $F^{\times}$,
\[
\gamma(s,\pi^{\SO_{2n+1}}_{a,\zeta}\times\tau,\psi)
=
\zeta\cdot\tau(-a^{-1}\varpi)\cdot q^{\frac{1}{2}-s}.
\]
\item
Suppose that $N=2n$.
\begin{enumerate}
\item
When $p=2$, there exists a unique tamely ramified quadratic character contained in $\phi$, denoted by $\phi_{r}$, which is
\[
\begin{cases}
\text{the trivial character} & \text{if $\zeta=1$,}\\
\text{the nontrivial unramified quadratic character} & \text{if $\zeta=-1$.}
\end{cases}
\]
Moreover, for any tamely ramified character $\tau$ of $F^{\times}$,
\[
\gamma(s,\pi^{\SO_{2n}}_{a,\zeta}\times\tau,\psi)
=
\zeta\cdot\tau(a^{-1}\varpi)\cdot q^{\frac{1}{2}-s}\cdot\gamma(s,\phi_{r}\otimes\tau,\psi).
\]
\item
When $p\neq2$, there exist exactly two tamely ramified quadratic characters contained in $\phi$, denoted by $\phi_{r-1}$ and $\phi_{r}$, which are given as follows:
\begin{itemize}
\item
$\phi_{r-1}$ is the unique unramified quadratic character of $F^{\times}$ satisfying $\phi_{r-1}(a^{-1}\varpi)=\zeta$,
\item
$\phi_{r}$ is the unique ramified quadratic character of $F^{\times}$ satisfying $\phi_{r}(a^{-1}\varpi)=\zeta\cdot\phi_{r}(-4\epsilon^{\kappa})$.
\end{itemize}
Moreover, for any tamely ramified character $\tau$ of $F^{\times}$,
\begin{multline*}
\gamma(s,\pi^{\SO_{2n}}_{\xi,\kappa,a,\zeta}\times\tau,\psi)
=
\xi\cdot\zeta\cdot\tau((-1)^{n+1}\varpi/4\epsilon^{\kappa}a)q^{1-s}G(\phi_{r},\psi)^{-1}\\
\cdot\gamma(s,\phi_{r-1}\otimes\tau,\psi)\cdot\gamma(s,\phi_{r}\otimes\tau,\psi).
\end{multline*}
\end{enumerate}
\end{enumerate}
\end{prop}

\begin{proof}
\begin{enumerate}
\item
By \cite[Corollary 7.3]{Adr16} and Remark \ref{rem:Adr16},
\[
\gamma(s,\pi^{\zeta}_{\chi'}[\varpi']\times\tau,\psi)
=
\zeta\cdot\tau((-1)^{n}\varpi')\cdot q^{\frac{1}{2}-s}
\]
for any tamely ramified character $\tau$ of $F^{\times}$.
Here, $\pi^{\zeta}_{\chi'}[\varpi']$ is (a modified version of) the simple supercuspidal representation constructed in \cite{Adr16}, as reviewed in Section \ref{subsubsec:SO-odd}.
By the discussion in Section \ref{subsubsec:SO-odd}, if we put $\varpi':=(-1)^{n+1}a^{-1}\varpi$, then $\pi^{\SO_{2n+1}}_{a,\zeta}\cong\pi^{\zeta}_{\chi'}[\varpi']$.
Thus we get
\begin{align*}
\gamma(s,\pi^{\SO_{2n+1}}_{a,\zeta}\times\tau,\psi)
&=\gamma(s,\pi^{\zeta}_{\chi'}[\varpi']\times\tau,\psi)\\
&=\zeta\cdot\tau((-1)^{n}\varpi')\cdot q^{\frac{1}{2}-s}\\
&=\zeta\cdot\tau(-a^{-1}\varpi)\cdot q^{\frac{1}{2}-s}.
\end{align*}

\item
\begin{enumerate}
\item
Let, $\pi^{\mathbbm{1}}_{1}[\varpi',\zeta]$ be the simple supercuspidal representation of $\SO_{2n}(F)$ constructed in \cite{AK21}, as reviewed in Section \ref{subsubsec:SO-even}.
By the discussion in Section \ref{subsec:AK-Langlands-parameter}, there exists a unique tamely ramified character $\tau_{1}$ contained in the $L$-parameter of $\pi^{\mathbbm{1}}_{1}[\varpi',\zeta]$, which is trivial when $\zeta=1$ and nontrivial quadratic unramified when $\zeta=-1$.
By the discussion in Section \ref{subsubsec:SO-even}, if we put $a:=\varpi\varpi^{\prime-1}\in k^{\times}$, then $\pi^{\SO_{2n}}_{a,\zeta}\cong\pi^{\mathbbm{1}}_{1}[\varpi',\zeta]$.
Then, by putting $\phi_{r}:=\tau_{1}$, we get $\phi_{r}$ as desired.
Moreover, by Theorem \ref{thm:AK}, for any tamely ramified character $\tau$ of $F^{\times}$,
\[
\gamma(s,\pi^{\mathbbm{1}}_{1}[\varpi',\zeta]\times\tau,\psi)
=
\zeta\cdot\tau(\varpi')\cdot q^{\frac{1}{2}-s}\cdot\gamma(s,\phi_{r}\otimes\tau,\psi).
\]
Thus we get
\begin{align*}
\gamma(s,\pi^{\SO_{2n}}_{a,\zeta}\times\tau,\psi)
&=\gamma(s,\pi_{1}^{\mathbbm{1}}[\varpi',\zeta]\times\tau,\psi)\\
&=\zeta\cdot\tau(\varpi')\cdot q^{\frac{1}{2}-s}\cdot\gamma(s,\phi_{r}\otimes\tau,\psi)\\
&=\zeta\cdot\tau(a^{-1}\varpi)\cdot q^{\frac{1}{2}-s}\cdot\gamma(s,\phi_{r}\otimes\tau,\psi).
\end{align*}

\item
Let $\pi^{\omega}_{\alpha}[\varpi',\zeta]$ be the simple supercuspidal representation of $\SO_{2n}(F)$ constructed in \cite{AK21}, as reviewed in Section \ref{subsubsec:SO-even}.
By the discussion in Section \ref{subsec:AK-Langlands-parameter}, there exist exactly two tamely ramified quadratic characters $\tau_{1}$ and $\tau_{2}$ contained in the $L$-parameter of $\pi^{\omega}_{\alpha}[\varpi',\zeta]$, which are given as follows:
\begin{itemize}
\item
$\tau_{1}$ is the unique unramified quadratic character of $F^{\times}$ satisfying $\tau_{1}(\varpi')=\zeta$,
\item
$\tau_{2}$ is the unique ramified quadratic character of $F^{\times}$ satisfying $\tau_{2}(\varpi')=\zeta\cdot\tau_{2}(-4\alpha)$.
\end{itemize}
By the discussion in Section \ref{subsubsec:SO-even}, if we let $\alpha:=\epsilon^{\kappa}$ and put $\xi:=\omega(-1)$ and $a:=\varpi\varpi'^{-1}\in k^{\times}$, then $\pi^{\SO_{2n}}_{\xi,\kappa,a,\zeta}\cong\pi^{\omega}_{\alpha}[\varpi',\zeta]$.
Then, by putting $\phi_{r-1}:=\tau_{1}$ and $\phi_{r}:=\tau_{2}$, we get $\phi_{r-1}$ and $\phi_{r}$ as desired.
Moreover, by Theorem \ref{thm:AK}, for any tamely ramified character $\tau$ of $F^{\times}$, the twisted $\gamma$-factor $\gamma(s,\pi^{\omega}_{\alpha}[\varpi',\zeta]\times\tau,\psi)$ is given by the product of
\[
\xi\cdot\zeta\cdot\tau((-1)^{n+1}\varpi'/4\alpha)q^{\frac{1}{2}-s}\varepsilon(s,\phi_{r},\psi)\phi_{r}(-1)
\]
with $\gamma(s,\tau_{1}\otimes\tau,\psi)\cdot\gamma(s,\phi_{r}\otimes\tau,\psi)$.
By the dictionary between $\pi^{\SO_{2n}}_{\xi,\kappa,a,\zeta}$ and $\pi^{\omega}_{\alpha}[\varpi',\zeta]$ and the equality \eqref{eq:Tate eps factor unramified}, this equals
\[
\xi\cdot\zeta\cdot\tau((-1)^{n+1}\varpi/4\epsilon^{\kappa}a)q^{1-s}G(\phi_{r},\psi)^{-1}.\qedhere
\]
\end{enumerate}
\end{enumerate}
\end{proof}

\begin{rem}\label{rem:gamma}
The $\gamma$-factors on the left-hand sides of the equalities in the above theorem are defined via the Rankin--Selberg integrals (or equivalently via the Langlands--Shahidi method; see \cite{Sou93,Sou95} for $N=2n+1$ and \cite{Kap13-MM} for $N=2n$; see also \cite{Kap15}).
Note that $\gamma(s,\pi\times\tau,\psi)$ is denoted by $\Gamma(s,\pi\times\tau,\psi)$ in \cite{Adr16}.
\end{rem}

\begin{rem}\label{rem:Adr16}
In fact, the result in \cite{Adr16} contains some errors.
We would like to take this opportunity to describe how to correct them.
The statement of \cite[Corollary 7.3]{Adr16} is that
\[
\gamma(s,\pi^{\zeta}_{\chi}[\varpi']\times\tau,\psi)
=
\zeta\cdot\tau((-1)^{n}\varpi')\cdot q^{\frac{1}{2}-s}.
\]
The problem is that the element $\diag(I_{n-1},\frac{1}{2},1,2,I_{n-1})$ taken at the end of \cite[Section 5]{Adr16} is incorrect and should instead be $\diag(\frac{1}{2}I_{n},1,2I_{n})$.
If we repeat the same computation as in \cite{Adr16} with this change, then we arrive at the formula
\begin{align}\label{eq:Adr1}
\gamma(s,\pi^{\zeta}_{\chi}[\varpi']\times\tau,\psi)
=
\zeta\cdot\tau((-1)^{n}4^{-1}\varpi')\cdot q^{\frac{1}{2}-s}.
\end{align}

However, there is also another issue that the definition of the simple supercuspidal representation $\pi^{\zeta}_{\chi}[\varpi']$ makes sense only when $p\neq2$, as discussed in Section \ref{subsubsec:SO-odd}.
For this reason, the better way is to adopt the affine generic character ``$\chi'$" and consider the associated simple supercuspidal representation $\pi^{\zeta}_{\chi'}[\varpi']$ (see Section \ref{subsubsec:SO-odd}).
In this case, since $\chi'$ and the generic character ``$\psi$'' taken in \cite[page 200]{Adr16} have the same restriction on $U\cap I^{+}_{\SO_{2n+1}}$, we do not need a twisting process as discussed at the end of \cite[Section 5]{Adr16} (i.e., the element $\gamma$ of \textit{loc.\ cit.}\ can be taken to be $I_{2n+1}$).
This makes the computation performed in \cite{Adr16} even simpler and enables us to get the formula
\begin{align}\label{eq:Adr2}
\gamma(s,\pi^{\zeta}_{\chi'}[\varpi']\times\tau,\psi)
=
\zeta\cdot\tau((-1)^{n}\varpi')\cdot q^{\frac{1}{2}-s}
\end{align}
for any $p$ uniformly.
The identity \eqref{eq:Adr2} is the one used in the proof of Proposition \ref{prop:gamma} (1).
We remark that the formulas \eqref{eq:Adr1} and \eqref{eq:Adr2} are consistent when $p\neq2$ since $\pi^{\zeta}_{\chi'}[\varpi']\cong\pi^{\zeta}_{\chi}[4\varpi']$ (see Section \ref{subsubsec:SO-odd}).
\end{rem}

\begin{cor}\label{cor:Swan}
The Swan exponent of $\phi$ is $1$.
\end{cor}

\begin{proof}
Let us consider the case where $N=2n$ and $p\neq2$. By Proposition \ref{prop:gamma} (2) (b) both
$\phi_{r-1}$ and $\phi_{r}$ are tamely ramified characters. Therefore $\Swan(\phi_{r-1})=\Swan(\phi_{r})=0$.
Let us check that $\sum_{i=0}^{r-2}\Swan(\phi_{i})=1$.
It is known that Arthur's local Langlands correspondence preserves the Rankin--Selberg $\gamma$-factors (we will explain a justification of this fact in Appendix \ref{sec:lift}).
In particular,
\[
\gamma(s,\pi\times\tau,\psi)
=\gamma(s,\phi\otimes\tau,\psi)
=\prod_{i=0}^{r}\gamma(s,\phi_{i}\otimes\tau,\psi)
\]
for any tamely ramified character $\tau$ of $F^{\times}$.
Thus, by Proposition \ref{prop:gamma} (2) (b),
\begin{align}\label{eq:AK}
\prod_{i=0}^{r-2}\gamma(s,\phi_{i}\otimes\tau,\psi)
=
\xi\cdot\zeta\cdot\tau((-1)^{n+1}\varpi/4\epsilon^{\kappa}a)q^{1-s}G(\phi_{r},\psi)^{-1}.
\end{align}
Recall that
\[
\gamma(s,\phi_{i}\otimes\tau,\psi)
=
\frac{L(1-s,\phi_{i}^{\vee}\otimes\tau^{\vee})}{L(s,\phi_{i}\otimes\tau)}\cdot\varepsilon(s,\phi_{i}\otimes\tau,\psi).
\]
We note that $L(s,\phi_{i}\otimes\tau)$ (and $L(1-s,\phi_{i}^{\vee}\otimes\tau^{\vee})$) can be nontrivial only when $\phi_{i}\otimes\tau$ is an unramified character since $\phi_{i}\otimes\tau$ is irreducible.
However, if $\phi_{i}\otimes\tau$ is unramified for $0\leq i \leq r-2$, then $\phi_{i}$ is a quadratic tamely ramified character, which is a contradiction since only $\phi_{r-1}$ and $\phi_{r}$ are such characters (Proposition \ref{prop:gamma} (2) (b)).
On the other hand, since $\phi_{i}$ is self-dual,
\[
|\varepsilon(s,\phi_{i},\psi)|
=
q^{\Swan(\phi_{i})(\frac{1}{2}-s)}
\]
(see \cite[(10)]{GR10}; note that $\psi$ is of level $1$ here).
Hence, by taking the absolute value of \eqref{eq:AK}, we get
\[
\prod_{i=0}^{r-2}q^{\Swan(\phi_{i})(\frac{1}{2}-s)}
=
q^{\frac{1}{2}-s}
\]
(recall that $|G(\phi_{r},\psi)|=q^{1/2}$).
This implies that $\sum_{i=0}^{r-2}\Swan(\phi_{i})=1$.

The case where $N=2n+1$ and the case where $N=2n$ and $p=2$ can be treated by a similar, but simpler, argument.
(In the case where $N=2n+1$, note that each irreducible constituent $\phi_{i}$ is symplectic, hence not a character).
\end{proof}

By Corollary \ref{cor:Swan}, we may suppose that
\begin{itemize}
\item
$\Swan(\phi_{0})=1$ and
\item
$\Swan(\phi_{i})=0$ for any $0<i\leq r$.
\end{itemize}
Moreover, by Proposition \ref{prop:gamma} and Lemma \ref{lem:tame-rep}, we see that when $N=2n$, the dimension of $\phi_{i}$ is even for $0<i<r$ when $p=2$ and for $0<i<r-1$ when $p\neq2$, and if $N=2n+1$, the dimension of $\phi_{i}$ is even for all $0<i\leq r$.
Thus the dimension of $\phi_{0}$ must be
\[
\begin{cases}
\text{even} & \text{when $N=2n+1$,}\\
\text{odd} & \text{when $N=2n$ and $p=2$,}\\
\text{even} & \text{when $N=2n$ and $p\neq2$.}
\end{cases}
\]

\subsection{Utilization of the formal degree conjecture}\label{subsec:SSC-FDC}
We next utilize the formal degree conjecture for $\tilde{\Pi}_{\phi}^{\SO_{N}}$.

\begin{lem}\label{lem:SSC-FD}
The absolute value of the formal degree $\deg_{\mu}(\pi)$ is specified as follows:
\[
|\deg_{\mu}(\pi)|
=
\gamma(0,\Ad\circ\phi_{\mathrm{pr}},\psi_{0})^{-1}\cdot q^{n^{2}}\cdot
\begin{cases}
q^{n} & \text{when $N=2n+1$,}\\
1 & \text{when $N=2n$ and $p=2$,}\\
\frac{1}{2} & \text{when $N=2n$ and $p\neq2$.}
\end{cases}
\]
\end{lem}

\begin{proof}
The formal degree of a simple supercuspidal representation with respect to the Euler--Poincare measure $\mu$ is computed in \cite[Section 9.4]{GR10}.
Although it is supposed that the connected reductive group is simply-connected in \cite[Section 7.1]{GR10}, we can easily modify their computation in our setting as follows.

Recall that the simple supercuspidal $\pi$ is given by the compact induction of a $1$-dimensional character of
\[
\begin{cases}
I^{+}_{\SO_{2n+1}}\langle\varphi_{a^{-1}}^{\SO_{2n+1}}\rangle & \text{if $N=2n+1$,}\\
Z_{\SO_{2n}} I_{\SO_{2n}}^{+} \lan\varphi_{\epsilon^{\kappa},-a^{-1}}^{\SO_{2n}}\ran & \text{if $N=2n$.}
\end{cases}
\]
(Section \ref{sec:SSC}).
Therefore, the well-known formula of the formal degree of a compactly-induced supercuspidal representation (e.g., see \cite[Theorem A.14]{BH96}) implies that
\begin{align}\label{eq:FD}
\deg_{\mu}(\pi)
=\frac{1}{2}\cdot\mu(I_{\SO_{N}}^{+})^{-1}\cdot
\begin{cases}
 1 & \text{if $N=2n+1$,}\\
1  & \text{if $N=2n$ and $p=2$,}\\
\frac{1}{2}& \text{if $N=2n$ and $p\neq2$.}
\end{cases}
\end{align}

By \cite[(55)]{GR10},
\begin{align}\label{eq:GR55}
|\mu(I_{\SO_{N}})|
=
\frac{L(1,\Ad\circ\phi_{\mathrm{pr}})}{L(0,\Ad\circ\phi_{\mathrm{pr}})\cdot|\L{Z}|}\cdot|T(q)|\cdot q^{-l}.
\end{align}
Here,
\begin{itemize}
\item
$\L{Z}$ denotes the center of the Langlands dual group of $\SO_{N}$, hence $|\L{Z}|=2$,
\item
$T(q)$ denotes the set of $F$-rational elements of the split maximal torus whose order is finite and prime to $p$,
\item
$l$ denotes the rank of $\SO_{N}$, hence $l=n$.
\end{itemize}
As explained in \cite[Section 9.4]{GR10},
\begin{align}\label{eq:GR9.4}
\gamma(0,\Ad\circ\phi_{\mathrm{pr}},\psi_{0})
=
q^{M}\cdot\frac{L(1,\Ad\circ\phi_{\mathrm{pr}})}{L(0,\Ad\circ\phi_{\mathrm{pr}})},
\end{align}
where $M$ denotes the number of positive roots in $\SO_{N}$, hence
\[
M=
\begin{cases}
n^{2} & \text{when $N=2n+1$,}\\
n(n-1) & \text{when $N=2n$.}
\end{cases}
\]
Therefore, the equalities \eqref{eq:GR55} and \eqref{eq:GR9.4} imply that
\[
|\mu(I_{\SO_{N}})|
=\frac{|T(q)|}{2}\cdot q^{-(M+l)}\cdot\gamma(0,\Ad\circ\phi_{\mathrm{pr}},\psi_{0}).
\]
Because $[I_{\SO_{N}}:I_{\SO_{N}}^{+}]=|T(q)|$, we get
\[
|\mu(I_{\SO_{N}}^{+})|
=
\frac{1}{2}\cdot q^{-(M+l)}\cdot\gamma(0,\Ad\circ\phi_{\mathrm{pr}},\psi_{0}).
\]
Thus, by the equality \eqref{eq:FD}, we get the desired equality.
\end{proof}

\begin{rem}\label{rem:Romano}
When $N=2n+1$, Lemma \ref{lem:SSC-FD} follows from \cite[Corollary 3.10 and (4.3)]{Rom23}.
\end{rem}

\begin{prop}\label{prop:SSC-FDC}
We have
\[
q^{\frac{1}{2}\Artin(\Ad\circ\phi)}\cdot\frac{|L(1,\Ad\circ\phi)|}{|L(0,\Ad\circ\phi)|}
=2^{r}\cdot q^{n^{2}}\cdot
\begin{cases}
q^{n} & \text{when $N=2n+1$,}\\
\frac12  & \text{when $N=2n$ and $p=2$,}\\
\frac14 & \text{when $N=2n$ and $p\neq2$.}
\end{cases}
\]
\end{prop}

\begin{proof}
Recall that the formal degree conjecture (Conjecture \ref{conj:FDC}) gives the identity
\[
|\deg_{\mu}(\pi)|
=
\frac{1}{|\overline{S}_{\phi}|}\cdot\frac{|\gamma(0,\Ad\circ\phi,\psi_{0})|}{|\gamma(0,\Ad\circ\phi_{\mathrm{pr}},\psi_{0})|}.
\]
According to the definition of $\overline{S}_{\phi}$ (see e.g., \cite[Section 4]{Oi24}),
\[
\overline{S}_{\phi}
\cong
\Bigl\{(s_i)_{i=0}^r\in\prod_{i=0}^{r}\{\pm1\} \,\Big\vert\, \prod_{i=0}^{r}s_i^{\dim\phi_i}=1\Bigr\} \Big/ \langle(-1,\ldots,-1)\rangle
.\]
Since the only representations $\phi_i$ with odd dimension (including characters) are important for the
condition $\prod_{i=0}^{r}s_i^{\dim\phi_i}=1$, the description of the dimensions from the end of Section \ref{subsec:AK} and the determination of the one-dimensional representations $\phi_i$ from Proposition \ref{prop:gamma} (2) imply that
\[
|\overline{S}_{\phi}|
=
\begin{cases}
2^{r} & \text{when $N=2n+1$,}\\
2^{r-1} & \text{when $N=2n$.}
\end{cases}
\]
Thus, Lemma \ref{lem:SSC-FD} implies that
\[
|\gamma(0,\Ad\circ\phi,\psi_{0})|
=
2^{r}\cdot q^{n^{2}}
\begin{cases}
q^{n} & \text{when $N=2n+1$,}\\
\frac{1}{2} & \text{when $N=2n$ and $p=2$,}\\
\frac{1}{4} & \text{when $N=2n$ and $p\neq2$.}
\end{cases}
\]
By recalling that
\begin{itemize}
\item
$|\gamma(0,\Ad\circ\phi,\psi_{0})|=\varepsilon(0,\Ad\circ\phi,\psi_{0})\cdot\frac{L(1,\Ad\circ\phi)}{L(0,\Ad\circ\phi)}$, and
\item
$|\varepsilon(0,\Ad\circ\phi,\psi_{0})|=q^{\frac{1}{2}\Artin(\Ad\circ\phi)}$ (here $\psi_{0}$ is taken to be of level $0$),
\end{itemize}
we get the desired equality.
\end{proof}

\begin{cor}\label{cor:L-factor-trivial}
We have $L(s,\Ad\circ\phi)=1$ and
\[
r=
\begin{cases}
0 & \text{when $N=2n+1$,}\\
1 & \text{when $N=2n$ and $p=2$,}\\
2 & \text{when $N=2n$ and $p\neq2$.}
\end{cases}
\]
\end{cor}

\begin{proof}
Let us consider only the case where $N=2n$ and $p\neq2$ because similar, but simpler, arguments work in the case where $N=2n+1$ and also in the case where $N=2n$ and $p=2$.

Since $\Ad\circ\phi=\wedge^{2}\phi$ when $N=2n$,
\[
\Ad\circ\phi
=
\Bigl(\bigoplus_{i=0}^{r}\wedge^{2}\phi_{i}\Bigr)
\oplus
\Bigl(\bigoplus_{0\leq i<j\leq r}\phi_{i}\otimes\phi_{j}\Bigr).
\]

We first investigate the terms $\wedge^{2}\phi_{i}$ for $0\leq i\leq r$.
\begin{enumerate}
\item
We consider the case where $i=0$.
By Proposition \ref{prop:par-L}, $L(s,\wedge^{2}\phi_{0})$ is given by either $1$ or $(1+q^{-e_{0}s})^{-1}$ for some non-negative integer $e_{0}$.
\item
We consider the case where $1\leq i\leq r-2$.
Note that since $\phi_{i}$ is tamely ramified and self-dual, $\phi_{i}$ cannot be a character by Proposition \ref{prop:gamma} (2) (b).
Thus, by Lemma \ref{lem:tame-rep}, $\phi_{i}$ is induced from a non-self-dual character of the Weil group of an unramified extension of $F$.
This implies that, by Proposition \ref{prop:par-L}, $L(s,\wedge^{2}\phi_{i})$ is given by $(1+q^{-e_{i}s})^{-1}$ for some non-negative integer $e_{i}$.
\item
We consider the case where $i=r-1,r$.
Since $\phi_{r-1}$ and $\phi_{r}$ are characters, $\wedge^{2}\phi_{r-1}=\wedge^{2}\phi_{r}=0$.
Hence $L(s,\wedge^{2}\phi_{r-1})=L(s,\wedge^{2}\phi_{r})=1$.
\end{enumerate}

We next investigate the terms $\phi_{i}\otimes\phi_{j}$ for $0\leq i<j\leq r$.
\begin{enumerate}
\item[(4)]
We consider the case where $i=0$ and $0<j\leq r$.
Since $\Swan(\phi_{0})=1$ and $\Swan(\phi_{j})=0$ ($0<j\leq r$), $\phi_{0}|_{P_{F}}$ cannot be isomorphic to $\phi_{j}|_{P_{F}}$.
In other words, $(\phi_{0}\otimes\phi_{j})|_{P_{F}}$ cannot contain the trivial character of $P_{F}$.
In particular, $\phi_{0}\otimes\phi_{j}$ does not contain any unramified character, hence $L(s,\phi_{0}\otimes\phi_{j})=1$.
\item[(5)]
We consider the case where $0<i<j\leq r-2$.
As discussed in the case (2), Lemma \ref{lem:tame-rep} implies that $\phi_{i}\cong\Ind_{W_{E_{i}}}^{W_{F}}\chi_{i}$ for an unramified extension $E_{i}$ of $F$ and a character $\chi_{i}$ of $E_{i}$ satisfying $\chi_{i}|_{E_{i}^{\prime\times}}=\mathbbm{1}$ ($E'_{i}$ is the unramified subextension of $E_{i}/F$ such that $[E_{i}:E'_{i}]=2$).
Similarly, $\phi_{j}\cong\Ind_{W_{E_{j}}}^{W_{F}}\chi_{j}$ for an unramified extension $E_{j}$ of $F$ and a character $\chi_{j}$ of $E_{j}$ satisfying $\chi_{j}|_{E_{j}^{\prime\times}}=\mathbbm{1}$.
Suppose that there exists an unramified character $\omega$ of $W_{F}$ which is contained in $\phi_{i}\otimes\phi_{j}$.
Then $\phi_{j}\cong\omega\otimes\phi_{i}$.
In particular, $\phi_{i}$ and $\phi_{j}$ have the same dimension, hence $E_{i}=E_{j}$ and $E'_{i}=E'_{j}$.
We put $E:=E_{i}=E_{j}$ and $E':=E'_{i}=E'_{j}$.
By looking at the restriction of $\phi_{j}\cong\omega\otimes\phi_{i}$ to $W_{E}$, we see that tensoring $\omega|_{W_{E}}$ maps the set $\{\chi_{i}^{w}\mid w\in\Gal(E/F)\}$ to the set $\{\chi_{j}^{w}\mid w\in\Gal(E/F)\}$.
Since both $\chi_{i}$ and $\chi_{j}$ are trivial on $E^{\prime\times}$, so is any element of $\{\chi_{i}^{w}\mid w\in\Gal(E/F)\}$ or $\{\chi_{j}^{w}\mid w\in\Gal(E/F)\}$.
In particular, the character $\omega|_{W_{E}}$ is, when regarded as a character of $E^{\times}$, trivial on $E'^{\times}$.
Since $\omega$ is an unramified character of $F^{\times}$, $\omega|_{W_{E}}$ is also an unramified character of $E^{\times}$.
As $E/E'$ is unramified, being trivial on $E'^{\times}$ implies that $\omega|_{W_{E}}$ is trivial.
Hence $\chi_{i}$ and $\chi_{j}$ are in the same $\Gal(E/F)$-orbit, but this implies that $\phi_{i}$ and $\phi_{j}$ are isomorphic.
Thus we get a contradiction.
Hence we get $L(s,\phi_{i}\otimes\phi_{j})=1$.
\item[(6)]
We consider the case where $0<i\leq r-2$ and $j=r-1,r$.
In this case $L(s,\phi_{i}\otimes\phi_{j})=1$ by the same argument as in the proof of Corollary \ref{cor:Swan}.
\item[(7)]
We consider the case where $i=r-1$ and $j=r$.
In this case, $\phi_{r-1}\otimes\phi_{r}$ is a ramified character by Proposition \ref{prop:gamma} (2) (b).
Thus the $L$-factor $L(s,\phi_{r-1}\otimes\phi_{r})$ is trivial.
\end{enumerate}

In summary, we see that $L(s,\Ad\circ\phi)$ is given by the product of $L(s,\wedge^{2}\phi_{0})$ (which equals either $1$ or $(1+q^{-e_{0}s})^{-1}$) and $(1+q^{-e_{i}s})^{-1}$'s ($1\leq i\leq r-2$).
Let us suppose that $L(s,\wedge^{2}\phi_{0})=(1+q^{-e_{0}s})^{-1}$ for the sake of contradiction.
Then
\[
\frac{|L(1,\Ad\circ\phi)|}{|L(0,\Ad\circ\phi)|}
=
\prod_{i=0}^{r-2}
\frac{(1+q^{-e_{i}\cdot1})^{-1}}{(1+q^{-e_{i}\cdot0})^{-1}}
=
\prod_{i=0}^{r-2}
\frac{2q^{e_{i}}}{1+q^{e_{i}}}.
\]
Thus, by Proposition \ref{prop:SSC-FDC}, we get
\[
\prod_{i=0}^{r-2}
\frac{2q^{e_{i}}}{1+q^{e_{i}}}
=
2^{r-2}\cdot q^{n^{2}-\frac{1}{2}\Artin(\Ad\circ\phi)},
\]
or equivalently,
\[
2^{2}
\cdot q^{\sum_{i=0}^{r-2}2e_{i}+\Artin(\Ad\circ\phi)}
=
q^{2n^{2}}\cdot\prod_{i=0}^{r-2}(1+q^{e_{i}})^{2}.
\]
By noting that $q$ is odd and prime to $1+q^{e_{i}}$, we must have $2^{2}=\prod_{i=0}^{r-2}(1+q^{e_{i}})^{2}$.
However, this cannot happen since $1+q^{e_{0}}>2$.
Thus we get $L(s,\wedge^{2}\phi_{0})=1$.

Now, again by using Proposition \ref{prop:SSC-FDC},
\[
\prod_{i=1}^{r-2}
\frac{2q^{e_{i}}}{1+q^{e_{i}}}
=
2^{r-2}\cdot q^{n^{2}-\frac{1}{2}\Artin(\Ad\circ\phi)},
\]
or equivalently,
\[
q^{\sum_{i=1}^{r-2}2e_{i}+\Artin(\Ad\circ\phi)}
=
q^{2n^{2}}\cdot\prod_{i=1}^{r-2}(1+q^{e_{i}})^{2}.
\]
Since $q$ is prime to $1+q^{e_{i}}$, we necessarily have $r=2$ so that this equality holds.
Accordingly, we get $L(s,\Ad\circ\phi)=1$.
\end{proof}

\subsection{Main Theorem}\label{Main Theorem}

Now let us determine the $L$-parameter $\phi$ as a $2n$-dimensional representation of $W_{F}$.

\subsubsection{The case of $\SO_{2n+1}$}
We first consider the case where $N=2n+1$.

\begin{thm}\label{thm:main-1}
Let $\pi^{\SO_{2n+1}}_{a,\zeta}$ be the simple supercuspidal representation of $\SO_{2n+1}(F)$ with $a\in k^{\times}$ and $\zeta\in\{\pm1\}$.
Then the $L$-parameter $\phi$ of $\pi^{\SO_{2n+1}}_{a,\zeta}$ is an irreducible symplectic representation of $W_{F}$ of dimension $2n$, which is the $L$-parameter of the simple supercuspidal representation $\pi^{\GL_{2n}}_{\mathbbm{1},a,\zeta}$ of $\GL_{2n}(F)$.
\end{thm}

\begin{proof}
We have shown that the $L$-parameter $\phi$ of the simple supercuspidal representation $\pi^{\SO_{2n+1}}_{a,\zeta}$ is an irreducible symplectic representation of $W_{F}$ of dimension $2n$ and Swan exponent $1$.
This implies that if we regard $\phi$ as an $L$-parameter of $\GL_{2n}$, then a simple supercuspidal representation of $\GL_{2n}(F)$ corresponds to $\phi$ (see \cite{BH14}).
Since the determinant of $\phi$ is trivial, the simple supercuspidal representation of $\GL_{2n}(F)$ has trivial central character, hence we may write $\pi^{\GL_{2n}}_{\mathbbm{1},a',\zeta'}$ for it.

For any tamely ramified character $\tau$ of $F^{\times}$, the Rankin--Selberg $\gamma$-factor for $(\pi^{\GL_{2n}}_{\mathbbm{1},a',\zeta'},\tau)$ is given by
\[
\gamma(s,\pi^{\GL_{2n}}_{\mathbbm{1},a',\zeta'}\times\tau,\psi)
=
\tau(-1)^{2n-1}\cdot\tau(\varpi a'^{-1})\cdot\zeta'\cdot q^{\frac{1}{2}-s}
\]
according to \cite[Corollary 3.12]{AL16} (see also Section \ref{subsubsec:GL}).
Therefore, by noting that the local Langlands correspondence for the general linear groups preserves the Rankin--Selberg local factors, Proposition \ref{prop:gamma} (1) implies that
\[
\tau(-1)^{2n-1}\cdot\tau(\varpi a'^{-1})\cdot\zeta'\cdot q^{\frac{1}{2}-s}
=
\zeta\cdot\tau(-a^{-1}\varpi)\cdot q^{\frac{1}{2}-s}.
\]
Since this identity holds for any tamely ramified character $\tau$, we conclude that $\zeta'=\zeta$ and $a'=a$.
\end{proof}

\subsubsection{The case of $\SO_{2n}$ with $p=2$}
We next consider the case where $N=2n$ and $p=2$.

\begin{thm}\label{thm:main-2-a}
Suppose that $N=2n$ and $p=2$.
Let $\pi^{\SO_{2n}}_{a,\zeta}$ be the simple supercuspidal representation of $\SO_{2n}(F)$ with $a\in k^{\times}$ and $\zeta\in\{\pm1\}$.
Then the $L$-parameter $\phi$ of $\pi^{\SO_{2n}}_{a,\zeta}$ is of the form $\phi=\phi_{0}\oplus\phi_{1}$, where
\begin{itemize}
\item
$\phi_{0}$ is an irreducible orthogonal representation of $W_{F}$ of dimension $2n-1$, which is the $L$-parameter of the simple supercuspidal representation $\pi^{\GL_{2n-1}}_{\mathbbm{1},a,\zeta}$ of $\GL_{2n-1}(F)$, and
\item
$\phi_{1}$ is the determinant character of $\phi_{0}$, which is the trivial character if $\zeta=1$ and the unique nontrivial unramified quadratic character if $\zeta=-1$.
\end{itemize}
\end{thm}

\begin{proof}
We have shown that the $L$-parameter $\phi$ of a simple supercuspidal representation $\pi^{\SO_{2n}}_{a,\zeta}$ is of the form $\phi=\phi_{0}\oplus\phi_{1}$, where
\begin{itemize}
\item
$\phi_{0}$ is an irreducible orthogonal representation of $W_{F}$ of dimension $2n-1$ and Swan exponent one, and
\item
$\phi_{1}$ is a tamely ramified character of $W_{F}$, hence equals the determinant character of $\phi_{0}$.
(Recall that $\phi_{1}$ is the trivial character if $\zeta=1$ and the unique nontrivial unramified quadratic character if $\zeta=-1$.)
\end{itemize}
Since the simple supercuspidal representation of $\GL_{2n-1}(F)$ corresponding to $\phi_{0}$ is self-dual, it is of the form $\pi^{\GL_{2n-1}}_{\mathbbm{1},a',\zeta'}$ (Section \ref{subsec:SSC-GL}).
Then, for any tamely ramified character $\tau$ of $F^{\times}$,
\[
\gamma(s,\pi^{\GL_{2n-1}}_{\mathbbm{1},a',\zeta'}\times\tau,\psi)
=
\tau(-1)^{2n-2}\cdot\tau(\varpi a'^{-1})\cdot\zeta'\cdot q^{\frac{1}{2}-s}.
\]
according to \cite[Corollary 3.12]{AL16}.
Therefore, by Proposition \ref{prop:gamma} (2) (a), we get
\[
\tau(-1)^{2n-2}\cdot\tau(\varpi a'^{-1})\cdot\zeta'\cdot q^{\frac{1}{2}-s}
=
\zeta\cdot\tau(a^{-1}\varpi)\cdot q^{\frac{1}{2}-s}.
\]
Since this identity holds for any tamely ramified character $\tau$, we conclude that $\zeta'=\zeta$ and $a'=a$.
\end{proof}

\subsubsection{The case of $\SO_{2n}$ with $p\neq2$}
We finally consider the case where $N=2n$ and $p\neq2$.

\begin{thm}\label{thm:main-2-b}
Let $\pi^{\SO_{2n}}_{\xi,\kappa,a,\zeta}$ be the simple supercuspidal representation of $\SO_{2n}(F)$ with $\xi\in\{\pm1\}$, $\kappa\in\{0,1\}$, $a\in k^{\times}$, and $\zeta\in\{\pm1\}$.
Then the $L$-parameter $\phi$ of a $\pi^{\SO_{2n}}_{\xi,\kappa,a,\zeta}$ is of the form $\phi=\phi_{0}\oplus\phi_{1}\oplus\phi_{2}$, where
\begin{itemize}
\item
$\phi_{1}$ is the unique unramified quadratic character of $F^{\times}$ satisfying $\phi_{1}(a^{-1}\varpi)=\zeta$,
\item
$\phi_{2}$ is the unique ramified quadratic character of $F^{\times}$ satisfying $\phi_{2}(a^{-1}\varpi)=\zeta\cdot\phi_{2}(-4\epsilon^{\kappa})$, and
\item
$\phi_{0}$ is an irreducible orthogonal representation of dimension $2n-2$, which is the $L$-parameter of the simple supercuspidal representation $\pi^{\GL_{2n-2}}_{\omega_{0},a',\zeta'}$ of $\GL_{2n-2}(F)$, where $\omega_{0}$ is the nontrivial quadratic character of $k^{\times}$, $a'=(-1)^{n}4a\epsilon^{\kappa}$, and $\zeta'=\xi\cdot\zeta\cdot q^{1/2}G(\phi_{2},\psi)^{-1}$.
\end{itemize}
\end{thm}

\begin{proof}
We have shown that the $L$-parameter $\phi$ of a simple supercuspidal representation $\pi^{\SO_{2n}}_{\xi,\kappa,a,\zeta}$ is of the form $\phi=\phi_{0}\oplus\phi_{1}\oplus\phi_{2}$, where
\begin{itemize}
\item
$\phi_{0}$ is an irreducible orthogonal representation of dimension $2n-2$ and Swan exponent one, and
\item
$\phi_{1}$ and $\phi_{2}$ are tamely ramified characters as described in Proposition \ref{prop:gamma} (2) (b).
\end{itemize}
We let $\pi^{\GL_{2n-2}}_{\omega',a',\zeta'}$ be the self-dual simple supercuspidal representation of $\GL_{2n-2}(F)$ corresponding to $\phi_{0}$.
Then, for any tamely ramified character $\tau$ of $F^{\times}$,
\[
\gamma(s,\pi^{\GL_{2n-2}}_{\omega',a',\zeta'}\times\tau,\psi)
=
\tau(-1)^{2n-3}\cdot\tau(\varpi a'^{-1})\cdot\zeta'\cdot q^{\frac{1}{2}-s}.
\]
according to \cite[Corollary 3.12]{AL16}.
Therefore, by Proposition \ref{prop:gamma} (2) (b), we get
\[
\tau(-1)^{2n-3}\cdot\tau(\varpi a'^{-1})\cdot\zeta'\cdot q^{\frac{1}{2}-s}
=
\xi\cdot\zeta\cdot\tau((-1)^{n+1}\varpi/4\epsilon^{\kappa}a)q^{1-s}G(\phi_{2},\psi)^{-1}.
\]
Since this identity holds for any tamely ramified character $\tau$, we conclude that $\zeta'=\xi\cdot\zeta\cdot q^{1/2}G(\phi_{2},\psi)^{-1}$ and $a'=(-1)^{n}4a\epsilon^{\kappa}$.
Finally, we note the determinant of $\phi$ is trivial, hence the determinant of $\phi_{0}$ is equal to (the inverse of) the product of $\phi_{1}$ and $\phi_{2}$.
By the description of $\phi_{1}$ and $\phi_{2}$ in Proposition \ref{prop:gamma} (2) (b), the product $\phi_{1}\phi_{2}$ is a nontrivial ramified quadratic character, hence so is the central character of $\pi^{\GL_{2n-2}}_{\omega',a',\zeta'}$.
This implies that $\omega'$ is the nontrivial quadratic character $\omega_{0}$ of $k^{\times}$.
\end{proof}


\appendix

\section{Several remarks on simple supercuspidal representations}

\subsection{Iwahori subgroups}\label{subsec:Iwahori}
In Section \ref{subsec:SSC-odd}, we explained that the Iwahori subgroup $I_{\SO_{2n+1}}$ of $\SO_{2n+1}$ can be thought of as matrices of the form
\[\label{matrix:Iwahori}
\begin{pmatrix}
\mcO^{\times}&&\multicolumn{1}{c:}{\mcO}&\multicolumn{1}{c:}{\mcO}&&&\\
 &\ddots&\multicolumn{1}{c:}{}&\multicolumn{1}{c:}{\vdots}&&\frac{1}{2}\mcO&\\
 \mfp&&\multicolumn{1}{c:}{\mcO^{\times}}&\multicolumn{1}{c:}{\mcO}&&&\\
 \cdashline{1-10}
 2\mfp&\cdots&\multicolumn{1}{c:}{2\mfp}&\multicolumn{1}{c:}{\mcO^{\times}}&\mcO&\cdots&\mcO\\
 \cdashline{1-10}
 &&\multicolumn{1}{c:}{}&\multicolumn{1}{c:}{2\mfp}&\mcO^{\times}&&\mcO\\
 &2\mfp&\multicolumn{1}{c:}{}&\multicolumn{1}{c:}{\vdots}&&\ddots&\\
 &&\multicolumn{1}{c:}{}&\multicolumn{1}{c:}{2\mfp}&\mfp&&\mcO^{\times}\\
\end{pmatrix}.
\tag{$\dagger$}
\]
We explain how this intuitive description of the Iwahori subgroup can be derived from \cite[Section 10]{BT72}.

In \cite[Section 10.1]{BT72}, the Bruhat--Tits theory is investigated thoroughly in the case of classical groups.
The description of Bruhat--Tits starts with taking the following data (\cite[(10.1.1)]{BT72}):
\begin{itemize}
\item
a field $K$ (not necessarily commutative),
\item
an involution $\sigma$ of $K$,
\item
a sign $\varepsilon\in K^{\times}$,
\item
a finite-dimensional right $K$-vector space $X$, and
\item
a $\sigma$-sesqui-linear form $f$ on $X$ satisfying
\begin{itemize}
\item
$f(y,x)=\varepsilon f(x,y)^{\sigma}$ for any $x,y\in X$ and
\item
$f(x,x)=0$ for any $x\in X$ when $(\sigma,\varepsilon)=(\id,-1)$.
\end{itemize}
\end{itemize}
We put $K_{\sigma,\varepsilon}:=\{t-\varepsilon t^{\sigma}\mid t\in K\}$ and associate a pseudo-quadratic form $q\colon X\rightarrow K/K_{\sigma,\varepsilon}$ to $f$, which satisfies
\begin{itemize}
\item
$q(xk)=k^{\sigma}q(x)k$ for any $k\in K$ and $x\in X$, and
\item
$q(x+y)=q(x)+q(y)+f(x,y)+K_{\sigma,\varepsilon}$ for any $x,y\in X$.
\end{itemize}
These data give rise to classical groups such as $\mathrm{Is}(f,q)$, which consists of isometries of $X$ with respect to $(f,q)$, or $\mathrm{Sim}(f,q)$, which consists of similitudes of $X$ with respect to $(f,q)$ (see \cite[(10.1.4-5)]{BT72}).

In order to realize the odd special orthogonal group, we choose $(K,\sigma,\varepsilon,X,f)$ as follows:
\begin{itemize}
\item
$K=F$,
\item
$\sigma=\id$,
\item
$\varepsilon=1$ (note that then $K_{\sigma,\varepsilon}=0$),
\item
$X=F^{\oplus 2n+1}$; we let $\{e_{-n},\ldots,e_{-1},e_{0},e_{1},\ldots,e_{n}\}$ be the canonical basis of $X$ and put $X_{i}:=Fe_{i}$),
\item
$f\colon X\times X\rightarrow F$ is the symmetric bilinear form satisfying
\begin{itemize}
\item
$f(e_{i},e_{j})=0$ if $i\neq-j$,
\item
$f(e_{i},e_{-i})=1$ for $i\in\{\pm1,\ldots,\pm n\}$,
\item
$f(e_{0},e_{0})=2$.
\end{itemize}
\end{itemize}
Then $G:=\mathrm{Is}(f,q)^{\circ}$ gives the odd special orthogonal group.
Here, we note that the matrix representation of the symmetric bilinear form $f$ with respect to the basis $\{e_{-n},\ldots,e_{-1},e_{0},e_{1},\ldots,e_{n}\}$ is given by
\[
J_{2n+1}^{\BT}
:=
\begin{pmatrix}
&&J'_{n}\\
&2&\\
J'_{n}&&
\end{pmatrix},
\]
where $J'_{n}$ denotes the anti-diagonal matrix of size $n$ whose anti-diagonal entries are given by $1$.
In particular, when regarded as matrices of size $2n+1$ via the basis $\{e_{-n},\ldots,e_{-1},e_{0},e_{1},\ldots,e_{n}\}$, the group $G$ is given by
\[
\SO(J_{2n+1}^{\BT})
:=\{g\in\SL_{2n+1} \mid {}^{t}\!gJ_{2n+1}^{\BT}g=J_{2n+1}^{\BT}\}.
\]

In the body of the paper we preferred to work with the bilinear form $f$ represented by $J_{2n+1}$. However in this section, it is more convenient to use the bilinear form $f$ represented by $J_{2n+1}^{\BT}$, because then $1\in q(X_{0})$, which makes the description of Bruhat--Tits simpler (cf.\ \cite[Remarque (10.1.3)]{BT72}). We note that the translation between $\SO(J_{2n+1}^{\BT})$ and $\SO_{2n+1}$ is given as follows.
Recall that
\[
\SO_{2n+1}:=\{g\in\SL_{2n+1} \mid {}^{t}\!gJ_{2n+1}g=J_{2n+1}\}.
\]
For example, by letting $X$ be the diagonal matrix
\[
\diag(\underbrace{(-1)^{n}2,(-1)^{n-1}2,...,-2}_{n},\underbrace{1,...,1}_{n+1})
\]
(the first $n$ entries of $X$ are given by $2$ and $-2$ alternatively so that the $n$-th entry is given by $-2$), $(-1)^{n}2\cdot J_{2n+1}={}^{t}XJ_{2n+1}^{\BT}X$.
Thus the conjugation by $X$ gives a group isomorphism between $\SO_{2n+1}$ and $\SO(J_{2n+1}^{\BT})$:
\[
\SO_{2n+1}\xrightarrow{\cong}\SO(J_{2n+1}^{\BT})\colon g\mapsto XgX^{-1}.
\]

It is not difficult to see that the matrices of the form \eqref{matrix:Iwahori} are mapped to matrices of the form
\[\label{matrix:Iwahori-BT}
\begin{pmatrix}
\mcO^{\times}&&\multicolumn{1}{c:}{\mcO}&\multicolumn{1}{c:}{2\mcO}&&&\\
 &\ddots&\multicolumn{1}{c:}{}&\multicolumn{1}{c:}{\vdots}&&\mcO&\\
 \mfp&&\multicolumn{1}{c:}{\mcO^{\times}}&\multicolumn{1}{c:}{2\mcO}&&&\\
 \cdashline{1-10}
 \mfp&\cdots&\multicolumn{1}{c:}{\mfp}&\multicolumn{1}{c:}{\mcO^{\times}}&\mcO&\cdots&\mcO\\
 \cdashline{1-10}
 &&\multicolumn{1}{c:}{}&\multicolumn{1}{c:}{2\mfp}&\mcO^{\times}&&\mcO\\
 &\mfp&\multicolumn{1}{c:}{}&\multicolumn{1}{c:}{\vdots}&&\ddots&\\
 &&\multicolumn{1}{c:}{}&\multicolumn{1}{c:}{2\mfp}&\mfp&&\mcO^{\times}\\
\end{pmatrix}
\tag{$\dagger^{\BT}$}
\]
under the conjugation by $X$.
Thus, in the following, let us check that the matrices of $G$ of the form \eqref{matrix:Iwahori-BT} constitute an Iwahori subgroup.

We let $T$ be the subgroup of $G$ given by
\[
T:=\{t\in G \mid \text{$t(X_{i})\subset X_{i}$ for any $i\in\{-n,\ldots,0,\ldots,n\}$}\},
\]
which is a maximal torus.
To describe the root system of $G$ with respect to $T$, we introduce a real vector space $V^{\ast}:=\R^{n}$ equipped with a canonical basis $\{a_{1},\ldots,a_{n}\}$ and the standard inner product.
We put $a_{-i}:=-a_{i}$ for $i\in\{1,\ldots,n\}$ and $a_{0}:=0$.
We also put $a_{ij}:=a_{i}+a_{j}$ for $i,j\in\{-n,\ldots,0,\ldots,n\}$.
Then the root system of $G$ is given by the set
\[
\Phi
:=\{a_{i}\mid i\in\{\pm1,\ldots,\pm n\}\}\cup \{a_{ij}\mid i,j\in\{\pm1,\ldots,\pm n\}, i\neq j\}.
\]
For each root $a\in\Phi$ of $G$, the corresponding root subgroup $U_{a}$ of $G$ is given as follows (\cite[(10.1.2)]{BT72}):
\begin{enumerate}
\item
When $a=a_{i}$ for $i\in\{\pm1,\ldots,\pm n\}$, we obtain
\[
U_{a_{i}}=\{u_{i}(z) \mid z\in X_{0}\},
\]
where $u_{i}(z)\in \Hom_{F}(X,X)$ is the isometry defined by
\[
\begin{cases}
e_{0}\mapsto e_{0}-f(z,e_{0})e_{-i}\\
e_{i}\mapsto e_{i}+z-q(z)e_{-i}\\
e_{j}\mapsto e_{j} & \text{for $j\in\{\pm1,\ldots,\pm n\}\smallsetminus\{i\}$.}
\end{cases}
\]
Note that, if we write $z=xe_{0}$ with $x'\in F$, then
\[
\begin{cases}
e_{0}-f(z,e_{0})e_{-i}=e_{0}-2xe_{-i}\\
e_{i}+z-q(z)e_{-i}=e_{i}+xe_{0}-x^{2}e_{-i}
\end{cases}
\]
by our choice of $f$ and $q$.
\item
When $a=a_{ij}$ for $i,j\in\{\pm1,\ldots,\pm n\}$ satisfying $i\neq j$,
\[
U_{a_{ij}}=\{u_{ij}(x) \mid x\in F\},
\]
where $u_{ij}(x)\in \Hom_{F}(X,X)$ is the isometry defined by
\[
\begin{cases}
e_{0}\mapsto e_{0}\\
e_{i}\mapsto e_{i}+xe_{-j}\\
e_{j}\mapsto e_{j}-xe_{-i}\\
e_{k}\mapsto e_{k} & \text{for $j\in\{\pm1,\ldots,\pm n\}\smallsetminus\{i,j\}$.}
\end{cases}
\]
\end{enumerate}
We consider the function $\varphi_{a}\colon U_{a}\rightarrow\R\cup\{\infty\}$ for each root $a\in\Phi$ as follows (\cite[(10.1.13)]{BT72}):
\begin{enumerate}
\item
When $a=a_{i}$ for $i\in\{\pm1,\ldots,\pm n\}$, we put
\[
\varphi_{a_{i}}(u_{i}(z)):=\frac{1}{2}\val_{F}(q(z)).
\]
Note that, if we write $z=xe_{0}$ with $x'\in F$, then $\varphi_{a_{i}}(u_{i}(z))=\val_{F}(x)$ by our choice of $q$.
\item
When $a=a_{ij}$ for $i,j\in\{\pm1,\ldots,\pm n\}$ satisfying $i\neq j$, we put
\[
\varphi_{a_{ij}}(u_{ij}(x)):=\val_{F}(x).
\]
\end{enumerate}
Then $\{\varphi_{a}\}_{a\in\Phi}$ defines a valuation of root datum of $G$ (\cite[Th\'eor\`eme(10.1.15)]{BT72}), hence determines a point $\bfo$ of an apartment $\mcA$ of the Bruhat--Tits building of $G$.

We put
\[
T_{0}:=\{t\in T \mid \text{$t(e_{i})\in \mcO e_{i}$ for any $i\in\{-n,\ldots,0,\ldots,n\}$}\}.
\]
If we choose a subset $\Delta\subset\Phi$ of simple roots, then the set $\Phi^{+}$ (resp.\ $\Phi^{-}$) of positive roots (resp.\ negative roots) is determined.
Accordingly, we get an Iwahori subgroup $I_{\Delta}$ of $G$ given by
\[
I_{\Delta}
=
\langle T_{0}, U_{a}(\mcO), U_{b}(\mfp) \mid a\in\Phi^{+}, b\in\Phi^{-} \rangle,
\]
where we put $U_{a}(\mcO):=\varphi_{a}^{-1}([0,\infty])\subset U_{a}$ and $U_{b}(\mfp):=\varphi_{b}^{-1}([1,\infty])\subset U_{b}$.

Now let us consider the matrix representation of the Iwahori subgroup.
For an element $g\in \End_{F}(X)$, we let $c_{ij}(g)$ be the element of $\Hom_{F}(X_{j},X_{i})$ given by composing $g$ with the injection $X_{j}\hookrightarrow X$ and the projection $X\twoheadrightarrow X_{i}$. We then write
\[
c_{ij}(g)(e_{j})=c'_{ij}(g)\cdot e_{i},\qquad c'_{ij}(g)\in F,\qquad \forall i,j\in\{-n,\ldots,0,\ldots,n\}.
\]
In other words, the matrix representation of $g$ with respect to the ordered basis $\{e_{-n},\ldots,e_{-1},e_{0},e_{1},\ldots,e_{n}\}$ is given by $(c'_{ij}(g))_{ij}$.
Then each subspace $\Hom_{F}(X_{j},X_{i})$ of $\End_{F}(X)$ is regarded as a root space for the root $a_{-i,j}$ with respect to the maximal torus $T$.
Therefore, by choosing $\Delta$ so that the upper-triangular part of the matrix representation of $\End_{F}(X)$ corresponds to the root spaces for $\Phi^{+}$ (i.e., $a_{-i,j}\in\Phi^{+}$ if and only if $i<j$), we see that any element of the associated Iwahori subgroup $I_{\Delta}$ has the matrix representation of the form \eqref{matrix:Iwahori-BT}.

Conversely, we can also see that any element of $G$ of the form \eqref{matrix:Iwahori-BT} indeed belongs to $I_{\Delta}$ by using the following proposition.

\begin{prop}[{\cite[(10.1.32)]{BT72}}]\label{prop:BT72}
We take a pair $(\x,E)$ of
\begin{itemize}
\item
a point $\x$ of the apartment $\mcA$ and
\item
a vectorial facet $E$ of the root system of $G$
\end{itemize}
and consider the subgroups $P_{\x,E}$ (``parahoric subgroup'') and $\hat{P}_{\x,E}$ associated to $(\x,E)$ according to (7.1.8) and (7.2.4) of \cite{BT72}.
Then $\hat{P}_{\x,E}$ consists of elements $g\in G$ satisfying the inequality
\[\label{BT72}
\omega_{ij}(c_{ij}(g))-\frac{1}{2}\val_{F}(c(g))\geq a_{i,-j}(\x)
\]
for any $i,j\in\{-n,\ldots,0,\ldots,n\}$, where the equality does not hold when $(a_{i,-j})(E)\subset\R_{>0}$.
Here, the meaning of the symbols used in the above inequality are as follows:
\begin{itemize}
\item
$c(g)$ denotes the similitude of $g$ (hence we always have $c(g)=1$ under our choice of $G$);
\item
for any $i,j\in\{-n,\ldots,0,\ldots,n\}$, we define $\omega_{ij}\colon \Hom_{F}(X_{j},X_{i})\rightarrow \R\cup\{\pm\infty\}$ by
\[
\omega_{ij}(\alpha)
=
\inf_{x_{j}\in X_{j}}\{\omega_{i}(\alpha(x_{j}))-\omega_{j}(x_{j})\}.
\]
for $\alpha\in\Hom_{F}(X_{j},X_{i})$.
Here, for any $i\in\{-n,\ldots,0,\ldots,n\}$, we define $\omega_{i}\colon X_{i}\rightarrow \R\cup\{\infty\}$ by
\[
\omega_{i}(x\cdot e_{i})
=
\begin{cases}
\frac{1}{2}\val_{F}(q(x\cdot e_{0})) & i=0,\\
\val_{F}(x) & i\neq0,
\end{cases}
\]
for $x\in F$ (hence $x\cdot e_{i}\in Fe_{i}=X_{i}$).
\end{itemize}
\end{prop}

Let us explain how this proposition can be utilized.
By choosing $\x$ to be the point $\bfo$ of the apartment introduced above and $E$ to be the dominant chamber $E_{\Delta}$ corresponding to $\Delta$ (and also $\x$), the group $P_{\x,E_{\Delta}}$ realizes the Iwahori subgroup $I_{\Delta}$ and $P_{\x,E_{\Delta}}=\hat{P}_{\x,E_{\Delta}}$.
Since $a_{i,-j}(E_{\Delta})\subset\R_{>0}$ if and only if $i>j$ by our choice of $E_{\Delta}$, Proposition \ref{prop:BT72} implies that $I_{\Delta}$ consists of the elements $g\in G(F)$ satisfying the inequality
\[
\omega_{ij}(c_{ij}(g))
\begin{cases}
\geq 0 & \text{if $i\leq j$,}\\
> 0 & \text{if $i>j$,}
\end{cases}
\]
for any $i,j\in\{-n,\ldots,0,\ldots,n\}$.
We note that, by our choice of $f$, we see that $q(x\cdot e_{0})=x^{2}$ for any $x\in F$.
Thus we simply have $\omega_{i}(x\cdot e_{i})=\val_{F}(x)$ for any $i\in\{-n,\ldots,0,\ldots,n\}$.
This implies that $\omega_{ij}(c_{ij}(g))=\val_{F}(c'_{ij}(g))$.
In other words, $I_{\Delta}$ exactly consists of elements of $G$ whose matrix representations with respect to the basis $\{e_{-n},\ldots,e_{-1},e_{0},e_{1},\ldots,e_{n}\}$ are of the form
\[\label{matrix:Iwahori2}
\begin{pmatrix}
 \mcO^{\times}&&\mcO\\
 &\ddots&\\
 \mfp&&\mcO^{\times}
\end{pmatrix}.
\tag{$\dagger'$}
\]
Thus it is enough to check that any element of $G$ of the form \eqref{matrix:Iwahori2} in fact belongs to \eqref{matrix:Iwahori-BT}.
Let $g\in G$ be an element of $G$ of the form \eqref{matrix:Iwahori2}.
We write
\[
g=
\begin{pmatrix}
g_{11}&g_{12}&g_{13}\\
g_{21}&g_{22}&g_{23}\\
g_{31}&g_{32}&g_{33}
\end{pmatrix},
\]
where $g_{11},g_{13}, g_{31},g_{33}\in M_{n,n}(F)$, $g_{12},g_{32}\in M_{n,1}(F)$, and $g_{21},g_{23}\in M_{1,n}(F)$.
Our task is to show that $\frac{1}{2}g_{12}\in M_{n,1}(\mcO)$ and $\frac{1}{2}g_{32}\in M_{n,1}(\mfp)$.
Since $g$ is an element of $G=\SO(J^{\BT}_{2n+1})$, ${}^{t}gJ^{\BT}_{2n+1}g=J^{\BT}_{2n+1}$, equivalently, $J^{\BT,-1}_{2n+1}{}^{t}gJ^{\BT}_{2n+1}=g^{-1}$.
Observe that
\begin{align*}
J^{\BT,-1}_{2n+1}{}^{t}gJ^{\BT}_{2n+1}
&=
\begin{pmatrix}
&&J'_{n}\\
&\frac{1}{2}&\\
J'_{n}&&
\end{pmatrix}
\begin{pmatrix}
{}^{t}g_{11}&{}^{t}g_{21}&{}^{t}g_{31}\\
{}^{t}g_{12}&{}^{t}g_{22}&{}^{t}g_{32}\\
{}^{t}g_{13}&{}^{t}g_{23}&{}^{t}g_{33}
\end{pmatrix}
\begin{pmatrix}
&&J'_{n}\\
&2&\\
J'_{n}&&
\end{pmatrix}\\
&=
\begin{pmatrix}
&&J'_{n}\\
&\frac{1}{2}&\\
J'_{n}&&
\end{pmatrix}
\begin{pmatrix}
{}^{t}g_{31}J'_{n}&2{}^{t}g_{21}&{}^{t}g_{11}J'_{n}\\
{}^{t}g_{32}J'_{n}&2{}^{t}g_{22}&{}^{t}g_{12}J'_{n}\\
{}^{t}g_{33}J'_{n}&2{}^{t}g_{23}&{}^{t}g_{13}J'_{n}
\end{pmatrix}\\
&=
\begin{pmatrix}
J'_{n}{}^{t}g_{33}J'_{n}&2J'_{n}{}^{t}g_{23}&J'_{n}{}^{t}g_{13}J'_{n}\\
\frac{1}{2}{}^{t}g_{32}J'_{n}&{}^{t}g_{22}&\frac{1}{2}{}^{t}g_{12}J'_{n}\\
J'_{n}{}^{t}g_{31}J'_{n}&2J'_{n}{}^{t}g_{21}&J'_{n}{}^{t}g_{11}J'_{n}
\end{pmatrix}.
\end{align*}
Since this equals $g^{-1}$, which again belongs to \eqref{matrix:Iwahori2}, we necessarily have $\frac{1}{2}{}^{t}g_{32}J'_{n} \in M_{1,n}(\mfp)$ and $\frac{1}{2}{}^{t}g_{12}J'_{n} \in M_{1,n}(\mcO)$.
Hence we get $\frac{1}{2}g_{12}\in M_{n,1}(\mcO)$ and $\frac{1}{2}g_{32}\in M_{n,1}(\mfp)$.

Thus we conclude that Iwahori subgroup $I_{\Delta}$ of $G=\SO(J_{2n+1}^{\BT})(F)$ exactly consists of elements of the form \eqref{matrix:Iwahori-BT}.

We finally give a comment on the Iwahori subgroup of $\SO_{2n}(F)$.
In this case, we choose the data $(K,\sigma,\varepsilon,X,f)$ as follows:
\begin{itemize}
\item
$K=F$,
\item
$\sigma=\id$,
\item
$\varepsilon=1$ (note that then $K_{\sigma,\varepsilon}=0$),
\item
$X=F^{\oplus 2n}$; we let $\{e_{-n},\ldots,e_{-1},e_{1},\ldots,e_{n}\}$ be the canonical basis of $X$ and put $X_{i}:=Fe_{i}$),
\item
$f\colon X\times X\rightarrow F$ is the symmetric bilinear form satisfying
\begin{itemize}
\item
$f(e_{i},e_{j})=0$ if $i\neq-j$,
\item
$f(e_{i},e_{-i})=1$ for $i\in\{\pm1,\ldots,\pm n\}$.
\end{itemize}
\end{itemize}
Then we can check that the matrices of the form
\[
\begin{pmatrix}
\mcO^{\times}&&\multicolumn{1}{c:}{\mcO}&\mcO&\multicolumn{1}{c:}{\mcO}&&&\\
 &\ddots&\multicolumn{1}{c:}{}&\vdots&\multicolumn{1}{c:}{\vdots}&&\mcO&\\
 \mfp&&\multicolumn{1}{c:}{\mcO^{\times}}&\mcO&\multicolumn{1}{c:}{\mcO}&&&\\
 \cdashline{1-10}
 \mfp&\cdots&\multicolumn{1}{c:}{\mfp}&\mcO^{\times}&\multicolumn{1}{c:}{\mfp}&\mcO&\cdots&\mcO\\
\mfp&\cdots&\multicolumn{1}{c:}{\mfp}&\mfp&\multicolumn{1}{c:}{\mcO^{\times}}&\mcO&\cdots&\mcO\\
 \cdashline{1-10}
 &&\multicolumn{1}{c:}{}&\mfp&\multicolumn{1}{c:}{\mfp}&\mcO^{\times}&&\mcO\\
 &\mfp&\multicolumn{1}{c:}{}&\vdots&\multicolumn{1}{c:}{\vdots}&&\ddots&\\
 &&\multicolumn{1}{c:}{}&\mfp&\multicolumn{1}{c:}{\mfp}&\mfp&&\mcO^{\times}\\
\end{pmatrix}
\]
constitute an Iwahori subgroup in a similar manner to above.
Note that, compared to the case of $\SO_{2n+1}$, every argument is even simpler because the factor $X_{0}$ does not exist.
For example, there is no root of the form $a_{i}$ in this case.
This explains why ``$2$'' does not appear in the above matrix description in contrast to \eqref{matrix:Iwahori} or \eqref{matrix:Iwahori-BT}.

\subsection{Another parametrization of simple supercuspidal representations}\label{subsec:parametrization}

The results of \cite{AL16,Adr16,AK21} (and also Section \ref{sec:AK}), which are needed for Proposition \ref{prop:gamma}, are stated based on a different parametrization of simple supercuspidal representations.
For this reason, in this section, we compare our parametrization of simple supercuspidal representations (Section \ref{sec:SSC}) with those of \cite{AL16,Adr16,AK21}.
The main difference is that the choice of a uniformizer $\varpi'$ can vary in the parametrizations in \cite{AL16,Adr16,AK21} while a uniformizer $\varpi$ is fixed and the parameter ``$a$'' can vary in our parametrization.

\subsubsection{The case of $\GL_{N}$}\label{subsubsec:GL}
Let us first look at the case of $\GL_{N}$.
In \cite[Section 3.1]{AL16}, a simple supercuspidal representation $\sigma(\varpi',\zeta,\omega)$ is associated to each tuple consisting of a uniformizer $\varpi'$, a tamely ramified character $\omega$ of $F^{\times}$, and an $n$-th root $\zeta$ of $\omega(\varpi')$.
By putting $a:=\varpi\varpi^{\prime-1}\in k^{\times}$, we obtain
\[
\pi^{\GL_{N}}_{\omega|_{k^{\times}},a,\zeta}
\cong
\sigma(\varpi',\zeta,\omega).
\]

\subsubsection{The case of $\SO_{2n+1}$}\label{subsubsec:SO-odd}
We next compare the parametrization given in Section \ref{subsec:SSC-odd} with the one of \cite{Adr16}.
Firstly, we must be careful that the odd special orthogonal group is realized as $\SO(J^{\BT}_{2n+1})$ in \cite{Adr16}.
Let $I_{\SO(J^{\BT}_{2n+1})}$ be the Iwahori subgroup $I_{\Delta}$ of $\SO(J^{\BT}_{2n+1})(F)$ as described in Section \ref{subsec:Iwahori} and $I^{+}_{\SO(J^{\BT}_{2n+1})}$ its pro-unipotent radical.
Secondly, we must be also careful that the construction of simple supercuspidal representations given in \cite{Adr16} contains a minor error.
Let us describe the error and how it can be fixed.

In \cite[page 205]{Adr16}, a simple supercuspidal representation $\pi^{\zeta}_{\chi}$ of $\SO(J_{2n+1}^{\BT})(F)$ is associated to each pair $(\varpi',\zeta)$ of a uniformizer $\varpi'$ of $F$ and a sign $\zeta\in\{\pm1\}$.
Let us write $\pi^{\zeta}_{\chi}[\varpi']$ for this simple supercuspidal representation $\pi^{\zeta}_{\chi}$ in order to emphasize that it depends on the choice of a uniformizer $\varpi'$.
This simple supercuspidal representation $\pi^{\zeta}_{\chi}[\varpi']$ is associated to a character $\chi$ of $I^{+}_{\SO(J^{\BT}_{2n+1})}$ given by
\[
\chi\colon
g=(g_{ij})
\mapsto
\psi\bigl(\overline{g_{12}}+\cdots +\overline{g_{n-1,n}}+\overline{g_{n,n+1}}+\overline{g_{2n,1}\varpi'^{-1}}\bigr).
\]
The problem is that $\chi$ is not affine generic when $p=2$.
This is because the $(n,n+1)$-entry of any element $g\in I^{+}_{\SO(J^{\BT}_{2n+1})}$ always belongs to $2\mcO$ (see the description \eqref{matrix:Iwahori-BT}).

This issue can be fixed by modifying the definition of $\chi$ as follows (let $\chi'$ denote the modified character):
\[
\chi'\colon
g=(g_{ij})
\mapsto
\psi\bigl(\overline{g_{12}}+\cdots +\overline{g_{n-1,n}}+\overline{g_{n,n+1}2^{-1}}+\overline{g_{2n,1}\varpi'^{-1}}\bigr).
\]
Then $\chi'$ is affine generic for any $p$ including $p=2$, hence we can produce a simple supercuspidal representation of $\SO(J^{\BT}_{2n+1})(F)$ by using $\chi'$ instead of $\chi$.
We write $\pi^{\zeta}_{\chi'}[\varpi']$ for this simple supercuspidal representation.
We remark that when $p\neq2$ (so that the construction of $\pi^{\zeta}_{\chi}[\varpi']$ makes sense), $\pi^{\zeta}_{\chi'}[\varpi']\cong\pi^{\zeta}_{\chi}[4\varpi']$.

Now let us go back to comparing the two parametrizations of simple supercuspidal representations.
By putting $a:=\varpi\varpi^{\prime-1}\in k^{\times}$, we obtain
\[
\pi^{\SO_{2n+1}}_{(-1)^{n+1}a,\zeta}
\cong
\pi^{\zeta}_{\chi'}[\varpi']
\]
under the identification between $\SO_{2n+1}$ and $\SO(J^{\BT}_{2n+1})$ via the conjugation by $X$ (see Section \ref{subsec:Iwahori}).

\subsubsection{The case of $\SO_{2n}$}\label{subsubsec:SO-even}
We finally consider the case of $\SO_{2n}$.
In \cite[Section 3]{AK21}, a simple supercuspidal representation $\pi^{\omega}_{\alpha}$ is associated to each tuple consisting of a uniformizer $\varpi'$ of $F$, $\alpha\in k^{\times}$, a character $\omega$ of the center of $\SO_{2n}(F)$, and a sign $\zeta\in\{\pm1\}$.
Similarly to the previous case, let us write $\pi^{\omega}_{\alpha}[\varpi',\zeta]$ for the associated simple supercuspidal representation of \cite{AK21}.
(Note that $\pi^{\omega}_{\alpha}[\varpi',\zeta]$ is denoted simply by ``$\pi$'' in Section \ref{sec:AK}).
If we let $\alpha$ be $\epsilon^{\kappa}$ for $\kappa\in\{0,1\}$ and put $\xi:=\omega(-1)$ and $a:=\varpi\varpi^{\prime-1}\in k^{\times}$, then
\[
\pi^{\SO_{2n}}_{\xi,\kappa,a,\zeta}
\cong
\pi^{\omega}_{\alpha}[\varpi',\zeta].
\]

\subsection{Comparison of our approach with others}\label{subsec:comparison}

We remark that our main results in the case where $p$ is odd (Theorems \ref{thm:main-1} and \ref{thm:main-2-b}) are not new:
\begin{itemize}
\item
when $N=2n+1$, the result of the same type has been obtained in \cite{Adr16} and \cite{Oi19-AJM};
\item
when $N=2n$, the result of the same type has been obtained in \cite{Oi24}.
\end{itemize}
In this section, we verify the consistency of Theorems \ref{thm:main-1} and \ref{thm:main-2-b} with those according to the dictionary given in Section \ref{subsec:parametrization}.

We note in passing that the results of \cite{Oi19-AJM,Oi24} (as well as \cite{Oi19-PRIMS}) were reproduced by Tam as part of his work \cite{Tam23} on epipelagic representations.

\subsubsection{The case of $\SO_{2n+1}$}

In \cite[Corollary 8.4]{Adr16} (with a modification explained in Remark \ref{rem:Adr16}), it is proved that the endoscopic lift of $\pi^{\zeta}_{\chi'}[\varpi']$ from $\SO_{2n+1}$ to $\GL_{2n}$ is given by $\sigma((-1)^{n+1}\varpi',\zeta,\mathbbm{1})$ when $p$ is sufficiently large (or, more generally, provided that the $L$-parameter of $\pi^{\zeta}_{\chi'}[\varpi']$ is irreducible).

By Sections \ref{subsubsec:GL} and \ref{subsubsec:SO-odd}, this amounts to saying that the endoscopic lift of $\pi^{\SO_{2n+1}}_{(-1)^{n+1}a,\zeta}$ is given by $\pi^{\GL_{2n}}_{\mathbbm{1},(-1)^{n+1}a,\zeta}$ by putting $a:=\varpi\varpi^{\prime-1}$.

On the other hand, in \cite[Theorem 5.15]{Oi19-AJM}, it is proved that the endoscopic lift of the simple supercuspidal representation of $\SO_{2n+1}(F)$ denoted by ``$\pi'_{a,\zeta}$'' is given by $\pi^{\GL_{2n}}_{\mathbbm{1},2a,\zeta}$ for any $a\in k^{\times}$ when $p$ is odd.
Recall that $\pi^{\SO_{2n+1}}_{2a,\zeta}$ in this paper is equal to $\pi'_{a,\zeta}$ in \cite{Oi19-AJM}; see Remark \ref{rem:ssc-odd-2}.

Thus the results of \cite{Adr16} and \cite{Oi19-AJM} are consistent with Theorem \ref{thm:main-1}.

\subsubsection{The case of split $\SO_{2n}$}
In \cite[Theorem 8.8]{Oi24}, it is proved that the endoscopic lift of $\pi^{\SO_{2n}}_{\xi,\kappa,a,\zeta}$ to $\GL_{2n}$ is given by
the parabolic induction of
\[
\begin{cases}
\pi_{\omega_{0},a',\eta}^{\GL_{2n-2}}\boxtimes\omega_{\omega_{0},a',\eta}^{\GL_{2n-2}}\boxtimes\mathbbm{1} & \text{if $\zeta=1$},\\
\pi_{\omega_{0},a',\zeta\eta}^{\GL_{2n-2}}\boxtimes\omega_{\omega_{0},a',\zeta\eta}^{\GL_{2n-2}}\cdot\mu_{\ur}\boxtimes\mu_{\ur} & \text{if $\zeta=-1$}
\end{cases}
\]
under the assumption that $p\neq2$.
Here, $\boxtimes$ denotes the external tensor product, $\mu_{\ur}$ is the unique nontrivial quadratic unramified character of $F^{\times}$, $\omega_0$ is the nontrivial quadratic character of $k^{\times}$, $\omega_{\omega_{0},a',\pm\eta}^{\GL_{2n-2}}$ is the central character of $\pi_{\omega_{0},a',\pm\eta}^{\GL_{2n-2}}$,
\[
\eta=q^{-\frac{1}{2}}G(\omega_{0},\psi)\omega_{0}(-1)\xi
\quad\text{and}\quad
a'=(-1)^{n}4a\epsilon^{\kappa},
\]
where $G(\omega_{0},\psi)$ denotes the Gauss sum.
Since $\pi_{\omega_{0},a',\pm\eta}^{\GL_{2n-2}}$ is self-dual, the character $\omega_{\omega_{0},a',\pm\eta}^{\GL_{2n-2}}$ must be quadratic.
Moreover, as the restriction of $\omega_{\omega_{0},a',\pm\eta}^{\GL_{2n-2}}$ to $\mcO^{\times}$ is the lift of $\omega_{0}$, $\omega_{\omega_{0},a',\pm\eta}^{\GL_{2n-2}}$ is a ramified character.

Let us check that this description is consistent with Theorem \ref{thm:main-2-b}.

Firstly, by the condition $\phi_{1}(a^{-1}\varpi)=\zeta$, the character $\phi_{1}$ equals $\mathbbm{1}$ or $\mu_{\ur}$ according to $\zeta=1$ or $\zeta=-1$.
Secondly, we check that $\phi_{2}$ is equal to $\omega_{\omega_{0},a',\zeta\eta}^{\GL_{2n-2}}$ or $\omega_{\omega_{0},a',\zeta\eta}^{\GL_{2n-2}}\cdot\mu_{\ur}$ according to $\zeta=1$ or $\zeta=-1$.
For this, by the characterization of $\phi_{2}$, it is enough to check that $\omega_{\omega_{0},a',\zeta\eta}^{\GL_{2n-2}}(a^{-1}\varpi)=\omega_{\omega_{0},a',\zeta\eta}^{\GL_{2n-2}}(-4\epsilon^{\kappa})$ (note that $\omega_{\omega_{0},a',\zeta\eta}^{\GL_{2n-2}}$ is a ramified quadratic character).
Since
\[
\omega_{\omega_{0},a',\zeta\eta}^{\GL_{2n-2}}(a'^{-1}\varpi)
=(\zeta\eta)^{2n-2}
=(q^{-\frac{1}{2}}G(\omega_{0},\psi))^{2n-2},
\]
we get
\begin{align*}
\omega_{\omega_{0},a',\zeta\eta}^{\GL_{2n-2}}(a^{-1}\varpi)
&=\omega_{0}(a^{-1}a')\cdot\omega_{\omega_{0},a',\zeta\eta}^{\GL_{2n-2}}(a'^{-1}\varpi)\\
&=\omega_{0}((-1)^{n}4\epsilon^{\kappa})\cdot(q^{-\frac{1}{2}}G(\omega_{0},\psi))^{2n-2}.
\end{align*}
By noting that $q^{-1}\cdot G(\omega_{0},\psi)^{2}=\omega_{0}(-1)$ (see \cite[Lemma A.5 (1)]{Oi24}), we get
\[
\omega_{\omega_{0},a',\zeta\eta}^{\GL_{2n-2}}(a^{-1}\varpi)
=
\omega_{0}(-4\epsilon^{\kappa}).
\]
On the other hand, as the restriction of $\omega_{\omega_{0},a',\zeta\eta}^{\GL_{2n-2}}$ to $\mcO^{\times}$ is the unique nontrivial quadratic character, we obtain $\omega_{\omega_{0},a',\zeta\eta}^{\GL_{2n-2}}(-4\epsilon^{\kappa})=\omega_{0}(-4\epsilon^{\kappa})$.
Finally, let us consider $\phi_{0}$; our task is to show that $\zeta\eta$ is equal to $\zeta'$ as in Theorem \ref{thm:main-2-b} (i.e., $\zeta'=\xi\cdot\zeta\cdot q^{1/2}G(\phi_{2},\psi)^{-1}$).
This follows from the identity $G(\phi_{2},\psi)^{2}=\phi_{2}(-1)q$ (see \cite[(23.6.3)]{BH06}) by noting that $\phi_{2}(-1)=\omega_{0}(-1)$.

\section{On lifting from classical groups to \texorpdfstring{$\GL_N$}{GL(N)}}\label{sec:lift}

\vspace{3mm}
\noindent
\ref{sec:lift}.1.\ \quad
Let $p$ be a prime number and $F$ a finite extension of the field of $p$-adic numbers.
(Thus, in particular, the characteristic of $F$ is $0$.)
Choose an algebraic closure of $F$ and let $W_F$ be the corresponding Weil group.
Let $\bfG$ be a split classical group, say over $\Z$, either $\Sp_{2n}$, $\SO_{2n+1}$ or $\SO_{2n}$ for a positive integer $n$.
Let $\hat{\bfG}$ be the (complex) dual group of $\bfG$ and $N$ the dimension of its standard representation, so that $\hat{\bfG}$ is $\SO_{2n+1}(\C)$, $\Sp_{2n}(\C)$, $\SO_{2n}(\C)$ and $N=2n+1$, $2n$, $2n$ accordingly.
Let $\pi$ be a smooth irreducible supercuspidal representation of $\bfG(F)$, and $\phi$ its $L$-parameter, a conjugacy class of morphisms of $W_F \times \SL_{2}(\C)$ to $\hat{\bfG}$ as given by Arthur (\cite[Theorem 1.5.2]{Art13}).
Composing with the standard representation of $\hat{\bfG}$, we get a parameter for $\GL_{N}(F)$, and a corresponding isomorphism class $\Pi$ of smooth irreducible representations of $\GL_{N}(F)$, sometimes called an endoscopic lift of $\pi$.
Now we fix Whittaker data of $\GL_{N}$ and $\bfG$ and assume that $\pi$ is generic with respect to the fixed Whittaker datum of $\bfG$.
Then Cogdell et al.\ \cite{CKPSS04} also associate to $\pi$ an isomorphism class $\Pi'$ of smooth irreducible representations of $\GL_{N}(F)$.
The goal of the present appendix is to prove that $\Pi=\Pi'$.
As a consequence, we deduce that for each positive integer $r$ and each smooth irreducible generic representation $\tau$ of $\GL_{r}(F)$, the Rankin--Selberg $\gamma$-factor $\gamma(s,\pi\times\tau,\psi)$ is equal to the Rankin--Selberg $\gamma$-factor $\gamma(s,\Pi\times\tau,\psi)$, for any choice of nontrivial additive character $\psi$ of $F$.
That is used in the main text when $\pi$ is a simple supercuspidal representation, to be able to apply the computation of the $\gamma$-factors $\gamma(s,\pi\times\tau,\psi)$ when $\tau$ is a tame character of $F^{\times}$.

Here and throughout this appendix, by the Rankin--Selberg $\gamma$-factor for representations of $\bfG(F)\times\GL_r(F)$ we mean the $\gamma$-factor defined in \cite{Kap15}, which is valid only for certain choices of Whittaker data for $\bfG$ (given explicitly in \textit{ibid.}).
The condition on the Whittaker datum, when $\pi$ is a simple supercuspidal representation and $r=1$, is removed in Section \ref{subsec:AK-twisted-gamma}.

\vspace{3mm}
\noindent
\ref{sec:lift}.2.\ \quad
The result has to be well-known to the experts, in fact almost obvious to them; indeed it is behind the scene in \cite[page 482--485]{Art13}.
It is only for completeness of our own results, because we have not found published our exact statement, that we write the proof below.
Note that the result has been used in the result \cite{Hen23} by the second named author.
Both $\Pi$ and $\Pi'$ are obtained via a local-global method, using trace formulas for $\Pi$ and converse theorems for $\Pi'$.
Thus the proof starts with a global part, and the local result is a consequence of the strong multiplicity one theorem for $\GL_N$.
See below \ref{sec:lift}.3 to \ref{sec:lift}.5 for the global results, and \ref{sec:lift}.6 for the local consequence.
Our reference for Arthur's results is of course Arthur's book \cite{Art13}, but the reader might benefit from the more expository papers on Arthur's webpage.
We note however that as far as we know the full twisted weighted fundamental lemma announced in \cite{CL10} has not been proved in print, and similarly the references [A24] to [A26] in \cite{Art13} have not appeared yet (reference [A27] refers to non-quasi-split groups, which do not concern us here)
\footnote{As mentioned in Introduction, those points have been addressed in a recent preprint of Atobe--Gan--Ichino--Kaletha--Minguez--Shin \cite{AGIKMS24}.}.
Our reference for the lifting via converse theorems is \cite{CKPSS04}.
The $\gamma$-factors there are obtained via the Langlands--Shahidi method, whereas we use the Rankin--Selberg version.
For $\GL_N\times\GL_r$, Shahidi proved that the two versions coincide (\cite{Sha85,Sha84}); for $\bfG\times\GL_r$ that was done by Kaplan (\cite[Theorem 1 and Corollary 1]{Kap15}).

\vspace{3mm}
\noindent
\ref{sec:lift}.3.\ \quad
Let $k$ be a number field and $\bbA_{k}$ its adele ring.
Let $\pi$ be a globally generic cuspidal automorphic representation of $\bfG(\bbA_{k})$.
There are two ways to associate to $\pi$ an automorphic representation of $\GL_{N}(\bbA_k)$.
The first one \cite[Theorem 1.1]{CKPSS04} uses converse theorems and produces ``a functorial lift'' of $\pi$.
A functorial lift is an automorphic representation $\Pi'$ of $\GL_{N}(\bbA_k)$ such that for all Archimedean places $v$ of $k$ and almost all finite places $v$ where $\pi_v$ is unramified, the local component $\Pi'_{v}$ is obtained via the local Langlands correspondences:
\begin{itemize}
\item
for Archimedean $v$, $\pi_{v}$ corresponds to a morphism of the local Weil group $W_{k_{v}}$ to $\hat{\bfG}$ and $\Pi'_{v}$ to the morphism into $\GL_{N}(\C)$ obtained by composing with the standard representation of $\hat{\bfG}$.
\item
similarly for a finite place $v$ where $\pi_v$ is unramified, $\pi_{v}$ corresponds to an unramified morphism of the local Weil group $W_{k_{v}}$ to $\hat{\bfG}$ and $\Pi'_{v}$ is the unramified representation corresponding to the morphism into $\GL_{N}(\C)$ obtained by composing with the standard representation of $\hat{\bfG}$.
\end{itemize}

\cite{CKPSS04} describe the image of the global lift in their Theorems 7.1 and 7.2, in particular showing that it is a full induced from a self-dual unitary cuspidal automorphic representation of a Levi subgroup of $\GL_N$. Consequently $\Pi'$ is isobaric and all components $\Pi'_{v}$ are generic.
Moreover (\textit{loc.\ cit.}\ Proposition 7.2), for any finite place $v$ of $k$, $\Pi'_{v}$ is the unique irreducible smooth generic representation of $\GL_{N}(k_v)$ such that, for any positive integer $r$ and any smooth irreducible supercuspidal representation $\tau$ of $\GL_{r}(k_v)$, one has, for any nontrivial additive character $\psi_{v}$ of $k_{v}$, $\gamma(s,\pi_{v}\times\tau,\psi_{v})=\gamma(s,\Pi'_{v}\times\tau,\psi_{v})$.
In fact $\Pi'_{v}$ is a ``local functorial lift'' of $\pi_{v}$ (\textit{loc.\ cit.}\ Definition 7.1): we also have $L(s,\pi_{v}\times \tau)=L(s, \Pi'_{v}\times\tau)$, where the $L$-factors are obtained by the Langlands--Shahidi method (\cite{Sha90}); for the right-hand side they can equally be defined via the Rankin--Selberg method (compare \textit{loc.\ cit.}\  Section 10 and \cite[Introduction]{JPSS83}).

\vspace{3mm}
\noindent
\ref{sec:lift}.4.\ \quad
Let us now turn to the lift $\Pi$ of $\pi$ obtained by Arthur.
Note first that $\bfG$ belongs to the set $\widetilde{\mcE}_{\mathrm{sim}}(N)$ (\cite[Chapter 1, page 12]{Art13}), so that Theorems 1.5.1 and 1.5.2 of \cite{Art13} apply to $\bfG$.
Theorem 1.5.2 implies that $\pi$, or more generally any automorphic representation of $\bfG(\bbA_k)$ occurring in the discrete spectrum, is obtained in the following manner: there is a parameter $\psi$ in the global set $\widetilde{\Psi}_{2}(G)$ (\cite[page 33]{Art13}) giving rise to a local parameter $\psi_{v}$ for any place $v$ of $k$, such that $\pi_v$ belongs to the local packet $\Pi_{\psi_{v}}$ associated to $\psi_{v}$ by Theorem 1.5.1.
Now the parameter $\psi$ is in the set $\widetilde{\Psi}_{\mathrm{ell}}(G)$ (\textit{loc.\ cit.}\ page 33) and in particular in the set $\Psi(N)$ (\textit{loc.\ cit.}\ page 28), so is a multiset of pairs $(\pi_{i}, m_{i})$, where $\pi_{i}$ is a cuspidal automorphic representation of $\GL_{N_{i}}(\bbA_k)$ and $m_{i}$ is a positive integer (or the class of irreducible representations of $\SU(2)$ of dimension $m_{i}$), with $N=\sum_{i}m_{i}N_{i}$.
A pair $(\pi_{i}, m_{i})$ gives an essentially discrete automorphic representation of $\GL_{N_{i}m_{i}}(\bbA_k)$, with cuspidal support $\pi_{i}(m_{i})$ made out of $\pi_{i}$'s shifted by powers of the norm, and by parabolic induction from all the components of $\pi_{i}(m_{i})$ (for all $i$) we obtain an isobaric automorphic representation $\Pi$ of $\GL_{N}(\bbA_k)$.
All that is explained in (\cite[Sections 1.2 and 1.4]{Art13}).
As stated above for any place $v$ the component $\pi_v$ belongs to the local packet attached to $\psi_v$.
What is not stated explicitly in \cite[Theorem 1.5.1]{Art13} but appears behind (\textit{loc.\ cit.}\ Foreword, page x) is that at almost all finite places $v$, where $\pi_v$ and $\Pi_{v}$ are unramified, the local unramified parameter of $\Pi_{v}$ is indeed obtained by composing the local parameter of $\pi_v$ by the standard representation of $\hat{\bfG}$ into $\GL_{N}(\C)$.
We have not been able to locate a precise statement, thus we give a justification in Section \ref{sec:unram}.

\vspace{3mm}
\noindent
\ref{sec:lift}.5.\ \quad
Since $\Pi$ and $\Pi'$ are both isobaric automorphic representations of $\GL_{N}(\bbA_k)$, proving they are equal is equivalent to proving that their components at almost all finite places are equal, by the strong multiplicity one theorem of Jacquet and Shalika (cf.\ \cite[Theorem 1.3.2]{Art13}).

But at almost all finite places where both $\Pi_{v}$ and $\Pi'_{v}$ are unramified, $\Pi'_{v}$, by construction, is given by the unramified local Langlands correspondence, and it is also the case for $\Pi_{v}$, as we have seen in \ref{sec:lift}.4.
Thus $\Pi=\Pi'$, that is $\Pi_{v}=\Pi'_{v}$ for all places $v$ of $k$.

\vspace{3mm}
\noindent
\ref{sec:lift}.6.\ \quad
If $F$ is a $p$-adic field as in \ref{sec:lift}.1, one can see it as a completion $k_v$ of some number field $k$, and a smooth irreducible generic supercuspidal representation $\rho$ of $\bfG(F)$ can be seen as the component $\pi_v$ at $v$ of a globally generic cuspidal automorphic representation $\pi$ of $\bfG(\bbA_k)$ (\cite[Proposition 5.1]{Sha90}).
Thus $\rho$ has a local functorial lift $R$ to $\GL_{N}( F)$, viz.\ (the class of) $\Pi'_{v}$, where $\Pi'$ is the global lift of $\pi$ to $\GL_N$ obtained by the converse theorems.
But we have seen in \ref{sec:lift}.5 that $\Pi'$ is also the global lift $\Pi$ given by Arthur.
By the compatibility of Arthur's local and global lifts, indeed by the fact that the global lift in (\cite[Theorem 1.5.2]{Art13}) is expressed in terms of the local one (\textit{loc.\ cit.}\ Theorem 1.5.1), we get the desired result that $R$ is also the local lift to $\GL_{N}( F)$ given by Arthur.

\vspace{3mm}
\noindent
\ref{sec:lift}.7.\ \quad
For ease of reference, let us restate our results in this appendix.

\begin{thm}\label{thm:lift-1}
Let $k$ be a number field.
Let $\bfG$ be a symplectic group $\Sp_{2n}$ or a split special orthogonal group $\SO_{n}$ over $k$; write $N$ for the dimension of the natural representation of $\hat{\bfG}$.
Let $\pi$ be a globally generic cuspidal automorphic representation of $\bfG$ over $k$.
Write $\Pi$ for the automorphic representation of $\GL_N$ over $k$ associated to the Arthur parameter of $\pi$, and $\Pi'$ for the automorphic representation of $\GL_N$ over $k$ associated to $\pi$ by the lifting of Cogdell et al.
Then $\Pi=\Pi'$, and for each place $v$ of $k$, $\Pi_{v}$ is the local lifting $\Pi'_{v}$ associated to $\pi_{v}$ by Cogdell et al.
\end{thm}

\begin{thm}\label{thm:lift-2}
Let $F$ be a $p$-adic field.
Let $\bfG$ be a symplectic group $\Sp_{2n}$ or a split special orthogonal group $\SO_{n}$ over $F$; write $N$ for the dimension of the natural representation of $\hat{\bfG}$.
Let $\pi$ be a generic supercuspidal representation of $\bfG(F)$.
Write $\Pi$ for the irreducible smooth representation of $\GL_{N}(F)$ associated to the Arthur parameter of $\pi$, and $\Pi'$ for the smooth irreducible representation of $\GL_{N}(F)$ associated to $\pi$ by the (local) lifting of Cogdell et al.
Then $\Pi=\Pi'$.
For any positive integer $r$ and any generic irreducible smooth representation $\tau$ of $\GL_{r}( F)$ the Rankin--Selberg (or Langlands--Shahidi) $\gamma$-factor $\gamma(s,\pi\times\tau,\psi)$ is equal to the Rankin--Selberg $\gamma$-factor $\gamma(s,\Pi\times\tau,\psi)$, for any choice of a nontrivial additive character $\psi$ of $F$.
The same is true for the $L$ and $\varepsilon$-factors.
\end{thm}

\begin{rem}
We have restrained here to the framework that is useful to us in the main part of the paper, but the approach obviously works much more generally.
\end{rem}

\section{Unramified case of Arthur's classification theorem}\label{sec:unram}

The aim of this section is to justify the compatibility of Arthur's local classification theorem (construction of local $A$-packets) in the unramified case with the classical Satake parametrization.
The idea of the arguments we present here is due to Jean-Loup Waldspurger.

\subsection{Classical groups as twisted endoscopy of \texorpdfstring{$\GL_{N}$}{GL(N)}}

Let $\bfG'$ be an unramified quasi-split classical (i.e., symplectic, orthogonal, or unitary) group over a $p$-adic field $F$.
Then we may regard $\bfG'$ as a twisted endoscopic group of a suitable general linear group $\bfG=\GL_{N}$ (or the Weil restriction $\bfG=\Res_{E/F}\GL_{N}$ for an unramified quadratic extension $E/F$) with respect to an outer automorphism $\theta$ of $\bfG$.
In particular, we have a natural $L$-embedding $\iota\colon \L\bfG'\hookrightarrow\L\bfG$.
(See \cite[Section 1.2]{Art13} and \cite[Section 2.1]{Mok15} for the standard realization of $\theta$, $\iota$, and so on.)
We put $\tilde{G}:=G\rtimes\theta$, which is a bi-$G$-torsor whose right and left actions of $G=\bfG(F)$ are given by $g_{1}\cdot(g\rtimes\theta)\cdot g_{2}=g_{1}g\theta(g_{2})\rtimes\theta$.

We fix a $\theta$-stable $F$-splitting of $\bfG$.
Note that a $\theta$-stable Whittaker datum $\mathfrak{w}$ of $\bfG$ is determined by this choice.
Similarly, we also fix an $F$-splitting of $\bfG'$.

We let $\mcH$ denote the Hecke algebra of $G$, i.e., the set of compactly supported locally constant $\C$-valued functions on $G$ equipped with the convolution product denoted by ``$\ast$''.
We let $\tilde{\mcH}$ denote the set of compactly supported locally constant $\C$-valued functions on $\tilde{G}$.
Similarly, we let $\mcH'$ denote the Hecke algebra of $G'$.
Then we can define the notion of \textit{a (Langlands--Kottwitz--Shelstad) transfer} between $\tilde{\mcH}$ and $\mcH'$; we say that $f'\in\mcH'$ is a transfer of $\tilde{f}\in\tilde{\mcH}$ if they satisfy a certain identity between the twisted orbital integrals of $\tilde{f}$ and the stable orbital integrals of $f'$.
See \cite[Section 2.1]{Art13} for the details.

\subsection{Fundamental lemma of Lemaire--M{\oe}glin--Waldspurger}
We next review a deep result of Lemaire--M{\oe}glin--Waldspurger (\cite{LW17,LMW18}) on the transfer for spherical Hecke algebras.

The fixed $\theta$-stable $F$-splitting of $\bfG$ gives rise to a $\theta$-stable hyperspecial open compact subgroup of $G$ (see \cite[Section 2.5]{LMW18}); we write $K$ for it.
We let $\mcH_{K}$ (resp.\ $\tilde{\mcH}_{K}$) be the subalgebra of $\mcH$ (resp.\ subspace of $\tilde{\mcH}$) consisting of bi-$K$-invariant functions.
Note that then $\tilde{\mcH}_{K}$ has right and left actions of $\mcH_{K}$ and $\tilde{\mcH}_{K}=\mcH_{K}\ast\mathbbm{1}_{\tilde{K}}$, where $\mathbbm{1}_{\tilde{K}}$ denotes the characteristic function of $\tilde{K}:=K\rtimes\theta$.
Similarly, we write $K'$ for the hyperspecial open compact subgroup of $G'$ determined by the fixed $F$-splitting of $\bfG'$ and let $\mcH'_{K'}$ be the subalgebra of $\mcH'$ consisting of bi-$K'$-invariant functions.
(When $\bfG'=\SO_{2n}$, we suppose that $K'$ is invariant under the conjugation given by an element of $\mathrm{O}_{2n}(F)\smallsetminus\SO_{2n}(F)$.)

Let $\hat{\mcH}$ denote the algebra of polynomial functions on $\hat{\bfG}\rtimes\Frob\subset\L\bfG$ invariant under the $\hat{\bfG}$-conjugation, where $\Frob$ is a fixed lift of the Frobenius.
Then $\mcH_{K}$ can be identified with $\hat{\mcH}$ via the Satake isomorphism $S$ for $\bfG$.
Similarly, $\mcH'_{K'}$ can be identified with the algebra $\hat{\mcH}'$ of polynomial functions on $\hat{\bfG}'\rtimes\Frob\subset\L\bfG'$ invariant under the $\hat{\bfG}'$-conjugation via the Satake isomorphism $S'$ for $\bfG'$.
We let $\hat{b}\colon\hat{\mcH}\rightarrow\hat{\mcH}'$ be the $\C$-algebra homomorphism given by the restriction along the $L$-embedding $\iota\colon\L\bfG'\hookrightarrow\L\bfG$.
We define $b\colon\mcH_{K}\rightarrow\mcH'_{K'}$ to be the unique $\C$-algebra homomorphism which makes the following diagram commutative:
\[
\xymatrix{
\mcH_{K} \ar_-{b}[d] \ar^-{S}_-{\cong}[r] & \hat{\mcH} \ar^-{\hat{b}}[d]\\
\mcH'_{K'} \ar_-{S'}^-{\cong}[r]& \hat{\mcH}'
}
\]

\begin{thm}[{\cite[Th\'eor\`eme 1, 2]{LMW18}}]\label{thm:FL}
For any $\tilde{f}\in\tilde{\mcH}_{K}$, if we write $\tilde{f}=f\ast\mathbbm{1}_{\tilde{K}}$ with $f\in\mcH_{K}$, then $b(f)\in\mcH'_{K'}$ is a transfer of $\tilde{f}$.
In particular, $\mathbbm{1}_{K'}\in\mcH'_{K'}$ is a transfer of $\mathbbm{1}_{\tilde{K}}\in\tilde{\mcH}_{K}$.
\end{thm}

\begin{rem}\label{rem:transfer-factor}
When we define the notion of a transfer of test functions from $\tilde{\mathcal{H}}$ to $\mathcal{H}'$, we need to fix a normalization of the \textit{transfer factor}.
In the above theorem, we adopt a normalization determined by the fixed choice of a $\theta$-stable hyperspecial open compact subgroup $K$ of $G$ (see \cite[Section 2.6]{LMW18} and \cite[I.6.3]{MW16-1}).
\end{rem}

\subsection{Arthur's local classification theorem}

We put $L_{F}:=W_{F}\times\SL_{2}(\C)$.
We say a homomorphism $\psi\colon L_{F}\times\SL_{2}(\C)\rightarrow \L{\bfG'}$ is an \textit{$A$-parameter of $\bfG'$} if
\begin{itemize}
\item
its restriction $\psi|_{L_{F}}$ to $L_{F}$ is a tempered $L$-parameter and
\item
its restriction $\psi|_{\SL_{2}(\C)}$ to $\SL_{2}(\C)$ is algebraic.
\end{itemize}
We let $\Psi(\bfG')$ be the set of $\hat{\bfG}'$-conjugacy classes of $A$-parameters of $\bfG'$.
We define $\tilde{\Psi}(\bfG')$ to be the set of $\mathrm{O}_{2n}(\C)$-conjugacy classes of $A$-parameters of $\bfG'$ when $\bfG'=\SO_{2n}$.
When $\bfG'$ is not $\SO_{2n}$, we simply put $\tilde{\Psi}(\bfG'):=\Psi(\bfG')$.
We let $\Pi_{\mathrm{unit}}(\bfG')$ be the set of irreducible unitary representations of $G'$.
We define $\tilde{\Pi}_{\mathrm{unit}}(\bfG')$ to be the set of $\mathrm{O}_{2n}(F)$-conjugacy classes of irreducible unitary representations of $G'$ when $\bfG'=\SO_{2n}$.
When $\bfG'$ is not $\SO_{2n}$, we simply put $\tilde{\Pi}_{\mathrm{unit}}(\bfG'):=\Pi_{\mathrm{unit}}(\bfG')$.

For an $A$-parameter $\psi\in\tilde{\Psi}(\bfG')$, we define a finite group $\overline{S}_{\psi}$ as follows:
\begin{align*}
S_{\psi}&:=\Cent_{\hat{\bfG}'}(\mathrm{Im}(\psi)),\\
\overline{S}_{\psi}&=S_{\psi}/(S_{\psi}^{\circ}Z_{\hat{\bfG}'}^{W_{F}}).
\end{align*}
Here, we implicitly fix a representative of the equivalence class $\psi$ and again write $\psi$ for it by abuse of notation.
We define an element $s_{\psi}$ of $S_{\psi}$ by
\[
s_{\psi}
:=
\psi\biggl(1,\begin{pmatrix}-1&0\\ 0&-1\end{pmatrix}\biggr).
\]

Any $A$-parameter $\psi\in\tilde{\Psi}(\bfG')$ can be regarded as an $A$-parameter of $\bfG$ by composing $\psi$ with the $L$-embedding $\iota\colon\L\bfG'\hookrightarrow\L\bfG$.
Let $\pi_{\psi}$ denote the irreducible unitary representation of $G$ determined by $\psi$, i.e., $\pi_{\psi}$ corresponds to the $L$-parameter $\phi_{\psi}$ of $\bfG$ under the local Langlands correspondence for $\bfG$, where $\phi_{\psi}\colon L_{F}\rightarrow \L\bfG$ is defined by
\[
\phi_{\psi}(u)
:=
\psi\biggl(u,\begin{pmatrix}|u|^{\frac{1}{2}}&0\\ 0&|u|^{-\frac{1}{2}}\end{pmatrix}\biggr).
\]
Note that, since the representation $\pi_{\psi}$ is self-dual, we can take a canonical extension $\tilde{\pi}_{\psi}$ of $\pi_{\psi}$ to the bi-torsor $\tilde{G}$ by using the fixed $\theta$-stable Whittaker datum $\mfw$ of $\bfG$.
(See \cite[Section 2.2]{Art13} for the details of the discussion here.)

Now we state a part of Arthur's local classification theorem (see \cite[Theorems 1.5.1 and 2.2.1]{Art13} for symplectic and orthogonal groups and \cite[Theorems 2.5.1 and 3.2.1]{Mok15} for unitary groups):

\begin{thm}\label{thm:Arthur}
For any $\psi\in\tilde{\Psi}(\bfG')$, there is a finite multi-set $\underline{\tilde{\Pi}}_{\psi}$ (called an ``$A$-packet'') over $\tilde{\Pi}_{\mathrm{unit}}(\bfG')$ equipped with a map
\[
\iota_{\mfw}\colon \underline{\tilde{\Pi}}_{\psi}\rightarrow \overline{S}_{\psi}^{\vee};\quad
\underline{\pi}\mapsto \langle-,\underline{\pi}\rangle,
\]
where $\overline{S}_{\psi}^{\vee}$ denotes the set of irreducible characters of $\overline{S}_{\psi}$.
The set $\underline{\tilde{\Pi}}_{\psi}$ satisfies the following identity (called the ``twisted endoscopic character relation'') for any $\tilde{f}\in\tilde{\mcH}$ and its transfer $f'\in\mcH'$:
\begin{align}\label{eq:TECR}
\sum_{\underline{\pi}\in\underline{\tilde{\Pi}}_{\psi}}\langle s_{\psi},\underline{\pi}\rangle\Tr(\underline{\pi}(f'))
=
c\cdot\Tr(\tilde{\pi}_{\psi}(\tilde{f})),
\end{align}
where $c$ is a complex number of absolute value $1$ which depends only on the fixed $\theta$-stable $F$-splitting of $\bfG$.
Furthermore, if $\underline{\pi}\in\underline{\tilde{\Pi}}_{\psi}$ is unramified (i.e., $K'$-spherical), then $\langle-,\underline{\pi}\rangle\in\overline{S}_{\psi}^{\vee}$ is the trivial character $\mathbbm{1}$ of $\overline{S}_{\psi}$.
\end{thm}

Here, the precise meaning of ``a finite multi-set $\underline{\tilde{\Pi}}_{\psi}$ over $\tilde{\Pi}_{\mathrm{unit}}(\bfG')$'' is that $\underline{\tilde{\Pi}}_{\psi}$ is a finite set equipped with a surjective map $\mu_{\psi}$ to a finite subset $\tilde{\Pi}_{\psi}\subset \tilde{\Pi}_{\mathrm{unit}}(\bfG')$:
\[
\mu_{\psi}\colon
\underline{\tilde{\Pi}}_{\psi} \twoheadrightarrow \tilde{\Pi}_{\psi}
;\quad
\underline{\pi}\mapsto\pi.
\]
When $\underline{\pi}\in\underline{\tilde{\Pi}}_{\psi}$ is mapped to $\pi\in\tilde{\Pi}_{\psi}$, we put $\Tr(\underline{\pi}(f')):=\Tr(\pi(f'))$ and say that $\underline{\pi}$ is unramified if so is $\pi$.
(Note that the quantity $\Tr(\pi(f'))$ is well-defined even when $\bfG'=\SO_{2n}$ since a transfer $f'$ can be taken to be $\mathrm{O}_{2n}$-invariant.)

We can reformulate the above statement by introducing a multiplicity function as follows.
For each $\chi\in\overline{S}_{\psi}^{\vee}$, we put
\[
\underline{\tilde{\Pi}}_{\psi,\chi}
:=
\{\underline{\pi}\in\underline{\tilde{\Pi}}_{\psi} \mid \langle-,\underline{\pi}\rangle=\chi \}
\]
and define $\tilde{\Pi}_{\psi,\chi}:=\mu_{\psi}(\underline{\tilde{\Pi}}_{\psi,\chi})$.
We let $\mu_{\psi,\chi}\colon
\underline{\tilde{\Pi}}_{\psi,\chi}\twoheadrightarrow \tilde{\Pi}_{\psi,\chi}$ be the restriction of the map $\mu_{\psi}\colon
\underline{\tilde{\Pi}}_{\psi}\twoheadrightarrow \tilde{\Pi}_{\psi}$ to $\underline{\tilde{\Pi}}_{\psi,\chi}$.
We define the multiplicity function $m_{\psi,\chi}\colon \tilde{\Pi}_{\psi,\chi}\rightarrow\Z_{>0}$ by
\[
m_{\psi,\chi}(\pi)
:=
|\mu_{\psi,\chi}^{-1}(\pi)|.
\]
Then the identity \eqref{eq:TECR} is rewritten as
\begin{align}\label{eq:TECR2}
\sum_{\chi\in\overline{S}_{\psi}^{\vee}}
\sum_{\pi\in\tilde{\Pi}_{\psi,\chi}}
m_{\psi,\chi}(\pi)\chi(s_{\psi})\Tr(\pi(f'))
=
c\cdot\Tr(\tilde{\pi}_{\psi}(\tilde{f})).
\end{align}

\begin{rem}\label{rem:Moeglin}
A priori, it is possible that the multiplicity $m_{\psi,\chi}(\pi)$ is greater than $1$ or that $\tilde{\Pi}_{\psi,\chi}$ and $\tilde{\Pi}_{\psi,\chi'}$ for distinct $\chi,\chi'\in\overline{S}_{\psi}^{\vee}$ have a nonempty intersection.
However, in fact, M{\oe}glin proved that $\underline{\tilde{\Pi}}_{\psi}$ is multiplicity-free, i.e., $\mu_{\psi}$ is bijective (\cite{Moeg11}, combined with the result of Bin Xu \cite{Xu17-CJM} on comparing M{\oe}glin's $A$-packets to Arthur's; see \cite[Theorem 8.12]{Xu17-CJM}).
Thus we may regard $\underline{\tilde{\Pi}}_{\psi}$ as a subset $\tilde{\Pi}_{\psi}$ of $\tilde{\Pi}_{\mathrm{unit}}(\bfG')$.
(But we do not have to appeal to this fact in the following argument.)
\end{rem}

\begin{rem}
As noted in Remark \ref{rem:transfer-factor}, we adopt a normalization of the transfer factor determined by the $\theta$-stable hyperspecial open compact subgroup $K$ according to \cite[I.6.3]{MW16-1}.
On the other hand, in \cite{Art13}, the transfer factor is normalized by using the $\theta$-stable Whittaker datum $\mathfrak{w}$ of $\bfG$ (see \cite[Section 2.1]{Art13} and \cite[Section 5.3]{KS99}).
The point is that a priori it is not clear whether these two normalizations coincide; this is the source of the constant $c$ in the identity \eqref{eq:TECR}.
We remark that, via both normalizations, the transfer factor takes values in unitary complex numbers, hence also the constant $c$ is unitary.
(For the unitarity of the transfer factor normalized via $K$, see \cite[I.7.2]{MW16-1}.
For the unitarity of the transfer factor normalized via $\mathfrak{w}$, see, for example, the explicit formula of Waldspurger \cite[I.10]{Wal10}.)
We believe that it should be possible to show that $c=1$ by examining the definitions of the two normalizations since both $K$ and $\mathfrak{w}$ are produced from the same $\theta$-stable $F$-splitting of $\G$.
However, we do not pursue this issue further because we only need the fact that $|c|=1$ in the following argument.
\end{rem}

\subsection{Unramified representation in an $A$-packet}
Recall that, by the Satake isomorphism, any unramified representation $\pi$ of $G$ corresponds to a $\hat{\bfG}$-conjugacy class $t_{\pi}$ of semisimple elements in $\hat{\bfG}\rtimes\Frob$.
Similarly, any unramified representation $\pi'$ of $G$ corresponds to a $\hat{\bfG}$-conjugacy class $t_{\pi'}$ of semisimple elements in $\hat{\bfG}'\rtimes\Frob$.
The image $\iota(t_{\pi'})$ of $t_{\pi'}$ under the $L$-embedding $\iota\colon\L\bfG'\hookrightarrow\L\bfG$ is contained in a unique $\hat{\bfG}$-conjugacy class of semisimple elements in $\hat{\bfG}\rtimes\Frob$, for which we write $\biota(t_{\pi'})$.

\begin{rem}
Suppose $\bfG'=\SO_{2n}$ and the $\mathrm{O}_{2n}(F)$-orbit of $\pi\in\Pi_{\mathrm{unit}}(\bfG')$ consists of two elements $\pi_{1}$ and $\pi_{2}$.
Then, one of $\pi_{1}$ and $\pi_{2}$ is unramified if and only if the other is also unramified by our choice of a hyperspecial open compact subgroup $K'$.
In this case, although $\pi_{1}$ and $\pi_{2}$ correspond to distinct $\hat{\bfG}'$-conjugacy classes $t_{\pi_{1}}$ and $t_{\pi_{2}}$ in $\hat{\bfG}'\rtimes\Frob$, $\biota(t_{\pi_{1}})=\biota(t_{\pi_{2}})$.
In other words, the symbol $\biota(t_{\pi})$ is well-defined for any unramified $\pi\in\tilde{\Pi}_{\mathrm{unit}}(\bfG')$.
\end{rem}

\begin{prop}\label{prop:unramified-A-packets}
Let $\psi\in\tilde{\Psi}(\bfG')$.
Then $\pi_{\psi}\in\Pi_{\mathrm{unit}}(\bfG)$ is unramified if and only if $\tilde{\Pi}_{\psi,\mathbbm{1}}$ contains an unramified representation.
Furthermore, in this case, such an unramified representation is unique. Denote it by $\pi_0$. Then $m_{\psi,\mathbbm{1}}(\pi_{0})=1$ and $t_{\pi_{\psi}}=\biota(t_{\pi_{0}})$.
\end{prop}

\begin{proof}
We apply the twisted endoscopic character relation \eqref{eq:TECR2}
\[
\sum_{\chi\in\overline{S}_{\psi}^{\vee}}
\sum_{\pi\in\tilde{\Pi}_{\psi,\chi}}
m_{\psi,\chi}(\pi)\chi(s_{\psi})\Tr(\pi(f'))
=
c\cdot\Tr(\tilde{\pi}_{\psi}(\tilde{f}))
\]
to a function $\tilde{f}\in\tilde{\mcH}_{K}$ given by $\tilde{f}=f\ast\mathbbm{1}_{\tilde{K}}$ with $f\in \mcH_{K}$ and its transfer $f'\in\mcH'$.
Let $V_{\pi_{\psi}}^{K}$ denote the subspace of $K$-fixed vectors in the representation space $V_{\pi_{\psi}}$ of $\pi_{\psi}$.
If $\pi_{\psi}$ is not unramified (i.e., $V_{\pi_{\psi}}^{K}=0$), then $\Tr(\tilde{\pi}_{\psi}(\tilde{f}))=0$ by the definition of the operator $\tilde{\pi}_{\psi}(\tilde{f})$.
If $\pi_{\psi}$ is unramified, the operator $\tilde{\pi}_{\psi}(\mathbbm{1}_{\tilde{K}})$ necessarily preserves the space $V_{\pi_{\psi}}^{K}$.
Furthermore, since the space $V_{\pi_{\psi}}^{K}$ is one-dimensional and the action of $\tilde{\pi}_{\psi}(\mathbbm{1}_{\tilde{K}})$ on $V_{\pi_{\psi}}^{K}$ is involutive, $\tilde{\pi}_{\psi}(\mathbbm{1}_{\tilde{K}})$ acts on $V_{\pi_{\psi}}^{K}$ via a sign $\epsilon_{\psi}\in\{\pm1\}$.
Thus $\Tr(\tilde{\pi}_{\psi}(\tilde{f}))=\epsilon_{\psi}\cdot\Tr(\pi_{\psi}(f))$.
On the other hand, recall that we can choose a transfer $f'$ of $\tilde{f}$ to be an element of $\mcH'_{K'}$ by Theorem \ref{thm:FL}.
In particular, $\Tr(\pi(f'))$ (for $\chi\in\overline{S}_{\psi}^{\vee}$ and $\pi\in\tilde{\Pi}_{\psi,\chi}$) can be nonzero only when $\pi$ is unramified, which furthermore implies that $\chi=\mathbbm{1}$ by Theorem \ref{thm:Arthur}.
Thus we get
\begin{align}\label{eq:TECR-ur}
\sum_{\begin{subarray}{c}\pi\in\tilde{\Pi}_{\psi,\mathbbm{1}}\\ V_{\pi}^{K'}\neq0 \end{subarray}} m_{\psi,\mathbbm{1}}(\pi)\Tr(\pi(f'))
=
\begin{cases}
c\cdot\epsilon_{\psi}\cdot\Tr(\pi_{\psi}(f))& \text{if $\pi_{\psi}$ is unramified,}\\
0& \text{otherwise}.
\end{cases}
\end{align}

Let us take $f$ in the identity \eqref{eq:TECR-ur} to be $\mathbbm{1}_{K}$.
Then we can choose $f'$ to be $\mathbbm{1}_{K'}$ by Theorem \ref{thm:FL}.
Since $\Tr(\pi(\mathbbm{1}_{K'}))=1$ whenever $V_{\pi}^{K'}\neq0$ (resp.\ $\Tr(\pi_{\psi}(\mathbbm{1}_{K}))=1$ whenever $V_{\pi_{\psi}}^{K}\neq0$), we get
\[
\sum_{\begin{subarray}{c}\pi\in\tilde{\Pi}_{\psi,\mathbbm{1}}\\ V_{\pi}^{K}\neq0 \end{subarray}} m_{\psi,\mathbbm{1}}(\pi)
=
\begin{cases}
c\cdot\epsilon_{\psi}& \text{if $\pi_{\psi}$ is unramified,}\\
0& \text{otherwise}.
\end{cases}
\]
Since $m_{\psi,\mathbbm{1}}(\pi)$ is positive and $|c|=1$, this implies that
\begin{itemize}
\item
if $\pi_{\psi}$ is not unramified, then the index set of the sum on the left-hand side is empty, i.e., $\tilde{\Pi}_{\psi}$ does not contain any unramified representation, and
\item
if $\pi_{\psi}$ is unramified, then $c\cdot\epsilon_{\psi}=1$ and the index set of the sum on the left-hand side consists of a unique element $\pi_{0}$ and $m_{\psi,\mathbbm{1}}(\pi_{0})=1$.
\end{itemize}

Let us finally check that $t_{\pi_{\psi}}=\biota(t_{\pi_{0}})$ by supposing that $\pi_{\psi}$ is unramified.
Now we know that the identity \eqref{eq:TECR-ur} simplifies to
\[
\Tr(\pi_{0}(f'))
=
\Tr(\pi_{\psi}(f))
\]
for any $\tilde{f}=f\ast\mathbbm{1}_{\tilde{K}}\in\tilde{\mcH}_{K}$ and its transfer $f'\in\mcH'_{K'}$.
If we take $f'\in\mcH'_{K'}$ as in Theorem \ref{thm:FL} (i.e., $f'=b(f)$), then this equality is rewritten as
\[
S'(b(f))(t_{\pi_{0}})
=
S(f)(t_{\pi_{\psi}})
\]
(recall that $S$ and $S'$ denote the Satake isomorphisms for $\bfG$ and $\bfG'$, respectively).
By the definition of the homomorphism $b\colon \mcH_{K}\rightarrow\mcH'_{K'}$, $S'(b(f))(t_{\pi_{0}})=S(f)(\biota(t_{\pi_{0}}))$.
Hence we conclude that the identity
\[
S(f)(\biota(t_{\pi_{0}}))
=
S(f)(t_{\pi_{\psi}})
\]
holds for any $f\in\mcH_{K}$.
This implies that $\biota(t_{\pi_{0}})=t_{\pi_{\psi}}$.
\end{proof}

\end{document}